\documentclass[a4paper,english]{amsart}

\usepackage[ruled,noline,noend]{algorithm2e}
\usepackage{amsaddr}

\usepackage{amssymb}
\usepackage{bm}
\usepackage[normalem]{ulem}
\usepackage{commath}
\usepackage{textcomp}
\usepackage{graphicx}
\usepackage{hyperref}
\usepackage{mathtools}
\usepackage{tikz}

\SetKwBlock{Loop}{repeat}{}
\SetKwBlock{Solve}{\normalfont{\underline{Solve:}}}{}

\usepackage{thmtools, thm-restate}
\numberwithin{equation}{section}
\declaretheorem{theorem}
\numberwithin{theorem}{section}

\newtheorem{lemma}[theorem]{Lemma}
\newtheorem{corol}[theorem]{Corollary}
\newtheorem{prop}[theorem]{Proposition}
\theoremstyle{definition}
\newtheorem{definition}[theorem]{Definition}
\newtheorem{remark}[theorem]{Remark}
\newtheorem{example}[theorem]{Example}

\usepackage{color}
\newcommand{\R}{\mathbb{R}}
\newcommand{\N}{\mathbb{N}}

\newcommand\D[1]{\mathrm{d}#1\,}

\renewcommand{\phi}{\varphi}

\DeclareMathOperator{\gp}{gp}
\DeclareMathOperator{\supp}{supp}

\DeclareMathOperator{\gen}{gen}

\DeclareMathOperator{\blockdiag}{blockdiag}

\DeclareMathOperator{\spann}{span}
\DeclareMathOperator{\dist}{dist}

\usepackage{stmaryrd}

\renewcommand{\vec}{\mathbf}

\DeclareMathOperator*{\argmin}{arg\,min}
\DeclareMathOperator*{\Id}{Id}
\definecolor{airforceblue}{rgb}{0.36, 0.54, 0.66}
\definecolor{ao(english)}{rgb}{0.0, 0.5, 0.0}

\newcommand{\bbT}{\mathbb{T}}
\newcommand{\cL}{\mathcal L}
\newcommand{\Lis}{\cL\mathrm{is}}
\newcommand{\udelta}{{\underline{\delta}}}

\SetKwRepeat{Do}{do}{while}%

\DeclareFontFamily{U}{matha}{\hyphenchar\font45}
\DeclareFontShape{U}{matha}{m}{n}{
  <5> <6> <7> <8> <9> <10> gen * matha
  <10.95> matha10 <12> <14.4> <17.28> <20.74> <24.88> matha12
  }{}
\DeclareSymbolFont{matha}{U}{matha}{m}{n}
\DeclareFontSubstitution{U}{matha}{m}{n}

\DeclareFontFamily{U}{mathx}{\hyphenchar\font45}
\DeclareFontShape{U}{mathx}{m}{n}{
  <5> <6> <7> <8> <9> <10>
  <10.95> <12> <14.4> <17.28> <20.74> <24.88>
  mathx10
  }{}
\DeclareSymbolFont{mathx}{U}{mathx}{m}{n}
\DeclareFontSubstitution{U}{mathx}{m}{n}

\DeclareMathDelimiter{\vvvert}{0}{matha}{"7E}{mathx}{"17}

\DeclareMathOperator{\diam}{diam}

\usepackage{bbm}
\newcommand{\bbone}{\mathbbm{1}}

\newcommand{\grad}{\nabla}

\renewcommand{\emptyset}{{\varnothing}}

\DeclareFontFamily{U}{mathx}{\hyphenchar\font45}
\DeclareFontShape{U}{mathx}{m}{n}{
  <5> <6> <7> <8> <9> <10>
  <10.95> <12> <14.4> <17.28> <20.74> <24.88>
  mathx10
  }{}
\DeclareSymbolFont{mathx}{U}{mathx}{m}{n}
\DeclareFontSubstitution{U}{mathx}{m}{n}
\DeclareMathAccent{\widecheck}{0}{mathx}{"71}
\DeclareMathAccent{\wideparen}{0}{mathx}{"75}

\DeclareMathOperator{\Span}{span}

\newcommand{\newer}[1]{#1}

\newcommand{\be}{\begin{equation}}
\newcommand{\ee}{\end{equation}}

\newcommand{\parent}{{\tt parent}}
\newcommand{\tria}{{\mathcal T}}

\makeatletter
\protected\def\tikz@nonactivecolon{\ifmmode\mathrel{\mathop\ordinarycolon}\else:\fi}
\makeatother

\usetikzlibrary{math}
\title[Efficient space-time adaptivity for parabolic equations]{Efficient space-time adaptivity for parabolic evolution equations using wavelets in time and finite elements in space}
\author{Raymond van Veneti\"e \and Jan Westerdiep}
\address{Korteweg--de Vries Institute for Mathematics, University of Amsterdam, \\
P.O. Box 94248, 1090 GE Amsterdam, The Netherlands}
\date{\today}
\subjclass[2010]{}

\thanks{\emph{Funding:} Both authors are supported by the Netherlands Organisation for Scientific Research (NWO) under contract.~no.~613.001.652}



\begin{document}
\begin{abstract}
Considering the space-time adaptive method for parabolic evolution equations introduced in [arXiv:2101.03956 [math.NA]], this work discusses an implementation of the method in which every step is of linear complexity.
Exploiting the product structure of the space-time cylinder, the method allows for a family of trial spaces
given as the spans of wavelets-in-time tensorized with (locally refined) finite element spaces-in-space.
On spaces whose bases are indexed by \emph{double-trees}, we derive an algorithm that applies the resulting bilinear forms in linear complexity. We provide extensive numerical experiments to demonstrate the linear runtime of the resulting adaptive loop.
~\\
~\\
\newer{{\noindent}\textsc{Supplementary material.} Source code is available at~\cite{VanVenetie2021}.}
\end{abstract}

\maketitle

\section{Introduction}
{\noindent}This paper deals with the adaptive numerical solution of parabolic evolution equations using a simultaneous space-time variational formulation. Compared to the more classical time-stepping schemes, these space-time methods are very flexible. Among other things, they are especially well-suited for massively parallel computation (\cite{Neumuller2019,VanVenetie2020a}), and some can guarantee quasi-best approximations from the trial space (\cite{Andreev2013,Fuhrer2019,Steinbach2020a}).

We are interested in those space-time methods that permit adaptive refinement locally in space \emph{and} time.
Within this class, wavelet-based methods (see~\cite{Schwab2009,Gunzburger2011,Kestler2015}) are attractive, as they can be shown to be \emph{quasi-optimal}:
they produce a sequence of solutions that converges at the best possible rate, at optimal linear computational cost. Moreover, they can overcome the \emph{curse of dimensionality} using a form of \emph{sparse tensor-product approximation}, solving the whole time evolution at a runtime proportional to that of solving the corresponding \emph{stationary} problem.

In~\cite{Stevenson2020b}, we constructed an $r$-linearly converging space-time adaptive solver for
parabolic evolution equations that exploits the
product structure of the space-time cylinder to construct a family of trial spaces given as
the spans of wavelets-in-time
tensorized with (locally refined) finite element spaces-in-space.

The principal difference between this and other wavelet-based methods  is that we use wavelets in time \emph{only}, and standard finite elements in space. This eases implementation, and alleviates the need for a suitable spatial wavelet basis, which is generally difficult for general domains (\cite{Rekatsinas2018a}).
Unfortunately, there is no free lunch: a proof of optimal convergence is, for our method, not yet available.

In this work we discuss an implementation of~\cite{Stevenson2020b} in which the different steps
(each iteration of the linear algebraic solver, the error estimation,
D\"orfler marking, and refinement of trial- and test spaces)
of the adaptive algorithm are of linear complexity.

Special care has to be taken for matrix-vector products.
For a bilinear form that is `local' and equals
(a sum of) tensor-product(s) of bilinear forms in time and space, and
`trial' and `test' spaces spanned by tensor-product multi-level bases with \emph{double-tree} index sets, the resulting
system matrix w.r.t.~both bases can be applied in linear complexity, even though this matrix is not sparse. The algorithm that realizes this complexity makes a clever use of multi- to single-scale transformations alternately in time and space.
This \emph{unidirectional principle} was introduced in~\cite{Balder1996} for `uniform' sparse grids, so without `local refinements',
and it was later extended to general \emph{downward closed} or \emph{lower sets}, also called \emph{adaptive sparse grids},~in \cite{Kestler2014a}. The definition of a lower set in~\cite{Kestler2014a}, there called multi-tree, is more restrictive than our current definition that allows more localized refinements.

To the best of our knowledge, other implementations for the efficient evaluation of tensor-product bilinear forms
(see~\cite{Pfluger2010,Kestler2014a,Pabel,Rekatsinas2018})
are based on the concept of hash maps. There, a hash function is used
to map basis functions to array indices.
In an adaptive loop, the final set of basis functions is unknown in advance so it is impossible to construct a hash function that guarantees an upper bound on the number of hash collisions.
Aiming at true linear complexity, we
implement these operations by traversing \emph{trees} and \emph{double-trees}, so without the use of hash maps.

\subsection*{Organization}
In \S\ref{sec:problem}, we look at the abstract parabolic problem, its stable discretization,
and the adaptive routine.
In~\S\ref{sec:lincomp},
we provide an abstract algorithm for the efficient evaluation of tensor-product bilinear forms w.r.t.~ multilevel bases indexed on \emph{double-trees}.
In \S\ref{sec:heat}, we take the \emph{heat equation} as a model problem, and provide a concrete family of trial- and test
spaces with bases indexed by double-trees that permits local space-time adaptivity. In \S\ref{sec:impl}, we discuss the practical
implementation of the adaptive algorithm.
Finally, in \S\ref{sec:numerical}, we provide extensive numerical experiments to demonstrate the linear
runtime of the algorithm.

\subsection*{Notation}
In this work, by $C \lesssim D$ we will mean that $C$ can be bounded by a multiple of $D$, independently of parameters which C and D may depend on.
Obviously, $C \gtrsim D$ is defined as $D \lesssim C$, and $C\eqsim D$ as $C\lesssim D$ and $C \gtrsim D$.

For normed linear spaces $E$ and $F$, by $\cL(E,F)$ we will denote the normed linear space of bounded linear mappings $E \to F$,
and by $\Lis(E,F)$ its subset of boundedly invertible linear mappings $E \to F$.
We write $E \hookrightarrow F$ to denote that $E$ is continuously embedded into $F$.
For simplicity only, we exclusively consider linear spaces over the scalar field $\R$.

\section{Space-time adaptivity for a parabolic model problem}
\label{sec:problem}
In this section, we summarize the relevant parts of~\cite[\S2--5]{Stevenson2020b}.

Let $V, H$ be separable Hilbert spaces of functions on some ``spatial domain'' such that $V \hookrightarrow H$ with dense and compact embedding.
Identifying $H$ with its dual, we obtain the Gelfand triple
$V \hookrightarrow H \simeq H' \hookrightarrow V'$.

For a.e.
\[ t \in I:=(0,T), \]
let $a(t;\cdot,\cdot)$ denote a bilinear form on $V \times V$ so that for
any $\eta,\zeta \in V$, $t \mapsto a(t;\eta,\zeta)$ is measurable on $I$,
and such that for a.e.~$t\in I$,
\begin{alignat*}{3}
|a(t;\eta,\zeta)|
& \lesssim  \|\eta\|_{V} \|\zeta\|_{V} \quad &&(\eta,\zeta \in V) \quad &&\text{({\em boundedness})},
\nonumber \\
a(t;\eta,\eta) &\gtrsim \|\eta\|_{V}^2 \quad
&&(\eta \in {V}) \quad &&\text{({\em coercivity})}.
\end{alignat*}

With $(A(t) \cdot)(\cdot) := a(t; \cdot, \cdot) \in \Lis({V},V')$, given a forcing function $g$ and initial value $u_0$, we want to solve the {\em parabolic initial value problem} of
\begin{equation} \label{11}
\text{finding $u: I \to V$ such that} \quad
\left\{
\begin{array}{rl}
\frac{\D u}{\D t}(t) +A(t) u(t)& = g(t) \quad(t \in I),\\
u(0) & = u_0.
\end{array}
\right.
\end{equation}
\begin{example}
For the model problem of the \emph{heat equation} on
some spatial domain $\Omega \subset \R^d$ we select $V := H_0^1(\Omega)$, $H := L_2(\Omega)$, and $a(t; \eta, \zeta) := \int_\Omega \grad_{\vec x} \eta \cdot \grad_{\vec x} \zeta \dif {\vec x}$.
\end{example}
In our simultaneous space-time variational formulation, the parabolic problem is to find $u$ s.t.
\[
(B u)(v):=\int_I
\langle { \textstyle \frac{\D u}{\D t}}(t), v(t)\rangle_H +
a(t;u(t),v(t)) \D t = \int_I
\langle g(t), v(t)\rangle_H =:g(v)
\]
for all $v$ from some suitable space of functions of time and space.
One possibility to enforce the initial condition is by testing against additional test functions.
\begin{theorem}[{\cite{Schwab2009}}] \label{thm0} With $X:=L_2(I;{V}) \cap H^1(I;V')$, $Y:=L_2(I;{V})$, we have
$$
\left[\begin{array}{@{}c@{}} B \\ \gamma_0\end{array} \right]\in \Lis(X,Y' \times H),
$$
where for $t \in \bar{I}$, $\gamma_t\colon u \mapsto u(t,\cdot)$ denotes the trace map.
In other words,
\be \label{x12}
\text{finding $u \in X$ s.t.} \quad (Bu, \gamma_0 u) = (g, u_0) \quad \text{given} \quad (g, u_0) \in Y' \times H
\ee
is a well-posed simultaneous space-time variational formulation of \eqref{11}.
\end{theorem}
We define $A \in \Lis(Y, Y')$ and $\partial_t \in \Lis(X, Y')$ as
\[
(Au)(v) := \int_I a(t; u(t), v(t)) \D t, \quad \text{and} \quad \partial_t := B - A.
\]
Following~\cite{Stevenson2020a}, we assume that $A$ is \emph{self-adjoint}.
Morever, in view of an efficient implementation, we assume that $A$ is a finite sum of tensor-product operators. If $A$ does not have this structure, one may alternatively consider (low-rank) tensor-product approximations of $A$, see e.g.~\cite{Hackbusch2012} for an overview.

We equip $Y$ and $X$ with \emph{`energy'-norms}
$$
\|\cdot\|_Y^2:=(A \cdot)(\cdot),\quad
\|\cdot\|_X^2:=\|\partial_t \cdot\|_{Y'}^2 + \|\cdot\|_Y^2 + \|\gamma_T \cdot\|_H^2,
$$
which are equivalent to the canonical norms on $Y$ and $X$.

The solution $u$ of \eqref{x12} equals the solution of the following minimization problem
\be\label{eqn:contmin}
u=\argmin_{w \in X} \|Bw-g\|_{Y'}^2+\|\gamma_0 w-u_0\|_H^2,
\ee
which in turn is the second component of the solution of
\begin{equation}\label{m1}
\text{finding $(\mu, u) \in Y \times X$ s.t.} \quad
\left[\begin{array}{@{}cc@{}} A & B\\ B' & -\gamma_0'\gamma_0 \end{array}\right]
\left[\begin{array}{@{}c@{}} \mu \\ u \end{array}\right]=
\left[\begin{array}{@{}c@{}} g \\ -u_0 \end{array}\right].
\end{equation}
Indeed, taking the Schur complement of~\eqref{m1} w.r.t.~the $Y$-block results in the Euler-Lagrange equations of~\eqref{eqn:contmin}.

\subsection{Discretizations}
Take a family $(X^\delta)_{\delta \in \Delta}$ of closed subspaces of $X$, and define
\begin{equation}\label{eq:galerkin}
u_\delta=\argmin_{w \in X^\delta} \|Bw-g\|_{Y'}^2+\|\gamma_0 w-u_0\|_H^2,
\end{equation}
being the best approximation to $u$ from $X^\delta$ w.r.t.~$\|\cdot\|_X$.
Solving this problem, however, is not feasible because of the presence of the dual norm.
Therefore, take $(Y^\delta)_{\delta \in \Delta}$ to be a family of closed subspaces of $Y$ such that
\begin{equation}
X^\delta \subseteq Y^\delta \quad (\delta \in \Delta), \quad \text{and} \quad \gamma_{\Delta} := \inf_{\delta \in \Delta} \inf_{0 \not= w \in X^\delta} \sup_{0 \not= v \in Y^\delta} \frac{(\partial_t w)(v)}{\|\partial_t w\|_{Y'} \|v\|_Y} > 0.
\label{eqn:stability}
\end{equation}
For $\udelta \in \Delta$ with $Y^{\udelta} \supseteq Y^\delta$, we replace $Y'$ by ${Y^{\udelta}}'$ in~\eqref{eq:galerkin} yielding the
approximation
\[
u^{\udelta \delta} = \argmin_{w \in X^\delta} \|Bw - g\|_{{Y^{\udelta}}'}^2 + \|\gamma_0 w - u_0\|_H^2.
\]
Notice that $u^{\udelta \delta}$ approximates $u_{\delta}$ in that $u^{\udelta \delta} = u_{\delta}$ when $Y^{\udelta} = Y$.

With $E_Y^{\udelta}: Y^{\udelta} \to Y$ and $E_X^\delta: X^\delta \to X$ denoting the trivial embeddings, $u^{\udelta \delta}$ is the second component of the solution of
\[
\begin{bmatrix} {E_Y^{\udelta}}' A E_Y^{\udelta} & {E_Y^{\udelta}}' B E_X^\delta \\ {E_X^\delta}' B' E_Y^{\udelta} & - {E_X^\delta}' \gamma_0' \gamma_0 E_X^\delta \end{bmatrix} \begin{bmatrix} \mu^{{\udelta} \delta} \\ u^{{\udelta} \delta} \end{bmatrix} = \begin{bmatrix} {E_Y^{\udelta}}' g \\ - {E_X^\delta}' \gamma_0' u_0 \end{bmatrix}.
\]
Taking the Schur complement w.r.t.~the $Y^{\udelta}$-block then leads to the equation
\begin{equation}
\label{eqn:discr-saddle}
\begin{split}
{E^\delta_X}'(B' E^{{\udelta}}_Y ({E_Y^{\udelta}}' A E_Y^{\udelta})^{-1} {E^{\udelta}_Y}' B &+\gamma_0'\gamma_0)E^\delta_X u^{{\udelta} \delta} \\
&= {E^\delta_X}' (B' E^{{\udelta}}_Y ({E_Y^{\udelta}}' A E_Y^{\udelta})^{-1} {E^{{\udelta}}_Y}' g+\gamma_0' u_0),
\end{split}
\end{equation}
which has a unique solution (cf.~\cite[Lem.~3.3]{Stevenson2020b}) that satisfies $\|u - u^{\udelta \delta}\|_X \leq \gamma_\Delta^{-1} \|u - u_\delta\|_X$ whenever $Y^{\udelta} \supseteq Y^\delta$; cf.~\cite[Thm.~3.7]{Stevenson2020a}.
For now, we assume the right-hand side of \eqref{eqn:discr-saddle} to be evaluated exactly. Later, in~\S\ref{sec:rhs}, we will discuss approximation of the right-hand side.

In view of obtaining an efficient solver,
we want to replace the inverses in~\eqref{eqn:discr-saddle} while aiming to preserve quasi-optimality of the solution.
To this end, let $K_Y^{\udelta} = {K_Y^{\udelta}}' \in \Lis({Y^{\udelta}}', Y^{\udelta})$ be a uniformly
optimal preconditioner for ${E_Y^{\udelta}}' A E_Y^{\udelta}$ that can be applied in linear complexity. Then, for some $\kappa_\Delta \geq 1$ we have
\[
\frac{((K^{\udelta}_Y)^{-1} v)(v)}{(Av)(v)} \in [\kappa^{-1}_\Delta, \kappa_\Delta] \quad (\delta \in \Delta, v \in Y^{\udelta}).
\]

\newer{Replacing $({E_Y^{\udelta}}' A E_Y^{\udelta})^{-1}$ by $K_Y^{\udelta}$, we denote the solution of~\eqref{eqn:discr-saddle} again by $u^{\udelta \delta}$. It} is quasi-optimal with $\|u - u^{\udelta \delta}\|_X \leq \tfrac{\kappa_\Delta}{\gamma_\Delta} \|u - u_\delta\|_X$; cf.~\cite[Rem.~3.8]{Stevenson2020a}.

\subsection{Adaptive refinement loop}
Our adaptive loop, given in Algorithm~\ref{alg:adaptive}, takes the familiar \underline{Solve}, \underline{Estimate}, \underline{Mark and refine} steps, and is driven by an efficient and reliable `hierarchical basis' a posteriori error estimator.
%
%

\newer{The adaptive loop below requires a saturation assumption.}
Define a \emph{partial order} on $\Delta$ by $\tilde \delta \succeq \delta$ whenever $X^{\tilde \delta} \supseteq X^{\delta}$.
Let $\delta \mapsto \udelta \succeq \delta$ be a mapping providing \emph{saturation} in that for some $\zeta < 1$,
\begin{equation}\label{eqn:saturation}
\|u - u_{\udelta}\|_X \leq \zeta \|u - u_\delta\|_X \quad (\delta \in \Delta).
\end{equation}
\newer{With this choice of $\udelta$, we are interested in finding $u^\delta := u^{\udelta \delta} \in X^\delta$ that solves
\be
\label{eqn:discr-schur}
\underbrace{{E^\delta_X}'(B' E^{{\udelta}}_Y K_Y^{{\udelta}} {E^{\udelta}_Y}' B +\gamma_0'\gamma_0)E^\delta_X}_{S^{\udelta \delta}:=}
u^{\delta}
=
\underbrace{{E^\delta_X}' (B' E^{{\udelta}}_Y K_Y^{\udelta} {E^{{\udelta}}_Y}' g+\gamma_0' u_0)}_{f^\delta :=}.
\ee}

Notice that~\eqref{eqn:discr-schur} is uniquely solvable even with $X^\udelta$ as `trial space', and we use this `room' between $X^\delta$ and $X^\udelta$ to our advantage.
Expanding $X^\delta$ to some intermediate space $X^\delta \subset X^{\tilde \delta} \subset X^\udelta$ yields a $u^{\tilde \delta}$ that is a better approximation to $u$ than $u^\delta$; cf.~\cite[Prop.~4.2]{Stevenson2020b}. This function will be the successor of $u^\delta$ in our loop, and we will show that the resulting sequence of functions converges $r$-linearly to $u$; see Algorithm~\ref{alg:adaptive} and Theorem~\ref{thm:adaptive}.

\subsubsection*{Solving}\label{S:solving}
Instead of solving the symmetric positive definite system~\eqref{eqn:discr-schur} exactly, we construct an approximate solution $\hat u^\delta$ using Preconditioned Conjugate Gradients (PCG).
To this end, let $K_X^\delta = {K_X^\delta}' \in \Lis({X^{\delta}}', X^{\delta})$ be a uniformly optimal preconditioner for $S^{\udelta \delta}$.
Then  $((K_X^\delta)^{-1}w)(w) \eqsim \|w\|^2_X \eqsim \|K_X^\delta S^{\udelta \delta}w\|^2_X $ for $w \in X^\delta$. Writing $w = K_X^\delta S^{\udelta \delta}(u^\delta - v^\delta$) reveals that this induces an algebraic error estimator
\be\label{eqn:algerr}
\beta^\delta(v^\delta) := \sqrt{(f^\delta - S^{ \udelta \delta} v^\delta)(K_X^\delta(f^\delta - S^{ \udelta \delta} v^\delta))} \eqsim \|u^\delta - v^\delta\|_X \quad (v^\delta \in X^\delta, \delta \in \Delta).
\ee
With $\hat u^\delta_k$ denoting the approximant at iteration $k$ of the PCG loop, $\beta^\delta(\hat u^\delta_k)$ is already available as $\sqrt{\beta_k}$, for $\beta_k$ the variable used in computing the next search direction.


\subsubsection*{Error estimation}
Let
$\Theta_\delta := \{\theta_\lambda: \lambda \in J_\delta\}$ be some uniformly $X$-stable basis satisfying  $X^\delta \oplus \spann \Theta_\delta = X^\udelta$, in that
\be \label{eqn:X-stable}
\|z + {\bf c}^\top \Theta_\delta\|_X^2 \eqsim \|z\|_X^2 + \|{\bf c}\|^2 \quad ({\bf c} \in \ell_2(J_\delta), z \in X^\delta, \delta \in \Delta).
\ee
Define the trivial embedding $P^\delta: X^\delta \to X^\udelta$. \newer{Akin to~\eqref{eqn:discr-schur}, we define $S^{\udelta \udelta}$ and  $f^{\udelta \udelta}$, and with it, the residual-based a posteriori error estimator ${\bf r}^\delta: X^\delta \to \ell_2(J_\delta)$, as}
\be\label{eqn:residual}
\begin{aligned}
&S^{\udelta \udelta} := \newer{{E^\udelta_X}' (B' E^{{\udelta}}_Y K_Y^{{\udelta}} {E^{\udelta}_Y}' B +\gamma_0'\gamma_0)E^\udelta_X}, \quad
f^{\udelta \udelta} := \newer{{E^\udelta_X}' (B' E^{{\udelta}}_Y K_Y^{\udelta} {E^{{\udelta}}_Y}' g+\gamma_0' u_0)}, \\
&{\bf r}^\delta({\hat u}^\delta) := (\newer{f^{\udelta \udelta}} - \newer{S^{\udelta \udelta}} P^\delta \hat u^\delta)(\Theta_\delta).
\end{aligned}
\end{equation}
For ${\hat u}^\delta$ close to $u^\delta$, the error estimator $\|{\bf r}^\delta({\hat u}^\delta)\|$ is reliable and efficient:


\begin{lemma}\label{lem:inexact}
Assume~\eqref{eqn:saturation} and~\eqref{eqn:X-stable}, $\tfrac{\kappa_\Delta}{\gamma_\Delta} < \tfrac{1}{\zeta}$, and fix some $\xi >0$ small enough. For $\hat u^\delta \in X^\delta$ satisfying $\beta({\hat u}^\delta) \leq \frac{\xi}{1- \xi} \| {\bf r}^\delta({\hat u}^\delta)\|$, we have
\[
\| {\bf r}^\delta({\hat u}^\delta) \| \eqsim \|u - {\hat u}^\delta\|_X \quad \text{and} \quad \|u - {\hat u}^\delta\|_X \lesssim \|u - u^\delta\|_X\quad (\delta \in \Delta).
\]
\end{lemma}
\begin{proof}
For convenience, we write $\hat{\bf r}^\delta := {\bf r}^\delta({\hat u}^\delta)$ and ${\bf r}^\delta := {\bf r}^\delta(u^\delta)$.

By \eqref{eqn:saturation}, \eqref{eqn:X-stable} and $\tfrac{\kappa_\Delta}{\gamma_\Delta} < \tfrac{1}{\zeta}$,  \cite[Prop.~4.4]{Stevenson2020b} shows that
\begin{equation}\label{eqn:residual-estim}
\|{\bf r}^\delta\| \eqsim \|u - u^{\delta}\|_X \quad(\delta \in \Delta).
\end{equation}

\newer{From~\eqref{eqn:X-stable} one deduces that $\|{\bf r}^\delta - \hat{\bf r}^\delta\| \lesssim  \|u^\delta - {\hat u}^\delta\|_X$; cf.~\cite[(4.13)]{Stevenson2020b}.}
By assumption, for $\xi < 1$, we find $\beta^\delta({\hat u}^\delta) \lesssim \xi \| \hat{\bf r}^\delta\|$.
Combined this reveals
\begin{equation}\label{eqn:res-perturb}
\| {\bf r}^\delta - \hat{\bf r}^\delta\| \stackrel{\mathclap{\eqref{eqn:X-stable}}}{\lesssim}  \|u^\delta - {\hat u}^\delta\|_X
\stackrel{\mathclap{\eqref{eqn:algerr}}}{\eqsim} \beta^\delta({\hat u}^\delta) \lesssim \xi \| \hat {\bf r}^\delta\|.
\end{equation}
Using this, we can show reliability of the estimator by
\begin{align*}
\|u-{\hat u}^\delta\|_X &\leq \|u-u^\delta\|_X + \|u^\delta - {\hat u}^\delta\|_X\\
&\stackrel{\mathclap{\makebox[0pt][r]{\scriptsize \eqref{eqn:residual-estim},}\eqref{eqn:algerr}}}{\eqsim} \| {\bf r}^\delta\| + \beta^\delta({\hat u}^\delta)
\leq \| \hat{\bf r}^\delta\| + \| {\bf r}^\delta - \hat{\bf r}^\delta\| + \beta^\delta({\hat u}^\delta)  \stackrel{\mathclap{\eqref{eqn:res-perturb}}}{\lesssim}  \| \hat{\bf r}^\delta\|.
\end{align*}
For efficiency of the estimator, we deduce
\begin{align*}
\| \hat{\bf r}^\delta\| &\stackrel{\mathclap{\eqref{eqn:residual-estim}}}{\lesssim} \|u - u^\delta\|_X + \| {\bf r}^\delta - \hat{\bf r}^\delta\|  \leq \|u - {\hat u}^\delta\|_X + \|u^\delta - {\hat u}^\delta\|_X + \|{\bf r}^\delta - \hat{\bf r}^\delta\| \\
&\stackrel{\mathclap{\eqref{eqn:res-perturb}}}{\lesssim} \|u - {\hat u}^\delta\|_X + \xi\| \hat{\bf r}^\delta\|,
\end{align*}
so taking $\xi$ sufficiently small and kicking back $\| \hat{\bf r}^\delta\|$ yields
\begin{equation}\label{eqn:residual-eff-bla}
\| \hat{\bf r}^\delta\|    \lesssim \|u - {\hat u}^\delta\|_X.
\end{equation}

Similarly, from~\eqref{eqn:residual-estim} and~\eqref{eqn:res-perturb} it follows that
\begin{equation}\label{eqn:residual-eff}
\| \hat{\bf r}^\delta\|    \lesssim \|u -  u^\delta\|_X.
\end{equation}
We infer quasi-optimality of ${\hat u}^\delta$  from
\[
\|u - {\hat u}^\delta\|_X
\stackrel{\mathclap{\eqref{eqn:res-perturb}}}{\lesssim}  \|u - u^\delta\|_X + \xi \| \hat{\bf r}^\delta\|
\stackrel{\mathclap{\eqref{eqn:residual-eff}}}{\lesssim} \|u - u^\delta\|_X.\qedhere
\]
\end{proof}

\newer{In the solve step, we need to iterate PCG until $\beta^\delta(\hat u^\delta_k) / \|{\bf r}^\delta(\hat u^\delta_k)\|$ is small enough. In the algorithm below, this is ensured by the do-while loop which also avoids the (expensive) recomputation of the residual at every PCG iteration.}



\subsubsection*{Marking and refinement}
\newer{Denoting the output of the solve step by $\hat u^\delta$, we} drive the adaptive loop  by performing D\"orfler marking on the residual $\hat {\bf r}^\delta := {\bf r}^\delta(\hat u^\delta)$, i.e., for some $\theta \in (0,1]$, we mark the smallest set $J \subset J_\delta$ for which
$\|\hat {\bf r}^\delta|_J\| \geq \theta \|\hat {\bf r}^\delta\|$.
We then construct the smallest $\tilde \delta \succeq \delta$ such that $X^{\tilde \delta}$ contains $\spann \Theta_\delta|_J$.

\begin{algorithm}\label{alg:adaptive}
\KwData{$\theta \in (0, 1]$, $\xi \in (0, 1)$, $\delta := \delta_{\text{init}} \in \Delta$;}
\BlankLine
$t_{\delta} := \mathcal E^{\delta}(0) = \sqrt{ ({E_Y^\udelta}'g)( K_Y^\udelta {E_Y^\udelta}' g) + \|u_0\|^2_{H}}$\;
\Loop{
\Solve{
    \Do{$t_{\delta} > \xi e_\delta$}{%
    Compute $\hat u_*^{\delta} \in X^{\delta}$ with $\beta^\delta(\hat u_*^{\delta}) \leq t_{\delta}/2$\;
    $t_{\delta} := \beta^\delta(\hat u_*^\delta)$\;
    $e_{\delta} := \|{\bf r}^\delta({\hat u}_*^\delta)\| + t_\delta$\;
    }
    $\hat u^\delta := \hat u_*^\delta$\;
}
\underline{Estimate:} Set $\hat {\bf r}^{\delta} := {\bf r}^{\delta}(\hat u^{\delta})$\;
\underline{Mark:}
Mark a smallest $J \subset J_\delta$ for which $\|\hat {\bf r}^\delta|_J\| \geq \theta \|\hat{\bf r}^\delta\|$\;
\underline{Refine:}
Determine the smallest $\tilde \delta \in \Delta$ such that $X^{\tilde \delta} \supset X^\delta \oplus \spann \Theta_\delta|_J$\;
$t_{\tilde \delta} := e_\delta$, $\delta := \tilde \delta$\;
}
\caption{Space-time adaptive refinement loop.}
\end{algorithm}

\begin{theorem}[{\cite[Thm.~4.9 with $\eta = 0$]{Stevenson2020b}}]\label{thm:adaptive}
Assume~\eqref{eqn:saturation} and~\eqref{eqn:X-stable}. For $\xi$ and $\tfrac{\kappa_\Delta}{\gamma_\Delta}-1$ sufficiently small with $\tfrac{\kappa_\Delta}{\gamma_\Delta}-1 \downarrow 0$ when $\theta \downarrow 0$, the sequence of approximations produced by Algorithm~\ref{alg:adaptive} converges $r$-linearly to $u$, in that after every iteration,
$\|u - \hat u^{\delta}\|_X$ decreases with a factor at least $\rho < 1$.
\end{theorem}
\begin{remark}
In a practical implementation, to ensure termination, Algorithm~\ref{alg:adaptive} has to be complemented by an appropriate stopping criterium; cf.~\cite[Alg.~4.8]{Stevenson2020b}.
\end{remark}
%
\begin{proof}
For convenience, we denote ${\bf r}^\delta := {\bf r}^\delta(u^\delta)$ and $\hat{\bf r}^\delta := {\bf r}^\delta({\hat u}^\delta)$.  The stopping criterium of the solve step ensures that
\newer{$\beta^\delta({\hat u}^\delta) \leq \xi \big(\|\hat{\bf r}^\delta\| + \beta^\delta({\hat u}^\delta)\big)$},
so for $\xi < 1$ we are in the situation of Lemma~\ref{lem:inexact}.

We have\[
\|\hat{\bf r}^\delta  - {\bf r}^\delta\| \stackrel{\mathclap{\eqref{eqn:res-perturb}}}{\lesssim}
\xi \| \hat{\bf r}^\delta\|
\leq \xi \big(\|{\bf r}^\delta\| +   \|\hat{\bf r}^\delta  - {\bf r}^\delta\|\big),\] so
taking $\xi$ sufficiently small and kicking back $\|\hat{\bf r}^\delta  - {\bf r}^\delta\|$ yields
\begin{equation}\label{eqn:residual-perturb-bla}
\|\hat{\bf r}^\delta  - {\bf r}^\delta\| \lesssim \xi \| {\bf r}^\delta\|.
\end{equation}
After marking, we have
$\| \hat{\bf r}^\delta\| \leq \theta^{-1} \| \hat{\bf r}^\delta|_J\|$, which shows that
\[
\| {\bf r}^\delta \|  \stackrel{\mathclap{\eqref{eqn:res-perturb}}}{\lesssim} \| \hat{\bf r}^\delta\|
\lesssim \| \hat{\bf r}^\delta|_J\|\\ \leq \| {\bf r}^\delta|_J\| + \| {\bf r}^\delta - \hat{\bf r}^\delta\|  \stackrel{\mathclap{\eqref{eqn:residual-perturb-bla}}}{\lesssim} \| {\bf r}^\delta|_J\| + \xi \| {\bf r}^\delta\|,
\]
so for $\xi$ small enough, kicking back $\| {\bf r}^\delta\|$  reveals that  for a $ \hat \theta > 0$ dependent on $\theta$,
\[
\| {\bf r}^\delta|_J\| \geq \hat \theta \|{\bf r}^\delta\|.
\]
From \cite[Prop.~4.3]{Stevenson2020b} we now find that, for $\frac{\kappa_\Delta}{\gamma_\Delta} - 1\downarrow 0$ when $\theta \downarrow 0$, there is a $\bar{\rho} < 1$ for which
\begin{equation}\label{eqn:contract}
\| u - u^{\tilde \delta} \|_X \leq \bar{\rho} \|u - u^\delta\|_X.
\end{equation}

Combining the results shows that
\begin{align*}
\| u - {\hat u}^{\tilde \delta}\|_X &\leq \|u - u^{\tilde \delta}\|_X + \|u^{\tilde \delta} - {\hat u}^{\tilde \delta}\|_X \\
&\stackrel{\mathclap{\makebox[0pt][r]{\scriptsize\eqref{eqn:res-perturb},}\eqref{eqn:residual-eff}}}{\leq} (1 + \mathcal{O}(\xi)) \|u - u^{\tilde \delta}\|_X\\
&\stackrel{\mathclap{\eqref{eqn:contract}}}{\leq} (1 + \mathcal{O}(\xi)) \bar \rho \|u - u^{\delta}\|_X\\
&\leq (1 +  \mathcal{O}(\xi)) \bar \rho (\|u - {\hat u}^{\delta}\|_X + \|u^{\delta} - {\hat u}^{\delta}\|_X)\\
&\stackrel{\mathclap{\makebox[0pt][r]{\scriptsize\eqref{eqn:res-perturb},}\eqref{eqn:residual-eff-bla}}}{\leq} \underbrace{(1 +  \mathcal{O}(\xi)) \bar \rho}_{=: \rho}\|u - {\hat u}^\delta\|_X,
\end{align*}
so for $\xi$ small enough, $\rho < 1$ and the proof of $r$-linear convergence is complete.
\end{proof}

\subsection{Adaptive trial- and test spaces}
\newer{The convergence rate of our adaptive loop} is determined by the approximation properties of the family $(X^\delta)_{\delta \in \Delta}$.
We want to construct a family that allows for \emph{local} refinements. Here, the crucial problem is guaranteeing the inf-sup stability condition~\eqref{eqn:stability}.
It is known that inf-sup stability is satisfied for \emph{full} tensor-products of (non-uniform) finite element spaces, and in~\cite[Prop.~4.2]{Andreev2013}, this result was generalized to families of \emph{sparse} tensor-products. Unfortunately,
neither family allows for adaptive refinements both locally in time and space.

In \S\ref{sec:heat} we will solve \newer{this} by first equipping $X$ with a tensor-product of (infinite) bases: a wavelet basis $\Sigma$ in time, and a hierarchical basis in space.
We then construct $X^\delta$ as the span of a (finite) subset of this tensor-product basis, which we grow
by adding particular functions.

By imposing a \emph{double-tree} constraint on the index set of the basis of $X^\delta$, we can apply tensor-product operators in linear complexity; see \S\ref{sec:lincomp}. Moreover, this constraint implies that for our model problem
the inf-sup condition~\eqref{eqn:stability} is satisfied and  we can construct optimal preconditioners $K_Y^\udelta$ and $K_X^\delta$.


\section{The application of linear operators in linear complexity}\label{sec:lincomp}
\newer{{\noindent}An efficient implementation of our adaptive loop requires the efficient application of the operators ${E_Y^\udelta}' B E_X^\delta$ and ${E_X^\delta}' \gamma_0' \gamma_0 E_X^\delta$ appearing in~\eqref{eqn:discr-schur}. Both terms are finite sums of tensor-products of operators in time and space.
When we equip our trial and test spaces with tensor-products of multilevel bases, it turns out that we can evaluate these operators in linear complexity.}

\newer{More precisely, this section will} show the abstract result that given
\begin{itemize}
\item tensor-products $\Psi := \Psi^0 \times \Psi^1$, $\breve \Psi := \breve \Psi^0 \times \breve \Psi^1$ of multilevel bases $\Psi^0$, $\Psi^1$, $\breve \Psi^0$, $\breve \Psi^1$ indexed by $\vee^0$, $\vee^1$, $\breve \vee^0$, $\breve \vee^1$, and
\item (finite) subsets
${\bf \Lambda} \subset \vee^0 \times \vee^1, \breve {\bf \Lambda} \subset \breve \vee^0 \times \breve \vee^1$ that are \emph{double-trees}, and
\item linear operators $A_i: \spann \Psi^i \to (\spann \breve \Psi^i)'$ that are \emph{local} ($i \in \{0,1\}$),
\end{itemize}
we can apply the matrix $((A_0 \otimes A_1) \Psi|_{\bf \Lambda})(\breve \Psi|_{\breve{\bf \Lambda}})$ 
in $\mathcal O(\# {\bf \Lambda} + \# \breve{\bf \Lambda})$ operations even though this matrix is not uniformly sparse.

\begin{example}\label{ex:B}
For our model problem, \newer{$\Psi^0$ and $\breve \Psi^0$ will be wavelets for $H^1(I)$ or $L_2(I)$ in time, and $\Psi^1 = \breve \Psi^1$ will be a hierarchical finite element basis for $H_0^1(\Omega)$ in space}. We will apply the result of this section to the operators $\gamma_0' \gamma_0$ and $B = \partial_t + A$.
\end{example}

We will achieve this complexity using a variant of the \emph{unidirectional principle}.
Denote with $I_{\bf \Lambda}$ the extension with zeros of a vector supported on ${\bf \Lambda}$ to one on $\vee^0 \times \vee^1$, and with $R_{\bf \Lambda}$ its adjoint; define $I_{\breve {\bf \Lambda}}$ and $R_{\breve {\bf \Lambda}}$ analogously.
Define ${\bf A}_i := (A_i \Psi^i)(\breve \Psi^i)$. We will split ${\bf A}_0$ in its upper and strictly lower triangular parts ${\bf U}_0$ and ${\bf L}_0$, so that
\[
R_{\breve {\bf \Lambda}} ({\bf A}_0 \otimes {\bf A}_1) I_{\bf \Lambda} =
R_{\breve {\bf \Lambda}} ({\bf L}_0 \otimes \Id)(\Id \otimes {\bf A}_1) I_{\bf \Lambda} +
R_{\breve {\bf \Lambda}} ({\bf U}_0 \otimes \Id)(\Id \otimes {\bf A}_1) I_{\bf \Lambda}.
\]
This in itself is not useful, as $(\Id \otimes {\bf A}_1) I_{\bf \Lambda}$ maps into a vector space which dimension we cannot control. However, the restriction $R_{\breve{\bf \Lambda}}$ gives us elbow room: in Theorem~\ref{thm1} we construct double-trees ${\bf \Sigma}, {\bf \Theta}$ with $\# {\bf \Sigma} + \# {\bf \Theta} \lesssim \# \breve {\bf \Lambda} + \# {\bf \Lambda}$ s.t.
\be\label{eqn:split}
\begin{cases}
R_{\breve {\bf \Lambda}} ({\bf L}_0 \otimes \Id)(\Id \otimes {\bf A}_1) I_{\bf \Lambda} =
R_{\breve {\bf \Lambda}} ({\bf L}_0 \otimes \Id)R_{\bf \Sigma} I_{\bf \Sigma}(\Id \otimes {\bf A}_1) I_{\bf \Lambda}, \\
R_{\breve {\bf \Lambda}} ({\bf U}_0 \otimes \Id)(\Id \otimes {\bf A}_1) I_{\bf \Lambda} =
R_{\breve {\bf \Lambda}} ({\bf U}_0 \otimes \Id)R_{\bf \Theta} I_{\bf \Theta}(\Id \otimes {\bf A}_1) I_{\bf \Lambda}.
\end{cases}
\ee
These right hand sides we \emph{can} apply efficiently, and their application boils down to applications of ${\bf L}_0$, ${\bf U}_0$, and ${\bf A}_1$ in a \emph{single} coordinate direction only.
Simple matrix-vector products are inefficient though, as these matrices are again not uniformly sparse. However, by using the properties of a double-tree and the sparsity of the operator in \emph{single scale}, we can evaluate ${\bf U}_0, {\bf L}_0$ and  ${\bf A}_1$ in linear time; see \S\ref{sec:eval_time}.


We follow the structure of~\cite[\S3]{Kestler2014a}, which applies the aforementioned idea to \emph{multi-trees} though with a slightly more restrictive definition of a \emph{tree}.
For readability, we defer the proofs of Theorems~\ref{thm:eval},~\ref{thm:evalupp},~\ref{thm:evallow}, and~\ref{thm1} to Appendix~\ref{app:proofs}.

\subsection{Evaluation of linear operators w.r.t.~trees.}
\label{sec:eval_time}
Let $\Psi$ be a (multilevel) collection of functions on some domain $Q$.
\begin{example}
In our application, $Q$ will be either the time interval $I$ with $\Psi$ being a collection of wavelets, or the spatial domain $\Omega$, in which case $\Psi$ is a collection of hierarchical basis functions.
\end{example}

Writing $\Psi=\{\psi_\lambda : \lambda \in \vee\}$, we assume that the $\psi_\lambda$ are \emph{locally supported} in the sense that with $|\lambda| \in \N_0$ denoting the \emph{level} of $\lambda$,
\begin{align} \label{c1}
& \sup_{\lambda \in \vee} 2^{|\lambda|} \diam{\supp \psi_\lambda} <\infty,\\
\label{c2}
& \sup_{\ell \in \N_0} \sup_{x \in Q} \#\{ \lambda \in \vee : |\lambda|=\ell \wedge
\supp \psi_\lambda \cap B(x;2^{-\ell}) \neq\emptyset\} <\infty.
\end{align}

We will refer to the functions $\psi_\lambda$ as being {\em wavelets}, although not necessarily they have vanishing moments or other specific wavelet properties.

For $\ell \in \N_0$, and any $\Lambda \subset \vee$, we set $\Lambda_\ell:=\{\lambda \in \Lambda  : |\lambda|=\ell\}$ and
$\Lambda_{\ell\uparrow}:=\{\lambda \in \Lambda  : |\lambda| \geq \ell\}$, and write $\Psi_\ell:=\Psi|_{\vee_\ell}$.

For $\ell \in \N_0$, we assume a collection
$
\Phi_\ell=\{\phi_\lambda : \lambda \in \Delta_\ell\},
$
whose members will be referred to as being {\em scaling functions},
with
\begin{align} \label{c4}
& \Span \Phi_{\ell+1} \supseteq  \Span \Phi_{\ell} \cup \Psi_{\ell+1},\quad \Phi_0=\Psi_0 \quad(\Delta_0:=\vee_0), \\
\label{c5}
& \sup_{\ell \in \N_0}\sup_{\lambda \in \Delta_\ell} 2^{\ell} \diam{\supp \phi_\lambda} <\infty, \\
\label{c6}
& \sup_{\ell \in \N_0} \sup_{x \in Q} \#\{\lambda \in \Delta_\ell  :  \supp \phi_\lambda \cap B(x;2^{-\ell}) \neq \emptyset\} <\infty,\\
\label{c7}
& \{\phi_\lambda|_\Sigma  :  \lambda \in \Delta_\ell,\,\phi_\lambda|_\Sigma \not\equiv 0\} \text{ is independent } (\text{for all open } \Sigma \subset Q,\,\ell \in \N_0).
\end{align}
W.l.o.g.~we assume that the index sets $\Delta_\ell$ for different $\ell$ are mutually disjoint, and set $\Phi:=\cup_{\ell \in \N_0} \Phi_\ell$ with index set $\Delta:=\cup_{\ell \in \N_0} \Delta_\ell$. For $\lambda \in \Delta$, we set $|\lambda|:=\ell$ when $\lambda \in \Delta_\ell$.

Viewing $\Psi_\ell$, $\Phi_\ell$ as column vectors, the assumptions we made so far guarantee the existence of matrices ${\mathfrak p}_{\ell}$, ${\mathfrak q}_{\ell}$ such that
\[
\left[\begin{array}{@{}cc@{}} (\Phi_{\ell-1})^\top & (\Psi_\ell)^\top \end{array} \right] = (\Phi_{\ell})^\top
\left[\begin{array}{@{}cc@{}} {\mathfrak p}_{\ell} & {\mathfrak q}_{\ell} \end{array}\right],
\]
where the number of non-zeros per row and column of ${\mathfrak p}_{\ell}$ and ${\mathfrak q}_{\ell}$ is finite, uniformly in the rows and columns and in $\ell \in \N$ (here also \eqref{c7} has been used).
We refer to ${\mathfrak p}_\ell$ as the \emph{prolongation matrix}. \newer{Columns of $\mathfrak p_\ell$ contain the \emph{mask} of the scaling functions, whereas columns of ${\mathfrak q}_\ell$ contain the mask of the wavelets}.

To each $\lambda \in \vee$ with $|\lambda|>0$, we associate one or more $\mu \in \vee$ with $|\mu|=|\lambda|-1$ and $|\supp \psi_\lambda \cap \supp \psi_\mu|>0$.
We call $\mu$ a {\em parent} of $\lambda$, and so $\lambda$ a {\em child} of $\mu$.

To each $\lambda \in \vee$, we associate some neighbourhood  $S(\lambda)$ of $\supp \psi_\lambda$, with diameter $\lesssim2^{-|\lambda|}$, such that for $|\lambda|>0$,
$S(\lambda) \subset \cup_{\mu \in \parent(\lambda)} S(\mu)$.

\begin{remark}
Such a neighborhood always exists even when a child has only one parent. Indeed with $C:=\sup_{\lambda \in \vee} 2^{|\lambda|} \diam{\supp \psi_\lambda}$ and
$S(\lambda):=\{x \in Q  : \dist(x,\supp \psi_\lambda) < C 2^{-|\lambda|}\}$,
for $\mu$ being a parent of $\lambda$ and  $x \in S(\lambda)$, $\dist(x,\supp \psi_\mu)\leq \dist(x,\supp \psi_\lambda)+\diam{\supp \psi_\lambda}< 2 C 2^{-|\lambda|}=C 2^{-|\mu|}$, i.e.,
$x \in S(\mu)$.
\end{remark}

\begin{definition}[Tree]
A finite $\Lambda \subset \vee_{\ell \uparrow}$ is called an \emph{$\ell$-tree}, or simply a \emph{tree} when $\ell=0$,  when for any $\lambda \in \Lambda$ its parents in $\vee_{\ell \uparrow}$ are in $\Lambda$. This is not a tree in the graph-theoretical sense, but rather one in the sense of a family history tree.
\end{definition}

\begin{example}[Hierarchical basis in one dimension]\label{ex:time-hbf}
Figure~\ref{fig:time-hbf} shows an example multilevel collection $\Psi$ of functions defined on the interval $[0,1]$. Its index set $\vee_{\mathfrak I}$ with parent-child relations is shown left, with a tree $\Lambda \subset \vee_{\mathfrak I}$ visualised in red. This collection is called the \emph{hierarchical basis}. With $S(\lambda) := \supp \psi_\lambda$ for $\lambda \in \vee_{\mathfrak I}$, the hierarchical basis satisfies conditions mentioned above.

\begin{figure}[ht]
\centering
\begin{tikzpicture}[scale=0.75]
    \fill[black!5] (-2.25,2) rectangle (14.5,1); \node[right] at (-2,1.5){$\ell=0$};
    \fill[black!0] (-2.25,1) rectangle (14.5,0); \node[right] at (-2,.5){$\ell=1$};
    \fill[black!5] (-2.25,0) rectangle (14.5,-1); \node[right] at (-2,-.5){$\ell=2$};
    \fill[black!0] (-2.25,-1) rectangle (14.5,-2); \node[right] at (-2,-1.5){$\ell=3$};
    \begin{scope}[shift={(0,.35)}]
    \node[red] at (0,1){\textbullet}; \node[red,above] at (0,1){$_{(0,0)}$};
    \node[red] at (4,1){\textbullet}; \node[red,above] at (4,1){$_{(0,1)}$};
    
    \draw[red] (0,1) -- (2,0) -- (4,1);
    \draw[red] (2,0) -- (1,-1);
    \draw (2,0) -- (3,-1);
    \node[red] at (2,0){\textbullet}; \node[red,above] at (2,0){$_{(1,0)}$}; 
    \node[red] at (1,-1){\textbullet}; \node[red,above] at (0.9,-1){$_{(2,0)}$};
    \node at (3,-1){\textbullet}; \node[above] at (3.1,-1){$_{(2,1)}$};
    
    \draw (1,-1) -- (0.5,-2);
    \node at (0.5,-2){\textbullet}; \node[above] at (0.35,-2){$_{(3,0)}$};
    \draw[red] (1,-1) -- (1.5,-2);
    \node[red] at (1.5,-2){\textbullet}; \node[red,above] at (1.65,-2){$_{(3,1)}$};
    \draw (3,-1) -- (2.5,-2);
    \node at (2.5,-2){\textbullet}; \node[above] at (2.35,-2){$_{(3,2)}$};
    \draw (3,-1) -- (3.5,-2);
    \node at (3.5,-2){\textbullet}; \node[above] at (3.65,-2){$_{(3,3)}$};
    
    \draw[dashed] (3.75,-2.65) -- (3.5,-2);
    \draw[dashed] (2.75,-2.65) -- (2.5,-2);
    \draw[dashed] (1.75,-2.65) -- (1.5,-2);
    \draw[dashed] (0.75,-2.65) -- (0.5,-2);
    \draw[dashed] (3.25,-2.65) -- (3.5,-2);
    \draw[dashed] (2.25,-2.65) -- (2.5,-2);
    \draw[dashed] (1.25,-2.65) -- (1.5,-2);
    \draw[dashed] (0.25,-2.65) -- (0.5,-2);
    \end{scope}
    \node[below] at (1.5,-3){Index set $\vee_{\mathfrak I}$ and tree ${\color{red} \Lambda} \subset \vee_{\mathfrak I}$};
    
    \begin{scope}[shift={(5.25,0)}]
    \draw (4.5,1.0) -- (4.5,1.75);
    \draw (4.4,1.0) -- (4.6,1.0);
    \draw (4.4,1.75) -- (4.6,1.75);
    \node at (4.25,1.75){$_1$};
    \node at (4.25,1.0){$_0$};
    \draw (0, 1) -- (4, 1); \draw[red] (0, 1.75) -- (4,1); \draw[red] (4, 1.75) -- (0,1);
    \node[red] at (0,1){\textbullet}; \node[red] at (4,1){\textbullet};
    \draw (0,0) -- (4,0); \draw[red] (0,0) -- (2,0.75) -- (4,0);
    \node[red] at (2,0){\textbullet};
    \draw (0,-1) -- (4,-1); \draw[red] (0,-1) -- (1,-0.25) -- (2,-1); \draw (2,-1) -- (3,-0.25) -- (4,-1);
    \node[red] at (1,-1){\textbullet}; \node at (3,-1){\textbullet};
    \draw (0,-2) -- (4,-2);
    \draw (0,-2) -- (0.5,-1.25) -- (1,-2);
    \draw[red] (1,-2) -- (1.5,-1.25) -- (2,-2);
    \draw (2,-2) -- (2.5,-1.25) -- (3,-2) -- (3.5,-1.25) -- (4,-2);
    \node at (0.5,-2){\textbullet}; \node[red] at (1.5,-2){\textbullet};
    \node at (2.5,-2){\textbullet}; \node at (3.5,-2){\textbullet};
    
    \node[below] at (2,-3){Multilevel functions $\Psi$};
    \end{scope}
    
    \begin{scope}[shift={(10,0)}]
    \draw (0, 1) -- (4, 1);
    \draw (0, 1.75) -- (4,1);
    \draw (4, 1.75) -- (0,1);
    
    \draw (0,0) -- (4,0);
    \draw (0,0) -- (2,0.75) -- (4,0);
    \draw (0,0.75) -- (2,0) -- (4,0.75);
    
    \draw (0,-1) -- (4,-1);
    \draw (0,-1) -- (1,-0.25) -- (2,-1) -- (3,-0.25) -- (4,-1);
    \draw (0,-0.25) -- (1,-1) -- (2,-0.25) -- (3,-1) -- (4,-0.25);
    
    \draw (0,-2) -- (4,-2);
    \draw (0,-2) -- (0.5,-1.25) -- (1,-2) -- (1.5,-1.25) -- (2,-2) -- (2.5,-1.25) -- (3,-2) -- (3.5,-1.25) -- (4,-2);
    \draw (0,-1.25) -- (0.5,-2) -- (1,-1.25) -- (1.5,-2) -- (2,-1.25) -- (2.5,-2) -- (3,-1.25) -- (3.5,-2) -- (4,-1.25);
    
    \node[below] at (2,-3){Scaling functions $\Phi$};
    \end{scope}
\end{tikzpicture}
\vspace{-1em}
\caption{Hierarchical basis for the interval $[0,1]$.}
\label{fig:time-hbf}
\end{figure}
\end{example}

\subsubsection{A routine \texttt{eval}} \label{Seval}
Let $(\Psi, \Phi)$ and $(\breve{\Psi}, \breve \Phi)$ satisfy the conditions of the previous subsection, and let $A \colon\Span \Phi \to (\Span \breve{\Phi})'$ be \emph{local} in that $(Au)(v)=(Au|_{\supp v})(v)$.
\newer{Typically, $A$ is a (partial) differential operator in variational form; e.g.~$A \in \cL(H^1(I), L_2(I)')$ with $(Au)(v) = \int_I \tfrac{\dif u}{\dif t} v \dif t$.}
For trees $\Lambda \subset \vee$ and $\breve \Lambda \in \breve \vee$, we are interested in the efficient application of the matrix $(A \Psi|_{\Lambda})(\breve \Psi|_{\breve \Lambda})$.

Just for brevity of the following argument, assume $\Psi = \breve \Psi$ and $\Phi = \breve \Phi$.  The matrix $(A \Psi|_{\Lambda})(\Psi|_{\Lambda})$ is not uniformly sparse, so a straight-forward matrix-vector product is not of linear complexity.
However, for $\Lambda$ a uniform tree up to level $\ell$, i.e.~$\Lambda = \{\lambda \in \vee: |\lambda| \leq \ell\}$, a solution is
provided by the multi- to single-scale transform $T$ characterized by $\Psi|_\Lambda = T^\top \Phi_{\ell}$ through the equality $(A \Psi|_{\Lambda})(\Psi|_{\Lambda}) = T^\top (A \Phi_\ell)(\Phi_\ell) T$, as the transforms can be applied in linear complexity and the single-scale matrix is uniformly sparse.

For general trees however, we don't have $\dim \Phi_\ell \lesssim \dim \Psi|_{\Lambda}$ so the previous approach is not of linear complexity.
Clever level-by-level multi-to-singlescale transformations and the prolongation of \emph{only} relevant functions \emph{does} allow applying $(A \Psi|_{\Lambda})(\breve \Psi|_{\breve \Lambda})$ in linear complexity; see Algorithm~\ref{alg:eval} below.

On several places the restriction of a vector (of scalars or of functions) to its indices in some subset of the index set should be read as the vector of full length where the entries with indices outside this subset are replaced by zeros.
For index sets $\Delta$ and $\breve \Delta$, matrix $\mathfrak{m} \in \R^{\# \breve \Delta \times \# \Delta}$, and subset $\Pi \subset \Delta$, we write $\supp(\mathfrak{m}, \Pi)  \subset \breve \Delta$ for the index set corresponding to the image of $\mathfrak{m}$ under $\{\vec x|_{\Pi}   :  \vec x \in \R^{\# \Delta}\}$.

\medskip

\begin{algorithm}[ht]\label{alg:eval}
\KwData{$\ell \in \N$, $\breve{\Pi} \subset \breve{\Delta}_{\ell-1}$, $\Pi \subset \Delta_{\ell-1}$,
    $\ell$-trees $\breve{\Lambda} \subset \breve{\vee}_{\ell\uparrow}$ and
    $\Lambda \subset {\vee}_{\ell\uparrow}$, $\vec{d}\in \R^{\# \Pi}$, $\vec{c}\in  \R^{\#  \Lambda}$.
}
\KwResult{$[\vec{e}, \vec{f}]$ where
    $\vec{e}=(Au)(\breve{\Phi}|_{\breve{\Pi}})$,
    $\vec{f}=(Au)(\breve{\Psi}|_{\breve{\Lambda}})$,
with $u:=\vec{d}^\top \Phi|_{\Pi}+\vec{c}^\top \Psi|_{\Lambda}$.
}
\If{$\breve{\Pi}\cup \breve{\Lambda}\neq \emptyset$}{
    $\breve{\Pi}_B:=
        \{\lambda \in \breve{\Pi}  : \big|\supp \breve{\phi}_\lambda \cap \cup_{\mu \in \Lambda_\ell}S(\mu)\big|>0\}$, $\breve{\Pi}_A:=\breve{\Pi} \setminus \breve{\Pi}_B$\\
    $\Pi_B:=\{\lambda \in \Pi : \big|\supp \phi_\lambda \cap \big(\cup_{\mu \in \breve{\Lambda}_\ell} \breve{S}(\mu) \cup_{\gamma \in \breve{\Pi}_B} \supp \breve{\phi}_\gamma \big)\big|>0\}$, $\Pi_A:=\Pi \setminus \Pi_B$\\
    $\breve{\underline{\Pi}} := \supp(\breve{\mathfrak{p}}_{\ell}, \breve{\Pi}_B) \cup \supp(\breve{\mathfrak{q}}_\ell, \breve{\Lambda}_\ell)$\\
    ${\underline{\Pi}} := \supp({\mathfrak{p}}_{\ell}, {\Pi}_B) \cup \supp({\mathfrak{q}}_\ell, {\Lambda}_\ell)$\\
    $\underline{\vec{d}}:=\mathfrak{p}_{\ell} \vec{d}|_{\Pi_B}+\mathfrak{q}_\ell \vec{c}|_{\Lambda_{\ell}}$

    $[\underline{\vec{e}},\,\underline{\vec{f}}]:=\mathtt{eval}(A)(\ell+1, \breve{\underline{\Pi}}, \breve{\Lambda}_{\ell+1\uparrow}, \underline{\Pi}, \Lambda_{\ell+1\uparrow},\underline{\vec{d}}, \vec{c}|_{\Lambda_{\ell+1\uparrow}})$\\
    $\vec{e}=
    \left[\begin{array}{@{}l@{}} \vec{e}|_{\breve{\Pi}_A}\\ \vec{e}|_{\breve{\Pi}_B}\end{array}\right]
    :=
    \left[\begin{array}{@{}l@{}} (A\Phi|_{\Pi})(\breve{\Phi}|_{\breve{\Pi}_A})\vec{d} \\
    (\breve{\mathfrak{p}}_{\ell} ^\top \underline{\vec{e}})|_{\breve{\Pi}_B}
    \end{array}\right]$\\
    $\vec{f}=
    \left[\begin{array}{@{}l@{}} \vec{f}|_{\breve{\Lambda}_\ell} \\ \vec{f}|_{\breve{\Lambda}_{\ell+1\uparrow}}\end{array}\right]
    :=
    \left[\begin{array}{@{}l@{}} (\mathfrak{\breve{q}}_\ell^\top \underline{\vec{e}})|_{\breve{\Lambda}_{\ell}} \\ \underline{\vec{f}} \end{array}\right]$\\
}
\caption{Function $\mathtt{eval}(A)$.}
\end{algorithm}

\begin{remark}  Let $\breve{\Lambda} \subset \breve{\vee}$, $\Lambda \subset \vee$ be  trees, and $\vec{c} \in \ell_2(\Lambda)$, then
\[
[\vec{e},\,\vec{f}]:=\mathtt{eval}(A)(1,\breve{\Lambda}_{0},\breve{\Lambda}_{1 \uparrow},\Lambda_{0},\Lambda_{1 \uparrow},\vec{c}|_{\Lambda_0},\vec{c}|_{\Lambda_{1 \uparrow}}),
\]
satisfies
\[
(A\Psi|_{\Lambda})(\breve{\Psi}|_{\breve{\Lambda}})\vec{c} = \begin{bmatrix}\vec{e} \\ \vec{f} \end{bmatrix}.
\]
\end{remark}

\begin{restatable}{theorem}{thmeval}\label{thm:eval}
A call of \texttt{eval} yields the output as specified, at the cost of
${\mathcal O}(\# \breve{\Pi}+\# \breve{\Lambda}+\# \Pi+\# \Lambda)$ operations.
\end{restatable}
\begin{proof}
See Appendix~\ref{app:proofs}.
\end{proof}

\subsubsection{Routines \texttt{evalupp} and \texttt{evallow}}
Let $A\colon\Span \Phi \to (\Span \breve{\Phi})'$ be local and \emph{linear}.
Set
\[
{\bf A}:=(A \Psi)(\breve{\Psi})= [(A\psi_\mu)(\breve{\psi}_\lambda)]_{(\lambda,\mu) \in \breve{\vee}\times \vee}
\]
as well as
${\bf U}:=[(A\psi_\mu)(\breve{\psi}_\lambda)]_{|\lambda| \leq |\mu|}$ and ${\bf L}:=[(A\psi_\mu)(\breve{\psi}_\lambda)]_{|\lambda| > |\mu|}$
so that  ${\bf A}={\bf L}+{\bf U}$.
As sketched in the introduction of this section, this splitting is going to be necessary for the application of system matrices in the tensor-product setting; see also~\eqref{eqn:split}. \newer{Algorithms~\ref{alg:evalupp} and~\ref{alg:evallow} below can be used to evaluate ${\bf U}$ and ${\bf L}$ in linear complexity.}

\begin{algorithm}[ht]\label{alg:evalupp}
\KwData{$\ell \in \N$, $\breve{\Pi} \subset \breve{\Delta}_{\ell-1}$, $\Pi \subset \Delta_{\ell-1}$,
    $\ell$-trees $\breve{\Lambda} \subset \breve{\vee}_{\ell\uparrow}$ and
    $\Lambda \subset {\vee}_{\ell\uparrow}$, $\vec{d}\in \R^{\# \Pi}$, $\vec{c}\in  \R^{\#  \Lambda}$.
}
\KwResult{$[\vec{e}, \vec{f}]$ where
    $\vec{e}=(Au)(\breve{\Phi}|_{\breve{\Pi}})$,
    $\vec{f}={\bf U}|_{\breve{\Lambda} \times \Lambda} \vec{c}$,
with $u:=\vec{d}^\top \Phi|_{\Pi}+\vec{c}^\top \Psi|_{\Lambda}$.
}
\If{$\breve{\Pi}\cup \breve{\Lambda}\neq \emptyset$}{
    $\breve{\Pi}_B:=\{\lambda \in \breve{\Pi} : \big|\supp \breve{\phi}_\lambda \cap \cup_{\mu \in \Lambda_\ell}S(\mu)\big|>0\}$, $\breve{\Pi}_A:=\breve{\Pi} \setminus \breve{\Pi}_B$\\

    $\breve{\underline{\Pi}} := \supp(\breve{\mathfrak{p}}_{\ell}, \breve{\Pi}_B) \cup \supp(\breve{\mathfrak{q}}_\ell, \breve{\Lambda}_\ell)$\\
    ${\underline{\Pi}} := \supp({\mathfrak{q}}_\ell, {\Lambda}_\ell)$\\
    $\underline{\vec{d}}:=\mathfrak{q}_\ell \vec{c}|_{\Lambda_{\ell}}$

    $[\underline{\vec{e}},\,\underline{\vec{f}}]:=\mathtt{evalupp}(A)(\ell+1, \breve{\underline{\Pi}}, \breve{\Lambda}_{\ell+1\uparrow}, \underline{\Pi}, \Lambda_{\ell+1\uparrow},\underline{\vec{d}}, \vec{c}|_{\Lambda_{\ell+1\uparrow}})$\\
    $\vec{e}=
    \left[\begin{array}{@{}l@{}} \vec{e}|_{\breve{\Pi}_A}\\ \vec{e}|_{\breve{\Pi}_B}\end{array}\right]
    :=
    \left[\begin{array}{@{}l@{}} (A\Phi|_{\Pi})(\breve{\Phi}|_{\breve{\Pi}_A})\vec{d}  \\
    (A\Phi|_{\Pi})(\breve{\Phi}|_{\breve{\Pi}_B})\vec{d}+(\breve{\mathfrak{p}}_{\ell} ^\top \underline{\vec{e}})|_{\breve{\Pi}_B}
    \end{array}\right]$\\
    $\vec{f}=
    \left[\begin{array}{@{}l@{}} \vec{f}|_{\breve{\Lambda}_\ell} \\ \vec{f}|_{\breve{\Lambda}_{\ell+1\uparrow}}\end{array}\right]
    :=
    \left[\begin{array}{@{}l@{}} (\mathfrak{\breve{q}}_\ell^\top \underline{\vec{e}})|_{\breve{\Lambda}_{\ell}} \\ \underline{\vec{f}} \end{array}\right]$\\
}
\caption{Function $\mathtt{evalupp}(A)$.}
\end{algorithm}

\begin{remark} Let $\breve{\Lambda} \subset \breve{\vee}$, $\Lambda \subset \vee$ be  trees, and $\vec{c} \in \ell_2(\Lambda)$, then
\[
[\vec{e},\,\vec{f}]:=\mathtt{evalupp}(A)(1,\breve{\Lambda}_{0},\breve{\Lambda}_{1 \uparrow},\Lambda_0,\Lambda_{1 \uparrow},\vec{c}|_{\Lambda_0},\vec{c}|_{\Lambda_{1 \uparrow}}),
\]
satisfies
\[
{\bf U}|_{\breve{\Lambda} \times \Lambda} \vec{c}
=
\left[\begin{array}{@{}c@{}} \vec{e} \\ \vec{f}
\end{array}\right].
\]
\end{remark}

\begin{restatable}{theorem}{thmevalupp}\label{thm:evalupp}
A call of \texttt{evalupp} yields the output as specified, at the cost of
${\mathcal O}(\# \breve{\Pi}+\# \breve{\Lambda}+\# \Pi+\# \Lambda)$ operations.
\end{restatable}
\begin{proof}
See Appendix~\ref{app:proofs}.
\end{proof}

\begin{algorithm}[ht]\label{alg:evallow}
\KwData{$\ell \in \N$, $\Pi \subset \Delta_{\ell-1}$,
    $\ell$-trees $\breve{\Lambda} \subset \breve{\vee}_{\ell\uparrow}$ and
    $\Lambda \subset {\vee}_{\ell\uparrow}$, $\vec{d}\in \R^{\# \Pi}$, $\vec{c}\in  \R^{\#  \Lambda}$.
}
\KwResult{$\vec{f}=(A\Phi|_{\Pi})(\breve{\Psi}|_{\breve{\Lambda}})\vec{d}+{\bf L}|_{\breve{\Lambda} \times \Lambda} \vec{c}$.
}
\If{$\breve{\Pi}\cup \breve{\Lambda}\neq \emptyset$}{
    $\Pi_B:=\{\lambda \in \Pi : \big|\supp \phi_\lambda \cap \cup_{\mu \in \breve{\Lambda}_\ell} \breve{S}(\mu) \big|>0\}$,\\
    $\underline{\Pi} := \supp(\mathfrak{p}_\ell, \Pi_B) \cup \supp(\mathfrak{q}_\ell, \Lambda_\ell)$\\
    $\underline{\Pi}_B := \supp(\mathfrak p_\ell, \Pi_B)$\\
    $\underline{\breve{\Pi}} := \supp(\breve{\mathfrak{q}}_\ell, \breve{\Lambda}_\ell)$\\

    $\underline{\vec{d}}:=\mathfrak{p}_{\ell} \vec{d}|_{\Pi_B}+\mathfrak{q}_\ell \vec{c}|_{\Lambda_{\ell}}$\\
    $\underline{\vec{e}}:=(A\Phi|_{\underline{\Pi}_B})(\breve{\Phi}|_{\underline{\breve{\Pi}}})\mathfrak{p}_{\ell} \vec{d}|_{\Pi_B}$\\
    $\vec{f}=
    \left[\begin{array}{@{}l@{}} \vec{f}|_{\breve{\Lambda}_\ell} \\ \vec{f}|_{\breve{\Lambda}_{\ell+1\uparrow}}\end{array}\right]
    :=
    \left[\begin{array}{@{}l@{}} (\mathfrak{\breve{q}}_\ell^\top \underline{\vec{e}})|_{\breve{\Lambda}_{\ell}} \\
    \mathtt{evallow}(A)(\ell+1, \breve{\Lambda}_{\ell+1\uparrow}, \underline{\Pi}, \Lambda_{\ell+1\uparrow},\underline{\vec{d}}, \vec{c}|_{\Lambda_{\ell+1\uparrow}}) \end{array}\right]$\\
}
\caption{Function $\mathtt{evallow}(A)$.}
\end{algorithm}

\begin{remark} \label{rem3}
Let $\breve{\Lambda} \subset \breve{\vee}$, $\Lambda \subset \vee$ be trees, and $\vec{c} \in \ell_2(\Lambda)$, then
\[
{\bf L}|_{\breve{\Lambda} \times \Lambda} \vec{c}=
\mathtt{evallow}(A)(1,\breve{\Lambda}_{1 \uparrow},\Lambda_0,\Lambda_{1 \uparrow},\vec{c}|_{\Lambda_0},\vec{c}|_{\Lambda_{1 \uparrow}}).
\]
\end{remark}

\begin{restatable}{theorem}{thmevallow}\label{thm:evallow}
A call of \texttt{evallow} yields the output as specified, at the cost of
${\mathcal O}(\# \breve{\Lambda}+\# \Pi+\# \Lambda)$ operations.
\end{restatable}
\begin{proof}
See Appendix~\ref{app:proofs}.
\end{proof}

\subsection{Application of tensor-product operators w.r.t.~double-trees}
\label{Sappl-of-tensors}
For $i \in \{0,1\}$, let
$A_i\colon\Span \Phi_i \to \Span \breve{\Phi}_i'$ be local and linear
and let
\[
{\bf A}_i=(A \Psi_i)(\breve{\Psi}_i)=[(A \psi^i_\mu)(\breve{\psi}^i_\lambda)]_{\lambda \in \breve{\vee}^i, \mu \in \vee^i} = {\bf L}_i+{\bf U}_i.
\]
where ${\bf U}_i:=[({\bf A}_i)_{\lambda,\mu}]_{|\lambda| \leq |\mu|}$ and ${\bf L}_i:=[({\bf A}_i)_{\lambda,\mu}]_{|\lambda| > |\mu|}$.
For $i \in \{0,1\}$, let $\neg i := 1-i$.

\begin{definition}[Double-tree]\label{def:dbltree}
Define the coordinate projector $P_i (b_0,b_1):=b_i$.
We call $\bm{\Lambda} \subset \{\breve{\vee}^{0} \times \breve{\vee}^1,\vee^0 \times \breve{\vee}^1, \breve{\vee}^0 \times \vee^1, \vee^0 \times \vee^1\}$,
a {\em double-tree} when for $i \in \{0,1\}$ and any $\mu\in P_{\neg i}\bm{\Lambda}$, the \emph{fiber}
\[
\bm{\Lambda}_{i,\mu}:=P_i(P_{\lnot i}|_{\bm{\Lambda}})^{-1} \{\mu\}
\]
is a tree (in $\breve{\vee}^i $ or $\vee^i$), i.e., $\bm{\Lambda}$ is a double-tree when `frozen' in each of its coordinates, at any value of that coordinate, it is a tree in the remaining coordinate.
\end{definition}
From $\bm{\Lambda}=\cup_{\mu\in P_{\lnot i}\bm{\Lambda}} (P_{\lnot i}|_{\bm{\Lambda}})^{-1} \{\mu\}$, we have
$P_i \bm{\Lambda}=\cup_{\mu\in P_{\lnot i}\bm{\Lambda}} \bm{\Lambda}_{i,\mu}$, which, being a union of trees, is a tree itself. See also Figure~\ref{fig:doubletree}.
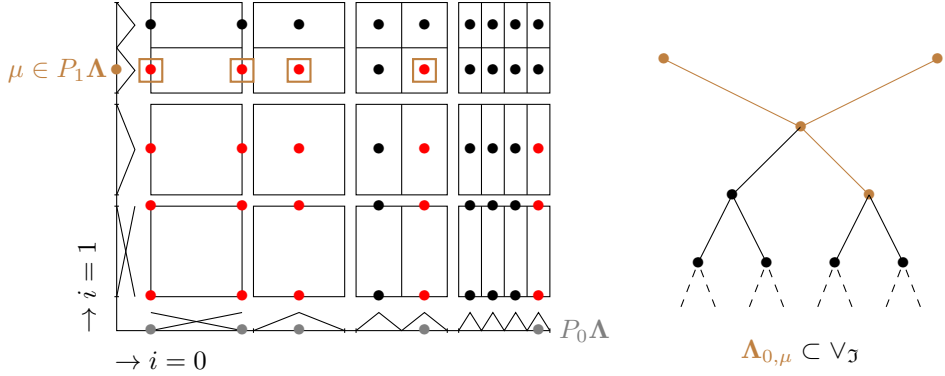
\begin{figure}[ht]
\centering
\begin{tikzpicture}[scale=0.3]
    \node[below right] at (-1.5,-1.5) {$\rightarrow i=0$};
    \node[above right, rotate=90] at (-1.5,-1.5) {$\rightarrow i=1$};
    \draw (-1,13.5) -- (-1,-1) -- (18.0,-1);
    \begin{scope}[shift={(0,-1)}]
    \foreach \xpos in {0.5, 4.5, 5.0, 9.0, 9.5, 13.5, 14.0, 18.0} {
        \draw (\xpos, -.1) -- (\xpos, .1);
    }
    \draw (0.5,0) -- (4.5, 0.8);
    \draw (0.5,0.8) -- (4.5, 0);
    \draw (5,0) -- (7,0.8) -- (9,0);
    \draw (9.5,0) -- (10.5,0.8) -- (11.5,0) -- (12.5,0.8) -- (13.5,0);
    \draw (14.0,0) -- (14.5,0.8) -- (15.0,0) -- (15.5,0.8) -- (16.0,0) -- (16.5,0.8) -- (17.0,0) -- (17.5,0.8) -- (18.0,0);
    \foreach \xpos in {0.5,4.5,7.0,12.5,17.5} {
        \node[gray] at (\xpos,0){\textbullet};
    }
    \node[gray, right] at (18.0,0){$P_0 {\bf \Lambda}$};
    \end{scope}
    \begin{scope}[shift={(-1,0)}]
    \foreach \xpos in {0.5, 4.5, 5.0, 9.0, 9.5, 13.5} {
        \draw (-.1,\xpos) -- (.1,\xpos);
    }
    \draw (0,0.5) -- (0.8,4.5);
    \draw (0.8,0.5) -- (0,4.5);
    \draw (0,5.0) -- (0.8,7.0) -- (0,9.0);
    \draw (0,9.5) -- (0.8,10.5) -- (0,11.5) -- (0.8,12.5) -- (0,13.5);
    \end{scope}
    \foreach \row in {0.5, 5.0, 9.5, 14.0} {
        \foreach \col in {0.5, 5.0, 9.5} {
            \draw (\row, \col) rectangle (\row + 4,\col + 4);
        }
    }
    \foreach \xpos in {11.5, 15.0, 16.0, 17.0} {
        \foreach \ypos in {0.5, 5.0, 9.5} {
            \draw (\xpos, \ypos) -- (\xpos,\ypos+4);
        }
    }
    \foreach \xpos in {11.5} {
        \foreach \ypos in {0.5, 5.0, 9.5, 14.0} {
            \draw (\ypos, \xpos) -- (\ypos+4,\xpos);
        }
    }
    \foreach \xy in {(0.5,0.5), (0.5,4.5), (0.5,7), (0.5,10.5),
    (4.5,0.5), (4.5,4.5), (4.5,7), (4.5,10.5),
    (7,0.5), (7,4.5), (7,7), (7,10.5),
    (12.5,0.5), (12.5,4.5), (12.5,7), (12.5,10.5),
    (17.5,0.5), (17.5,4.5), (17.5,7)} {
        \node[red] at \xy{\textbullet};
    }
    \foreach \xy in {(0.5,12.5),
    (4.5,12.5),
    (7,12.5),
    (10.5, 0.5),(10.5,4.5), (10.5,7), (10.5,10.5), (10.5,12.5),
    (12.5,12.5),
    (14.5,0.5), (14.5,4.5), (14.5,7), (14.5,10.5), (14.5,12.5),
    (15.5,0.5), (15.5,4.5), (15.5,7), (15.5,10.5), (15.5,12.5),
    (16.5,0.5), (16.5,4.5), (16.5,7), (16.5,10.5), (16.5,12.5),
    (17.5,10.5), (17.5,12.5)} {
        \node at \xy{\textbullet};
    }
    \node[brown, left] at (-1.0, 10.5){$\mu \in P_1{\bf \Lambda}$};
    \node[brown] at (-1,10.5){\textbullet};
    \foreach \x in {0.5, 4.5, 7.0, 12.5} {
        \draw[brown,thick] (\x-0.5,10.0) rectangle (\x+.5,11.0);
    }
    
    \begin{scope}[shift={(23,8.0)}, scale=0.9/0.3]
    \node[brown] at (0,1){\textbullet};
    \node[brown] at (4,1){\textbullet};
    
    \draw[brown] (0,1) -- (2,0) -- (4,1);
    \node[brown] at (2,0){\textbullet};
    \draw (2,0) -- (1,-1);
    \node at (1,-1){\textbullet};
    \draw[brown] (2,0) -- (3,-1);
    \node[brown] at (3,-1){\textbullet};
    
    \draw (1,-1) -- (0.5,-2);
    \node at (0.5,-2){\textbullet};
    \draw (1,-1) -- (1.5,-2);
    \node at (1.5,-2){\textbullet};
    \draw (3,-1) -- (2.5,-2);
    \node at (2.5,-2){\textbullet};
    \draw (3,-1) -- (3.5,-2);
    \node at (3.5,-2){\textbullet};
    
    \draw[dashed] (3.75,-2.65) -- (3.5,-2);
    \draw[dashed] (2.75,-2.65) -- (2.5,-2);
    \draw[dashed] (1.75,-2.65) -- (1.5,-2);
    \draw[dashed] (0.75,-2.65) -- (0.5,-2);
    \draw[dashed] (3.25,-2.65) -- (3.5,-2);
    \draw[dashed] (2.25,-2.65) -- (2.5,-2);
    \draw[dashed] (1.25,-2.65) -- (1.5,-2);
    \draw[dashed] (0.25,-2.65) -- (0.5,-2);
    \node[below] at (2,-3){${\color{brown} {\bf \Lambda}_{0, \mu}} \subset \vee_{\mathfrak I}$};
    \end{scope}
\end{tikzpicture}
\vspace{-1em}
\caption{With $\vee_{\mathfrak I}$ from Figure~\ref{fig:time-hbf}: $\vee_{\mathfrak I} \times \vee_{\mathfrak I}$ in black; a double-tree ${\bf \Lambda} \subset \vee_{\mathfrak I} \times \vee_{\mathfrak I}$ in red; the projection $P_0{\bf \Lambda}$ in gray, and a fiber ${\bf \Lambda}_{0, \mu}$ for $\mu \in P_1{\bf \Lambda}$ in brown.}
\label{fig:doubletree}
\end{figure}

For a subset $\lhd$ of a (double) index set $\Diamond$, let $I_\lhd^\Diamond$ denote the extension operator with zeros of a vector supported on $\lhd$ to one on $\Diamond$, and let $R_\lhd^\Diamond$ denotes its (formal) adjoint, being the restriction operator of a vector supported on $\Diamond$ to one on $\lhd$. Since the set $\Diamond$ will always be clear from the context, we will denote these operators simply by $I_\lhd$ and $R_\lhd$.%

As sketched in the introduction of this section, the pieces are now in place to apply $R_{\breve {\bf \Lambda}} ({\bf A}_0 \otimes {\bf A}_1) I_{\bf \Lambda}$ in linear complexity.

\begin{restatable}{theorem}{thmevaldbl} \label{thm1} Let $\breve{\bm{\Lambda}} \subset \breve{\vee}^0 \times \breve{\vee}^1$, $\bm{\Lambda} \subset \vee^0 \times \vee^1$ be finite double-trees. Then
\begin{align*}
\bm{\Sigma}&:=\bigcup_{\lambda \in P_0 \bm{\Lambda}} \Big(\{\lambda\} \times
\bigcup_{\big\{\mu \in P_0 \breve{\bm{\Lambda}} : |\mu|=|\lambda|+1,\,|\breve{S}^0(\mu) \cap S^0(\lambda)|>0\big\}} \breve{\bm{\Lambda}}_{1,\mu}
\Big),\\
\bm{\Theta}&:=\bigcup_{\lambda \in P_{1} \bm{\Lambda}} \Big(
\{\mu \in  P_{0} \breve{\bm{\Lambda}} : \exists \gamma \in \bm{\Lambda}_{0,\lambda} \text{ s.t. } |\gamma|= |\mu|,\,|\breve{S}^0(\mu) \cap S^0(\gamma)|>0\}
\times \{\lambda\}\Big),
\end{align*}
are double-trees with $\# \bm{\Sigma} \lesssim \# \breve{\bm{\Lambda}}$ and $\# \bm{\Theta} \lesssim \# \bm{\Lambda}$, and
\begin{equation} \label{splitting}
\begin{split}
R_{\breve{\bm{\Lambda}}} ({\bf A}_0 \otimes {\bf A}_1) I_{\bm{\Lambda}}=
 & R_{\breve{\bm{\Lambda}}} ({\bf L}_0 \otimes \mathrm{Id}) I_{\bm{\Sigma}}  R_{\bm{\Sigma}} (\mathrm{Id} \otimes {\bf A}_1) I_{\bm{\Lambda}}+ \\
&R_{\breve{\bm{\Lambda}}} (\mathrm{Id} \otimes {\bf A}_1) I_{\bm{\Theta}}  R_{\bm{\Theta}} ({\bf U}_0 \otimes \mathrm{Id}) I_{\bm{\Lambda}}.
\end{split}
\end{equation}
\end{restatable}
\begin{proof}
See Appendix~\ref{app:proofs}.
\end{proof}

The application of $R_{\breve{\bm{\Lambda}}} ({\bf L}_0 \otimes \mathrm{Id}) I_{\bm{\Sigma}}$ boils down to the application of
$R_{\breve{\bm{\Lambda}}_{0,\mu}} {\bf L}_0 I_{\bm{\Sigma}_{0,\mu}}$ for every $\mu \in P_1 \bm{\Sigma}\cap P_1\breve{\bm{\Lambda}}$.
Such an application can be performed in ${\mathcal O}(\#\breve{\bm{\Lambda}}_{0,\mu}+\#\bm{\Sigma}_{0,\mu})$ operations by means of a call of $\mathtt{evallow}({A_0})$; see also Algorithm~\ref{alg:tensor}.
Since
$\sum_{\mu \in \breve{\vee}_1}\#\breve{\bm{\Lambda}}_{0,\mu} + \#\bm{\Sigma}_{0,\mu} = \# \breve{\bm{\Lambda}}+ \# \bm{\Sigma}$,
we conclude that the application of $R_{\breve{\bm{\Lambda}}} ({\bf L}_0 \otimes \mathrm{Id}) I_{\bm{\Sigma}}$ can be performed in ${\mathcal O}(\#\breve{\bm{\Lambda}}+\#\bm{\Sigma})$ operations.

Similarly, applications of
$R_{\bm{\Sigma}} (\mathrm{Id} \otimes {\bf A}_1) I_{\bm{\Lambda}}$,
$R_{\breve{\bm{\Lambda}}} (\mathrm{Id} \otimes {\bf A}_1) I_{\bm{\Theta}}$, and
$R_{\bm{\Theta}} ({\bf U}_0 \otimes \mathrm{Id}) I_{\bm{\Lambda}}$
using calls of $\mathtt{eval}({A_1})$, $\mathtt{eval}({A_1})$, and $\mathtt{evalupp}({A_0})$ respectively,
can be done in ${\mathcal O}(\# \bm{\Sigma}+ \# \bm{\Lambda})$, ${\mathcal O}(\#\breve{\bm{\Lambda}}+ \#\bm{\Theta})$, and ${\mathcal O}(\#\bm{\Theta}+ \#\bm{\Lambda})$ operations. From $\# \bm{\Sigma} \lesssim \# \breve{\bm{\Lambda}}$ and $\# \bm{\Theta} \lesssim \# \bm{\Lambda}$ we conclude the following.

\begin{corol}\label{cor:spacetime-linear} Let $\breve{\bm{\Lambda}} \subset \breve{\vee}^0 \times \breve{\vee}^1$, $\bm{\Lambda} \subset \vee^0 \times \vee^1$ be finite double-trees, then $R_{\breve{\bm{\Lambda}}} ({\bf A}_0 \otimes {\bf A}_1) I_{\bm{\Lambda}}$ can be applied in ${\mathcal O}(\# \breve{\bm{\Lambda}}+\#\bm{\Lambda})$ operations.
\end{corol}

\section{The heat equation and practical realization}\label{sec:heat}
In this section, we consider the numerical approximation of the \emph{heat equation
\be \label{eqn:heat}
\left\{
\begin{array}{rl}
\frac{\D u}{\D t}(t) -(\Delta_{\vec x} u)(t)& = g(t) \quad(t \in I),\\
u(0) & = u_0.
\end{array}
\right.
\ee}
For some bounded domain $\Omega \subset \R^2$, we take $H := L_2(\Omega)$ and $V := H_0^1(\Omega)$, so that $X = L_2(I; H_0^1(\Omega)) \cap H^1(I; H^{-1}(\Omega))$ and $Y = L_2(I; H_0^1(\Omega))$.
We define
\[
a(t; \eta, \zeta) := \int_\Omega \grad \eta \cdot \grad \zeta \dif {\bf x},
\]
and aim to solve the parabolic initial value problem~\eqref{11} numerically.
The bilinear forms present in our variational formulation~\eqref{m1} satisfy
\[
A = M_t \otimes A_{\vec x}, \quad B = D_t \otimes M_{\vec x} + A, \quad \text{and} \quad \gamma_0' \gamma_0 = G_t \otimes M_{\vec x}
\]
where
\be \label{eqn:bilforms}
\begin{aligned}
(M_t v)(w) &:= \int_I v w \dif t,\quad
(D_t v)(w) := \int_I v' w \dif t, \quad
(G_t v)(w) := v(0) w(0), \\
(A_{\bf x} \eta)(\zeta) &:= \int_\Omega \nabla \eta \cdot \nabla \zeta \dif \vec x, \quad
(M_{\bf x} \eta)(\zeta) := \int_\Omega \eta \zeta \dif \vec x.
\end{aligned}
\ee

In this section, we
first construct suitable tensor-product bases for $X$ and $Y$ which functions are wavelets in time and hierarchical finite element functions in space. We then build our discrete `trial' and `test' spaces $(X^\delta, Y^\delta)_{\delta \in \Delta}$ as the span of subsets of these tensor-product bases.
\newer{We finish with concrete uniformly optimal preconditioners $K_X^\delta$ and $K_Y^\udelta$,
the basis necessary for error estimation in the adaptive loop, and evaluation of the right-hand side of~\eqref{eqn:discr-schur} using interpolants.}

\subsection{Wavelets in time}\label{sec:wit}
We construct piecewise linear wavelet bases $\Sigma$ for $H^1(I)$ and $\Xi$ for $L_2(I)$.

\subsubsection{Basis on the trial side}
For $\Sigma$, we choose the three-point wavelet basis from~\cite{Stevenson1998}; for completeness, we include its construction.
For $\ell \geq 0$, 
define the scaling functions as the nodal continuous piecewise linears w.r.t.~a uniform partition into $2^\ell$ subintervals, that is
$\Phi^\Sigma_\ell := \{\phi_{(\ell,n)} : 0 \leq n \leq 2^\ell\}$ with
$\phi_{(\ell,n)}(k2^{-\ell}) = \delta_{kn}$ for $0 \leq k \leq 2^\ell$.
Define $\Sigma_0 := \Phi^\Sigma_0$, and for $\ell \geq 1$, define
$\Sigma_\ell := \{\sigma_\lambda: \lambda := (\ell, n)~~\text{with}~~ 0 \leq n < 2^{\ell-1}\}$ with $\sigma_\lambda = \sigma_{(\ell,n)}$
as in the right of Figure~\ref{fig:3pt-tree}.
Note that each $\sigma_\lambda$
is a linear combination of three nodal functions from $\Phi^\Sigma_\ell$, hence the name \emph{three-point wavelet}.

By imposing the parent-child structure
\be \label{eqn:parent-child}
\tilde \lambda \triangleleft_\Sigma \lambda \iff  |\tilde \lambda| + 1 = |\lambda| ~~ \text{and} ~~ |\supp \sigma_\lambda \cap \supp \sigma_{\tilde \lambda}| > 0,
\ee
on any two indices $\tilde \lambda, \lambda$,
we get
the tree shown left in Figure~\ref{fig:3pt-tree}.

Define $\Sigma := \cup_{\ell \geq 0} \Sigma_\ell$, $\vee_\Sigma := \{\lambda : \sigma_\lambda \in \Sigma\}$, and $S(\sigma_\lambda) := \supp \sigma_\lambda$. We see that $\Sigma$
satisfies~\eqref{c1}--\eqref{c2} and that the $\Phi^\Sigma_\ell$
satisfy~\eqref{c4}--\eqref{c7}. Moreover, one can show that $\Sigma$ is a Riesz basis for
$L_2(I)$ (cf.~\cite[Thm.~4.2]{Stevenson1998}), and that $\{2^{-|\lambda|} \sigma_\lambda\}$ is
a Riesz basis for $H^1(I)$ (cf.~\cite[Thm.~4.3]{Stevenson1998}).

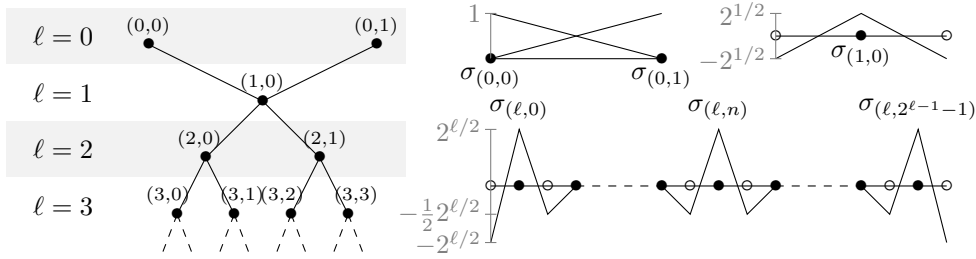
\begin{figure}[ht]
\centering
\begin{tikzpicture}[scale=0.75]
    \begin{scope}[shift={(0,0)}]
    \fill[black!5] (-2.5, 0.65) rectangle (4.5, 1.65); \node at (-1.5, 1.15){$\ell=0$};
    \fill[black!0] (-2.5, -0.35) rectangle (4.5, 0.65); \node at (-1.5, 0.15){$\ell=1$};
    \fill[black!5] (-2.5, -1.35) rectangle (4.5, -0.35); \node at (-1.5, -0.85){$\ell=2$};
    \fill[black!0] (-2.5, -2.35) rectangle (4.5, -1.35); \node at (-1.5, -1.85){$\ell=3$};
    
    \node at (0,1){\textbullet}; \node[above] at (0,1){$_{(0,0)}$};
    \node at (4,1){\textbullet}; \node[above] at (4,1){$_{(0,1)}$};
    
    \draw (0,1) -- (2,0) -- (4,1);
    \node at (2,0){\textbullet}; \node[above] at (2,0){$_{(1,0)}$}; 
    \draw (2,0) -- (1,-1);
    \node at (1,-1){\textbullet}; \node[above] at (0.9,-1){$_{(2,0)}$};
    \draw (2,0) -- (3,-1);
    \node at (3,-1){\textbullet}; \node[above] at (3.1,-1){$_{(2,1)}$};
    
    \draw (1,-1) -- (0.5,-2);
    \node at (0.5,-2){\textbullet}; \node[above] at (0.35,-2){$_{(3,0)}$};
    \draw (1,-1) -- (1.5,-2);
    \node at (1.5,-2){\textbullet}; \node[above] at (1.65,-2){$_{(3,1)}$};
    \draw (3,-1) -- (2.5,-2);
    \node at (2.5,-2){\textbullet}; \node[above] at (2.35,-2){$_{(3,2)}$};
    \draw (3,-1) -- (3.5,-2);
    \node at (3.5,-2){\textbullet}; \node[above] at (3.65,-2){$_{(3,3)}$};
    
    \draw[dashed] (3.75,-2.66) -- (3.5,-2);
    \draw[dashed] (2.75,-2.66) -- (2.5,-2);
    \draw[dashed] (1.75,-2.66) -- (1.5,-2);
    \draw[dashed] (0.75,-2.66) -- (0.5,-2);
    \draw[dashed] (3.25,-2.66) -- (3.5,-2);
    \draw[dashed] (2.25,-2.66) -- (2.5,-2);
    \draw[dashed] (1.25,-2.66) -- (1.5,-2);
    \draw[dashed] (0.25,-2.66) -- (0.5,-2);
    \end{scope}
    
    \begin{scope}[shift={(6.0,0.75)}]
    \begin{scope}[shift={(0,0)}]
        \draw (0,0) -- (3,0);
        \node at (0,0){\textbullet}; \node at (3,0){\textbullet};
        \draw (0,.8) -- (3,0); \node[below] at (0,0){$\sigma_{(0,0)}$};
        \node[gray, left] at (0,.8){$1$};
        \draw[gray] (0,0) -- (0,.8);
        \draw[gray] (-0.1,.8) -- (0.1,.8);
        \draw (0,0) -- (3,.8); \node[below] at (3,0){$\sigma_{(0,1)}$};
    \end{scope}
    \begin{scope}[shift={(5,0.4)}]
        \draw (0,0) -- (3,0);
        \node at (0,0){\textopenbullet}; \node at (3,0){\textopenbullet}; \node at (1.5,0){\textbullet};
        \draw (0,-0.4) -- (1.5,0.4) -- (3,-0.4); \node[below] at (1.5,0){ $\sigma_{(1,0)}$};
        \node[gray, left] at (0,0.4){$2^{1/2}$};
        \node[gray, left] at (0,-0.4){$-2^{1/2}$};
        \draw[gray] (0,-0.4) -- (0,0.4);
        \draw[gray] (-0.1,0.4) -- (0.1,0.4);
        \draw[gray] (-0.1,-0.4) -- (0.1,-0.4);
    \end{scope}
    \begin{scope}[shift={(0,-2.25)}]
        \node at (0,0){\textopenbullet}; \node at (1,0){\textopenbullet};
        \node at (0.5,0){\textbullet}; \node at (1.5,0){\textbullet};
        \draw (0,-1) -- (0.5,1) node[above]{$\sigma_{(\ell,0)}$} -- (1,-0.5) -- (1.5,0);
        \node[gray, left] at (0,1){$2^{\ell/2}$};
        \node[gray, left] at (0,-1){$-2^{\ell/2}$};
        \node[gray, left] at (0,-0.5){$-\tfrac{1}{2}2^{\ell/2}$};
        \draw[gray] (0,-1) -- (0,1);
        \draw[gray] (-0.1,1) -- (0.1,1);
        \draw[gray] (-0.1,-1) -- (0.1,-1);
        \draw[gray] (-0.1,-0.5) -- (0.1,-0.5);
        
        \draw (0,0) -- (1.5,0);
        \draw[dashed] (1.5,0) -- (3.0,0);
        \draw (3.0,0) -- (5.0,0);
        
        \node at (3.0,0){\textbullet}; \node at (3.5,0){\textopenbullet};
        \node at (4.0,0){\textbullet}; \node at (4.5,0){\textopenbullet};
        \node at (5.0,0){\textbullet};
        \draw (3.0, 0) -- (3.5,-0.5) -- (4.0, 1.0) -- (4.5, -0.5) -- (5.0, 0);
        \node[above] at (4.0, 1.0){$\sigma_{(\ell, n)}$};
        
        \draw[dashed] (5.0, 0) -- (6.5, 0);
        \draw (6.5, 0) -- (8,0);
        
        \node at (8,0){\textopenbullet}; \node at (7,0){\textopenbullet};
        \node at (7.5,0){\textbullet}; \node at (6.5,0){\textbullet};
        \draw (8,-1) -- (7.5,1) node[above]{$\sigma_{(\ell,2^{\ell-1}-1)}$} -- (7,-0.5) -- (6.5,0);
    \end{scope}
    \end{scope}
\end{tikzpicture}
\vspace{-2em}
\caption{Left: three-point wavelet index set $\vee_\Sigma$ with parent-child relations; right: three-point wavelets.}
\label{fig:3pt-tree}
\end{figure}

\subsubsection{Basis on the test side}
We construct an $L_2(I)$-orthonormal basis $\Xi$.

For $\ell \geq 0$, define  the (discontinuous) piecewise linear scaling functions w.r.t.~a uniform partition into $2^\ell$ subintervals by $\Phi_\ell^\Xi := \{ \phi_{(\ell, n)}\colon 0 \leq n < 2^{\ell+1} \}$ where $\phi_{(0, 0)}(t) := \bbone_{[0,1]}(t)$ and $\phi_{(0, 1)}(t) := \sqrt{3}(2t - 1) \bbone_{[0,1]}$, and for $\ell \geq 1$, $\phi_{(\ell, 2k)}(t) := \phi_{(0, 0)}(2^\ell t - k)$ and $\phi_{(\ell, 2k+1)}(t) := \phi_{(0,1)}(2^\ell t - k)$. Let $\Xi_0 := \Phi_0^\Xi$, and define $\Xi_1 := \{\xi_{(1,0)}, \xi_{(1,1)}\}$ as in the right of Figure~\ref{fig:ortho-tree}. For $\ell \geq 2$, we take $\Xi_{\ell} := \{\xi_{(\ell, n)} : 0 \leq n < 2^\ell\}$ with
\[
\xi_{(\ell, 2k)}(t) := 2^{(\ell-1)/2} \xi_{(1, 0)}(2^{\ell-1}t - k), \quad  \xi_{(\ell, 2k+1)}(t) :=  2^{(\ell-1)/2} \xi_{(1, 1)}(2^{\ell-1}t - k).
\]

The resulting $\Xi := \cup_{\ell \geq 0} \Xi_\ell$ is an
\emph{orthonormal} basis for $L_2(I)$, and together with its scaling functions $\cup_\ell \Phi^\Xi_\ell$, the
conditions from \S\ref{sec:eval_time} are satisfied with $S(\xi_\mu) := \supp \xi_\mu$. We impose a parent-child relation analogously to~\eqref{eqn:parent-child}; see the left of Figure~\ref{fig:ortho-tree}.

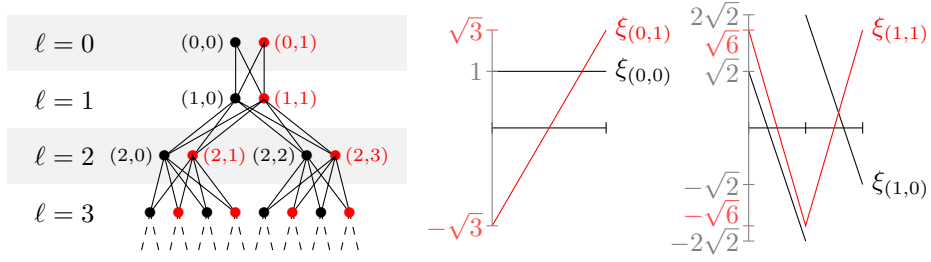
\begin{figure}[ht]
\centering
\begin{tikzpicture}[scale=0.75]
    \fill[black!5] (-2.5, 0.5) rectangle (4.5, 1.5); \node at (-1.5, 1){$\ell=0$};
    \fill[black!0] (-2.5, -0.5) rectangle (4.5, 0.5); \node at (-1.5, 0){$\ell=1$};
    \fill[black!5] (-2.5, -1.5) rectangle (4.5, -0.5); \node at (-1.5, -1){$\ell=2$};
    \fill[black!0] (-2.5, -2.5) rectangle (4.5, -1.5); \node at (-1.5, -2){$\ell=3$};
    \node at (1.5,1){\textbullet}; \node[left] at (1.5,1){$_{(0,0)}$};
    \node[red] at (2,1){\textbullet}; \node[red, right] at (2,1){$_{(0,1)}$};
    
    \draw (2,0) -- (1.5,1) -- (1.5,0);
    \draw (2,0) -- (2,1) -- (1.5,0);
    
    \node at (1.5,0){\textbullet}; \node[left] at (1.5,0){$_{(1,0)}$};
    \node[red] at (2,0){\textbullet}; \node[red, right] at (2,0){$_{(1,1)}$};
    
    \draw (0.75, -1) -- (1.5,0) -- (0.25,-1);
    \draw (2.75, -1) -- (1.5,0) -- (3.25,-1);
    \draw (0.75, -1) -- (2,0) -- (0.25,-1);
    \draw (2.75, -1) -- (2,0) -- (3.25,-1);
    
    \node at (0.25,-1){\textbullet}; \node[left] at (0.25,-1){$_{(2,0)}$};
    \node[red] at (0.75,-1){\textbullet}; \node[red, right] at (0.75,-1){$_{(2,1)}$};
    \node at (2.75,-1){\textbullet}; \node[left] at (2.75,-1){$_{(2,2)}$};
    \node[red] at (3.25,-1){\textbullet}; \node[red, right] at (3.25,-1){$_{(2,3)}$};
    
    \draw (0.25,-1) -- (0,-2) -- (0.75,-1) -- (0.5,-2) -- (0.25,-1);
    \draw (0.25,-1) -- (1,-2) -- (0.75,-1) -- (1.5,-2) -- (0.25,-1);
    
    \draw (2.75,-1) -- (2,-2) -- (3.25,-1) -- (2.5,-2) -- (2.75,-1);
    \draw (2.75,-1) -- (3,-2) -- (3.25,-1) -- (3.5,-2) -- (2.75,-1);
    
    \foreach \x in {0,0.5,1,1.5,2,2.5,3,3.5} {
    \draw[dashed] (\x - .15,-2.66) -- (\x,-2);
    \draw[dashed] (\x + .15,-2.66) -- (\x,-2);
    }
    \foreach \x in {0,1,2,3} {
    \node at (\x,-2){\textbullet};
    \node[red] at (\x+0.5,-2){\textbullet};
    }
    
    \begin{scope}[shift={(6.0,-.5)}]
    \draw (0,0) -- (2,0);
    \foreach \x in {0,2}
    {
        \draw (\x,-0.1) -- (\x,0.1);
    }
    
    \draw (0,1) -- (2,1); \node[right] at (2,1){$\xi_{(0,0)}$};
    \node[gray, left] at (0,1){$1$};
    \draw[red] (0,-1.73) -- (2,1.73); \node[red, right] at (2,1.73){$\xi_{(0,1)}$};
    \node[red!70, left] at (0,1.73){$\sqrt{3}$};
    \node[red!70, left] at (0,-1.73){$-\sqrt{3}$};
    \draw[gray] (0,-1.73) -- (0,1.73);
    \draw[gray] (-0.1,1) -- (0.1,1);
    \draw[red!70] (-0.1,1.73) -- (0.1,1.73);
    \draw[red!70] (-0.1,-1.73) -- (0.1,-1.73);
    
    \begin{scope}[shift={(4.5,0)}]
        \draw (0,0) -- (2,0);
        \foreach \x in {0,1,2}
        {
            \draw (\x,-0.1) -- (\x,0.1);
        }
        \draw[black] (0,1) -- (1, -2); \draw[black] (1,2) -- (2,-1);
        \node[right] at (2,-1.0){$\xi_{(1,0)}$};
        \draw[red] (0,1.73) -- (1,-1.73) -- (2,1.73);
        \node[right, red] at (2,1.73){$\xi_{(1,1)}$};
        \node[red!70, left] at (0,1.5){$\sqrt{6}$};
        \node[red!70, left] at (0,-1.5){$-\sqrt{6}$};
        
        \node[gray, left] at (0,2){$2\sqrt{2}$};
        \node[gray, left] at (0,-2){$- 2\sqrt{2}$};
        \node[gray, left] at (0,1){$\sqrt{2}$};
        \node[gray, left] at (0,-1){$- \sqrt{2}$};
        
        \draw[gray] (0,-2) -- (0,2);
        \draw[gray] (-0.1,1) -- (0.1,1);
        \draw[gray] (-0.1,2) -- (0.1,2);
        \draw[gray] (-0.1,-1) -- (0.1,-1);
        \draw[gray] (-0.1,-2) -- (0.1,-2);
        \draw[red!70] (-0.1,1.73) -- (0.1,1.73);
        \draw[red!70] (-0.1,-1.73) -- (0.1,-1.73);
    \end{scope}
    \end{scope}
\end{tikzpicture}
\vspace{-2em}
\caption{Left: orthonormal wavelet index set $\vee_\Xi$ with parent-child relations; right: the wavelets at levels 0 and 1.}
\label{fig:ortho-tree}
\end{figure}

\subsection{Finite elements in space}\label{sec:fem}
Let $\bbT$ be the family of all \emph{conforming} partitions of $\Omega$
into triangles that can be created by Newest Vertex
Bisection from some given conforming initial triangulation $\tria_\bot$
with an assignment of newest vertices satisfying the matching condition;
cf.~\cite{Stevenson2008TheBisection}.

Define $\mathfrak T := \cup_{\tria \in \mathbb T} \{T : T \in \tria\}$. For $T \in \mathfrak T$, set
$\gen(T)$ as the number of bisections needed to create $T$ from its
`ancestor' $T' \in \tria_\bot$. With $\mathfrak N$ the set of all
vertices of all $T \in \mathfrak T$, for $\nu \in \mathfrak N$ we set
$\gen(\nu) := \min\{ \gen(T) : \nu \text{ is a vertex of }T \in \mathfrak T\}$.

Any $\nu \in \mathfrak N$ with $\gen(\nu) > 0$ is the midpoint of an edge $e_\nu$ of
one or two $T \in \mathfrak T$ with $\gen(T) = \gen(\nu) - 1$. The set of
newest vertices $\tilde \nu$ of these $T$, so those vertices of $T$ with
$|\tilde \nu| = \gen(\nu) - 1$, are defined as the parents of $\nu$, denoted $\tilde \nu \triangleleft_{\mathfrak N} \nu$.
The set of \emph{godparents} of $\nu$, denoted $\gp(\nu)$,
are defined as the two endpoints of $e_\nu$. Vertices with $\gen(\nu) = 0$ have no parents or godparents.

\begin{example}
In Figure~\ref{fig:vertex-tree}, the parents of $\nu_4$ are $\nu_1$ and $\nu_3$ and its godparents are $\nu_0$, $\nu_2$; the sole parent of $\nu_5$ is $\nu_4$, and its godparents are $\nu_0$ and $\nu_3$.
\end{example}

\begin{prop}[{\cite{Diening2016}}] \label{prop:tree}
An (essentially) non-overlapping partition $\tria$ of $\overline \Omega$
into triangles is in $\mathbb T$ if and only if the set $N_\tria$ of
vertices of all $T \in \mathcal T$ forms a \emph{tree} in the sense of
\S\ref{sec:eval_time}, meaning that it contains every vertex of
generation zero as well as all parents of any $\nu \in N_\tria$;
see also Figure~\ref{fig:vertex-tree}.
\end{prop}

\begin{figure}[ht]
\centering
\begin{tikzpicture}[scale=1.8]
    \foreach \step in {0, 1, 2, 3}
    {
    \begin{scope}[shift={(1.25 * \step, 0)}]
        \draw (1,0) -- (1,1) -- (0,0) -- (1,0);
        \draw (0,1) -- (0,0) -- (1,1) -- (0,1);

        \ifnum \step=0
            \node[below] at (1,0){${\nu}_1$};
            \node[above] at (1,1){${\nu}_2$};
            \node[below] at (0,0){${\nu}_0$};
            \node[above] at (0,1){${\nu}_3$};
            \node at (0,0){\texttimes};
            \node at (1,0){\texttimes};
            \node at (0,1){\texttimes};
            \node at (1,1){\texttimes};
            \node[below] at (0.5,-0.25){$\tria_0$};
        \fi
        \ifnum \step>0
            \draw (1,0) -- (0,1);
        \fi
        \ifnum \step=1
            \node[right] at (0.5, 0.5){${\nu}_4$};
            \node at (0.5,0.5){\textbullet};
            \node[below] at (0.5,-0.25){$\tria_1$};
        \fi
        \ifnum \step>1
            \draw (0,0.5) -- (0.5,0.5);
            \draw (0.5, 1) -- (0.5,0.5);
        \fi
        \ifnum \step=2
            \node[below right] at (0,0.5){${\nu}_5$};
            \node[above] at (0.5,1){${\nu}_6$};
            \node at (0,0.5){\textbullet};
            \node at (0.5,1){\textbullet};
            \node[below] at (0.5,-0.25){$\tria_2$};
        \fi
        \ifnum \step>2
            \draw (0,0.5) -- (0.5,1.0);
        \fi
        \ifnum \step=3
            \node[right] at (0.25,0.75){${\nu}_7$};
            \node at (0.25,0.75){\textbullet};
            \node[below] at (0.5,-0.25){$\tria_3$};
        \fi
    \end{scope}
    }
    \begin{scope}[shift={(-1.5,1)}]
    \foreach \k in {0, 1, 2, 3}
    {
    \node at (\k/3, 0){\texttimes}; \node[above] at (\k/3, 0){${\nu}_\k$};
    }
    \draw (1/3, 0) -- (0.5, -0.33);
    \draw (1, 0) -- (0.5, -0.33);

    \node at (0.5, -0.33){\textbullet}; \node[right] at (0.5, -0.33){${\nu}_4$};
    \node at (0.25 + 0.5, -.66){\textbullet}; \node[right] at (0.25 + 0.5, -.66){${\nu}_6$};
    \draw (0.25 + 0.5, -.66) -- (0.5, -0.33);
    \node at (0.25, -.66){\textbullet}; \node[right] at (0.25, -.66){${\nu}_5$};
    \draw (0.25, -.66) -- (0.5, -0.33);

    \node at (0.5,-1.0){\textbullet}; \node[right] at (0.5,-1.0){${\nu}_7$};
    \draw (0.25, -.66) -- (0.5,-1.0) -- (0.75,-.66);
    \node[below] at (0.5,-1.25){$N_\tria$};
    \end{scope}
\end{tikzpicture}
\caption{Vertex tree $N_{\tria}$ and its triangulation $\tria$ shown level-by-level.}
\label{fig:vertex-tree}
\end{figure}
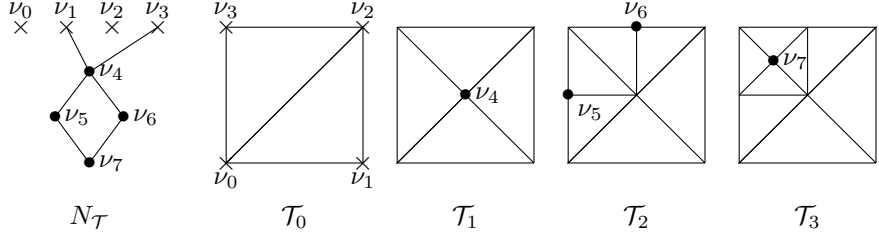

Let $\mathcal O$ be the collection of spaces $W_\tria$ of continuous
piecewise linears w.r.t.~$\tria \in \bbT$ vanishing on $\partial \Omega$.
For $\nu \in \mathfrak N$, we set $\psi_\nu$ as that continuous piecewise linear function on the \emph{uniform partition} $\tria_\nu := \{T \in \mathfrak T : \gen(T) = \gen(\nu)\} \in \bbT$ for which $\psi_\nu(\tilde \nu) = \delta_{\nu \tilde \nu}$ for $\tilde \nu \in \tria_\nu$. Setting $\mathfrak N_0 := \mathfrak N \setminus \partial \Omega$, the collection $\{\psi_\nu : \nu \in \mathfrak N_0\}$ is known as the \emph{hierarchical basis}.
For $\tria \in \bbT$, write $N_{\tria,0} := N_\tria \setminus \partial \Omega$ and $\Psi_\tria := \{ \psi_\nu : \nu \in N_{\tria, 0}\}$; it holds that $W_\tria = \spann \Psi_\tria$.

\subsubsection{Applying stiffness matrices}
The hierarchical basis satisfies conditions~\eqref{c1} and~\eqref{c2}, and so, the application of stiffness matrices $(A \Psi_\tria)(\Psi_\tria)$ for $A \in \{A_{\vec x}, M_{\vec x}\}$
can be done through $\mathtt{eval}(A)$.\footnote{This would require the definition of a suitable single-scale basis.}   
However, the computation in Theorem~\ref{thm1} does not involve the lower and upper parts of $A$. This crucial insight allows for a faster and easier approach using standard finite \newer{element techniques}:
$\spann \Psi_\tria$ is a continuous piecewise linear finite element space, so it has a \emph{canonical} single-scale basis
$\Phi_\tria := \spann\{ \phi_{\tria, \nu} \}$ characterized by $\phi_{\tria, \nu}(\tilde \nu) = \delta_{\nu \tilde \nu}$ for $\tilde \nu \in N_{\tria, 0}$, for which the application of $(A\Phi_\tria)(\Phi_\tria)$ at linear cost using local element matrices is standard. This is different from the general setting in \S\ref{sec:eval_time}, in that $\dim \Phi_\tria = \dim \Psi_\tria$ also for locally refined triangulations.
Let $T$ be the transformation characterized by $\Psi_\tria = T^\top \Phi_\tria$,
we find
\begin{equation}
(A \Psi_{\tria})(\Psi_{\tria}) = T^\top (A \Phi_\tria)(\Phi_{\tria}) T.
\label{eqn:hb2ss}
\end{equation}
We can apply $T$ in linear complexity by iterating over the vertices bottom-up
while applying elementary local transformations in which not parent-child, but \emph{godparent}-child relations play a role.


\subsection{Inf-sup stable family of trial- and test spaces}\label{sec:stable}
With $\Sigma$ and $\Xi$ from \S\ref{sec:wit} and $\Psi_{\mathfrak N_0} := \{\psi_\nu : \nu \in \mathfrak N_0\}$ from~\S\ref{sec:fem}, we find that \newer{$X = \overline{\spann(\Sigma \otimes \Psi_{\mathfrak N_0})}$} and \newer{$Y = \overline{\spann (\Xi \otimes \Psi_{\mathfrak N_0})}$}. We now turn to the construction of $X^\delta$ and $Y^\delta$.

\begin{definition}
For a double-tree ${\bf \Lambda}^\delta  \subset \vee_\Sigma \times \mathfrak N$, define ${\bf \Lambda}^\delta_0 := {\bf \Lambda}^\delta \setminus \vee_\Sigma \times \partial \Omega$. We
construct our `trial' space as
\[
X^\delta := \spann \{ \sigma_\lambda \otimes \psi_\nu : (\lambda, \nu) \in {\bf \Lambda}^\delta_0\}.
\]
Defining the double-tree ${\bf \Lambda}_{Y,0}^\delta \subset \vee_\Xi \times \mathfrak N_0$ as
\[
{\bf \Lambda}_{Y,0}^\delta := \{ (\mu, \nu) : \exists (\lambda, \nu) \in {\bf \Lambda}_0^\delta, \, \mu \in \vee_\Xi, \,  |\mu| = |\lambda|, \,  |\supp \xi_\mu \cap \supp \sigma_\lambda| > 0\},
\]
we construct our `test' space as $Y^\delta = Y^\delta(X^\delta) := \spann \{\xi_\mu \otimes \psi_\nu : (\mu, \nu) \in {\bf \Lambda}_{Y,0}^\delta \}$.
\end{definition}

\begin{theorem}[{\cite[Props.~5.2, 5.3]{Stevenson2020b}}]
Define
$\Delta := \{\delta : {\bf \Lambda}^\delta \subset \vee_\Sigma \times \mathfrak N\text{ is a double-tree}\}$
equipped with the partial ordering $\delta \preceq \tilde \delta \iff {\bf \Lambda}^\delta \subseteq {\bf \Lambda}^{\tilde \delta}$.
With $X^\delta$ and $Y^\delta$ as above, uniform inf-sup stability holds; cf.~\eqref{eqn:stability}.
\end{theorem}

\begin{definition}
Given a double-tree ${\bf \Lambda}^\delta \subset \vee_\Sigma \times \mathfrak N$, we define ${\bf \Lambda}^\udelta \supset {\bf \Lambda}^\delta$ by adding, for $(\lambda, \nu) \in {\bf \Lambda}^\delta$ and any child $\tilde \lambda$ of $\lambda$ and descendant $\tilde \nu$ of $\nu$ up to generation $2$, all pairs $(\tilde \lambda, \nu)$ and $(\lambda, \tilde \nu)$.
We expect this choice of $X^{\udelta}$ to provide saturation; cf.~\eqref{eqn:saturation}.
\end{definition}

\subsection{Preconditioners}\label{sec:preconds}
We follow~\cite[\S5.6]{Stevenson2020b} for the construction of optimal preconditioners $K_Y^\delta $ for ${E_Y^\delta}' A E_Y^\delta$ and $K_X^\delta$ for $S^{\udelta \delta}$ necessary for solving~\eqref{eqn:discr-schur}.
With notation from Definition~\ref{def:dbltree}, we equip $X^\delta$ and $Y^\delta$ with bases
\[
\begin{cases}
\bigcup\limits_{\lambda \in P_0{\bf \Lambda}^\delta_0} \sigma_\lambda \otimes \Psi_\lambda^\delta \quad \text{with} \quad
\Psi_\lambda^\delta := \{\psi_\nu : \nu \in ({\bf \Lambda}^\delta_0)_{1,\lambda}\}, \\
\bigcup\limits_{\mu \in P_0{\bf \Lambda}^\delta_{Y,0}} \xi_\mu \otimes \Psi_\mu^\delta \quad \text{with} \quad
\Psi_\mu^\delta := \{\psi_\nu : \nu \in ({\bf \Lambda}^\delta_{Y,0})_{1,\mu}\}.
\end{cases}
\]
Matrix representations of preconditioners from~\cite[\S5.6]{Stevenson2020b} are then given by
\[
\begin{cases}
{\bf K}_Y^\delta := \blockdiag[{\bf K}_\mu^\delta]_{\mu \in P_0{\bf \Lambda}^\delta_{Y,0}}
\quad \text{where} \quad
{\bf K}_\mu^\delta \eqsim ({\bf A}_\mu^\delta)^{-1}, \\
{\bf K}_X^\delta := \blockdiag[{\bf K}_\lambda^\delta {\bf A}_\lambda^\delta {\bf K}_\lambda^\delta]_{\lambda \in P_0 {\bf \Lambda}^\delta_0} \quad \text{where} \quad {\bf K}_\lambda^\delta \eqsim ({\bf A}_\lambda^\delta + 2^{|\lambda|} {\bf M}_\lambda^\delta)^{-1}
\end{cases}
\]
with ${\bf A}_\mu^\delta := (A_{\vec x} \Psi_\mu^\delta)(\Psi_\mu^\delta)$, ${\bf A}_{\lambda}^\delta := (A_{\vec x} \Psi_\lambda^\delta)(\Psi_\lambda^\delta)$, and ${\bf M}_{\lambda}^\delta := (M_{\vec x} \Psi_\lambda^\delta)(\Psi_\lambda^\delta)$.
Suitable spatial preconditioners ${\bf K}_\mu^\delta$ are provided by multigrid methods.
In~\cite{Olshanskii2000} it was shown that for quasi-uniform triangulations, satisfying a `full-regularity' assumption, a multiplicative multigrid method yields suitable ${\bf K}_\lambda^\delta$, and we assume these results to hold for our locally refined triangulations $\tria \in \mathbb T$ as well. In \S\ref{sec:multigrid} below, we detail our linear-complexity multigrid implementation following~\cite{Wu2017a}.

\subsection{Right-hand side}
\label{sec:rhs}
We follow~\cite[\S 6.4]{Stevenson2020b}.
For $g \in C(\overline{I \times \Omega})$, $u_0 \in C(\overline \Omega)$, we can approximate the right-hand side of~\eqref{eqn:discr-schur} by  interpolants,  avoiding quadrature issues.

The procedure of \S\ref{sec:fem} for constructing the hierarchical basis $\Psi_{\mathfrak N} := \{\psi_\nu : \nu \in \mathfrak N\}$ can be applied in time as well, yielding the basis $\{\psi_\lambda : \lambda \in \vee_\mathfrak I\}$ from Figure~\ref{fig:time-hbf} which index set $\vee_{\mathfrak I}$ coincides with $\vee_\Sigma$.
We construct $\{\tilde{\psi}_\nu : \nu \in \mathfrak N\} \subset C(\overline{\Omega})'$ biorthogonal to $\Psi_{\mathfrak N}$, with $\tilde{\psi}_\nu := \delta_\nu - \sum_{\tilde \nu \in \gp(\nu)}\delta_{\tilde \nu}/2$.
In time, define $\{\tilde{\psi}_\lambda : \lambda \in \vee_{\mathfrak I}\} \subset C(\overline I)'$ analogously.
Define the vectors ${\bf g} := [(\tilde{\psi}_\lambda \otimes \tilde{\psi}_\nu)(g)]_{(\lambda, \nu) \in {\bf \Lambda}^\delta}$ and ${\bf u}_0 := [\tilde{\psi}_\nu(u_0)]_{\nu \in P_1{\bf \Lambda}^\delta}$.
Upon replacing $(g, u_0)$ in~\eqref{eqn:discr-schur} by the interpolants
\[
{}^{\delta\!} g := \sum_{(\lambda, \nu) \in {\bf \Lambda}^\delta} {\bf g}_{(\lambda, \nu)} \psi_\lambda \otimes \psi_\nu, \quad
{}^{\delta\!} u_0 := \sum_{\nu \in P_1 {\bf \Lambda}^\delta} {\bf u}_{0,\nu} \psi_\nu,
\]
we can evaluate its right-hand side in linear complexity by computing the quantities
\begin{align*}
[\langle \xi_{\mu} \otimes \psi_{\nu}, {}^{\delta\!} g \rangle_{L_2(I \times \Omega)}]_{(\mu, \nu) \in {\bf \Lambda}_{Y,0}^{\hat\delta}}
&= R_{{\bf \Lambda}^{\hat\delta}_{Y,0}} (M_t \otimes M_{\bf x}) I_{{\bf \Lambda}^\delta} {\bf g}, \\
[\sigma_{\lambda}(0) \langle
\psi_{\nu}, {}^{\delta\!} u_0 \rangle_{L_2(\Omega)}]_{(\lambda, \nu) \in {\bf \Lambda}^\delta_0}
&= [\sigma_{\lambda}(0) {\bf w}_{\nu}]_{(\lambda, \nu) \in {\bf \Lambda}^\delta_0}\\
\text{where} \quad {\bf w} &=(M_{\bf x} \Psi_{\mathfrak N}|_{P_1 {\bf \Lambda}^\delta})(\Psi_{\mathfrak N}|_{P_1 {\bf \Lambda}^\delta}) {\bf u}_0.
\end{align*}

\subsection{Two-level basis}\label{sec:theta-delta}
We now discuss the construction of a uniformly $X$-stable basis $\Theta_\delta$,
needed in the local error estimator ${\bf r}^\delta$ of~\eqref{eqn:residual}.
Following \cite[\S 6.3]{Stevenson2020b},
define a \emph{modified hierarchical basis} $\{\hat{\psi}_\nu : \nu \in \mathfrak N_0\}$ by
\[
\hat{\psi}_\nu = \psi_\nu~~\text{when}~~\gen(\nu) = 0, \quad \text{else} \quad \hat{\psi}_\nu := \psi_\nu - \frac{\sum_{\{\tilde \nu \in \mathfrak N: \tilde \nu \triangleleft_{\mathfrak N} \nu\}} \frac{\int_\Omega \psi_\nu \dif {\bf x}}{ \int_\Omega \psi_{\tilde \nu} \dif {\bf x}}\psi_{\tilde \nu}}{\# \{\tilde \nu \in \mathfrak N: \tilde \nu \triangleleft_{\mathfrak N} \nu\}}.
\]
For any $\tria \in \mathbb T$, $W_\tria = \spann \{ \hat{\psi}_\nu : \nu \in N_{\tria,0}\} = \spann \Psi_\tria$  and the transformation from modified to unmodified hierarchical basis can be performed in linear complexity.
For $\underline{\tria} \succeq \tria \in \mathbb T$, ${\bf d} \in \ell_2(N_{\underline \tria,0} \setminus N_{\tria,0})$ and $v \in W_\tria$,~\cite[Lem.~6.7]{Stevenson2020b} shows that
\be\label{eqn:h1stab}
\begin{cases}
\|v + \sum_{\nu} {\bf d}_\nu \hat \psi_\nu\|_{H^1(\Omega)}^2 \eqsim \|v\|^2_{H^1(\Omega)} + \|{\bf d}\|^2, \\
\|v + \sum_\nu {\bf d}_\nu \hat \psi_\nu\|_{H^{-1}(\Omega)}^2 \eqsim \|v\|^2_{H^{-1}(\Omega)} + \sum_\nu 4^{-\gen(\nu)} |{\bf d}_\nu|^2,
\end{cases}
\ee 
with the constants in the $\eqsim$-symbols dependent on $\max\limits_{\{\underline \tria \ni \underline T \subset T \in \tria\}}\{\gen(\underline T) - \gen(T)\}$ only.
We then construct a basis for $X^\udelta \ominus X^\delta$ as
\be \label{eqn:theta}
\Theta_\delta := \{ e_{\lambda \nu} \sigma_\lambda \otimes \hat{\psi}_\nu : (\lambda, \nu) \in {\bf \Lambda}^\udelta_0 \setminus {\bf \Lambda}^\delta_0\} \quad \text{where} \quad \frac{1}{e_{\lambda \nu}} = \sqrt{1 + 4^{|\lambda| - \gen(\nu)}}
\ee

Define the \emph{gradedness} of a double-tree ${\bf \Lambda}^\delta \subset \vee_\Sigma \times \mathfrak N$ as the smallest $L_\delta \in \N$ for which every $(\lambda, \nu) \in {\bf \Lambda}^\delta$
with $\tilde \nu$ an ancestor of $\nu$ with $\gen(\nu) - \gen(\tilde \nu) = L_\delta$, it holds that $(\breve \lambda, \tilde \nu) \in {\bf \Lambda}^\delta$ for all $\breve\lambda \triangleleft_\Sigma \lambda$.
Thanks to $\Sigma$ being a (scaled) Riesz basis for $L_2(I)$ and $H^1(I)$, together with the $H^1(\Omega)$- and $H^{-1}(\Omega)$-stable splittings of~\eqref{eqn:h1stab}, it holds that
\[
\|z + {\bf c}^\top \Theta_\delta\|_X^2 \eqsim \|z\|_X^2 + \|{\bf c}\|^2 \quad ({\bf c} \in \ell_2({\bf \Lambda}^\udelta_0 \setminus {\bf \Lambda}^\delta_0), z \in X^\delta),
\]
with the constant in the $\eqsim$-symbol dependent on $L_\delta$ only,
so when $L_\delta$ is uniformly bounded, condition~\eqref{eqn:X-stable} is satisfied.

\section{Implementation}\label{sec:impl}
A tree-based implementation of the aforementioned adaptive algorithm in C++ can be found at~\cite{VanVenetie2021}.
In this section, we describe our design choices for a linear complexity implementation.

\subsection{Trees and linear operators in one axis}
In \S\ref{sec:lincomp}, we consider an abstract multilevel collection $\Psi$ indexed on $\vee_\Psi$. Endowed with a parent-child relation, $\vee_\Psi$ has a tree-like structure that we call a \emph{mother tree}; see also Figures~\ref{fig:3pt-tree} and~\ref{fig:ortho-tree}.

In our applications, the support of a wavelet $\psi_\lambda$ is a union of simplices of generation $|\lambda|$.
In time, these simplices are subintervals of $I$ found by dyadic refinement.
In space, they are elements of $\mathfrak T$, the collection of all triangles found by newest vertex bisection.
Endowed with the natural parent-child relation, both collections of simplices have a tree structure we call the \emph{domain mother tree}.
Every wavelet $\psi_\lambda$ stores references to the simplices $T$ of generation $|\lambda|$ that make up its support; conversely, every $T$ stores a reference to $\psi_\lambda$.

Every mother tree $\vee$ is stored once in memory, and every node $\lambda \in \vee$ stores references to its parents, children, and siblings.
We treat the mother tree as infinite by \emph{lazy initialization}, constructing new nodes as they are needed.

\subsubsection{Trees}
We store a tree $\Lambda \subset \vee$ using the parent-child relation,
and additionally, at each $\lambda \in \Lambda$ store a reference to the corresponding node in $\vee$.
This allows us to compare different trees subject to the same mother tree.
This tree-like representation does not allow direct access of arbitrary nodes: in any
operation, we traverse $\Lambda$ from its roots in breadth-first, or level-wise, order.

\subsubsection{Tree operations}
One important operation is the \texttt{union} of one tree $\Lambda$ into another $\breve \Lambda$.
This can be implemented by traversing both trees simultaneously in breadth-first order.
The union allows us to easily perform high-level operations, such as vector addition: given two vectors ${\bf c} \in \ell_2(\Lambda)$,
${\bf d} \in \ell_2(\breve \Lambda)$ on the same mother tree $\vee$, we use the union to perform
${\bf c} := {\bf c} + {\bf d}$. See Figure~\ref{fig:tree-refine} for an example.
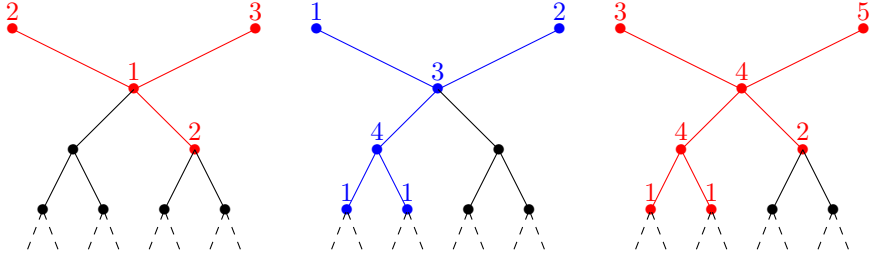
\begin{figure}[ht]
\centering
\begin{tikzpicture}[scale=0.8]
    \begin{scope}
    \node[red] at (0,1){\textbullet}; \node[red,above] at (0,1){$2$};
    \node[red] at (4,1){\textbullet}; \node[red,above] at (4,1){$3$};
    
    \draw[red] (0,1) -- (2,0) -- (4,1);
    \node[red] at (2,0){\textbullet}; \node[red,above] at (2,0){$1$};
    \draw (2,0) -- (1,-1);
    \node at (1,-1){\textbullet};
    \draw[red] (2,0) -- (3,-1);
    \node[red] at (3,-1){\textbullet}; \node[red,above] at (3,-1){$2$};
    
    \draw (1,-1) -- (0.5,-2);
    \node at (0.5,-2){\textbullet};
    \draw (1,-1) -- (1.5,-2);
    \node at (1.5,-2){\textbullet};
    \draw (3,-1) -- (2.5,-2);
    \node at (2.5,-2){\textbullet};
    \draw (3,-1) -- (3.5,-2);
    \node at (3.5,-2){\textbullet};
    
    \draw[dashed] (3.75,-2.65) -- (3.5,-2);
    \draw[dashed] (2.75,-2.65) -- (2.5,-2);
    \draw[dashed] (1.75,-2.65) -- (1.5,-2);
    \draw[dashed] (0.75,-2.65) -- (0.5,-2);
    \draw[dashed] (3.25,-2.65) -- (3.5,-2);
    \draw[dashed] (2.25,-2.65) -- (2.5,-2);
    \draw[dashed] (1.25,-2.65) -- (1.5,-2);
    \draw[dashed] (0.25,-2.65) -- (0.5,-2);
    \end{scope}
    \begin{scope}[shift={(5,0)}]
    \node[blue] at (0,1){\textbullet}; \node[blue,above] at (0,1){$1$};
    \node[blue] at (4,1){\textbullet}; \node[blue,above] at (4,1){$2$};
    
    \draw[blue] (0,1) -- (2,0) -- (4,1);
    \node[blue] at (2,0){\textbullet}; \node[blue,above] at (2,0){$3$};
    \draw[blue] (2,0) -- (1,-1);
    \node[blue] at (1,-1){\textbullet}; \node[blue,above] at (1,-1){$4$};
    \draw (2,0) -- (3,-1);
    \node at (3,-1){\textbullet};
    
    \draw[blue] (1,-1) -- (0.5,-2);
    \node[blue] at (0.5,-2){\textbullet}; \node[blue,above] at (0.5,-2){$1$};
    \draw[blue] (1,-1) -- (1.5,-2);
    \node[blue] at (1.5,-2){\textbullet}; \node[blue,above] at (1.5,-2){$1$};
    \draw (3,-1) -- (2.5,-2);
    \node at (2.5,-2){\textbullet};
    \draw (3,-1) -- (3.5,-2);
    \node at (3.5,-2){\textbullet};
    
    \draw[dashed] (3.75,-2.65) -- (3.5,-2);
    \draw[dashed] (2.75,-2.65) -- (2.5,-2);
    \draw[dashed] (1.75,-2.65) -- (1.5,-2);
    \draw[dashed] (0.75,-2.65) -- (0.5,-2);
    \draw[dashed] (3.25,-2.65) -- (3.5,-2);
    \draw[dashed] (2.25,-2.65) -- (2.5,-2);
    \draw[dashed] (1.25,-2.65) -- (1.5,-2);
    \draw[dashed] (0.25,-2.65) -- (0.5,-2);
    \end{scope}
    \begin{scope}[shift={(10,0)}]
    \node[red] at (0,1){\textbullet}; \node[red,above] at (0,1){$3$};
    \node[red] at (4,1){\textbullet}; \node[red,above] at (4,1){$5$};
    
    \draw[red] (0,1) -- (2,0) -- (4,1);
    \node[red] at (2,0){\textbullet}; \node[red,above] at (2,0){$4$};
    \draw[red] (2,0) -- (1,-1);
    \node[red] at (1,-1){\textbullet}; \node[red,above] at (1,-1){$4$};
    \draw[red] (2,0) -- (3,-1);
    \node[red] at (3,-1){\textbullet}; \node[red,above] at (3,-1){$2$};
    
    \draw[red] (1,-1) -- (0.5,-2);
    \node[red] at (0.5,-2){\textbullet}; \node[red,above] at (0.5,-2){$1$};
    \draw[red] (1,-1) -- (1.5,-2);
    \node[red] at (1.5,-2){\textbullet}; \node[red,above] at (1.5,-2){$1$};
    \draw (3,-1) -- (2.5,-2);
    \node at (2.5,-2){\textbullet};
    \draw (3,-1) -- (3.5,-2);
    \node at (3.5,-2){\textbullet};
    
    \draw[dashed] (3.75,-2.65) -- (3.5,-2);
    \draw[dashed] (2.75,-2.65) -- (2.5,-2);
    \draw[dashed] (1.75,-2.65) -- (1.5,-2);
    \draw[dashed] (0.75,-2.65) -- (0.5,-2);
    \draw[dashed] (3.25,-2.65) -- (3.5,-2);
    \draw[dashed] (2.25,-2.65) -- (2.5,-2);
    \draw[dashed] (1.25,-2.65) -- (1.5,-2);
    \draw[dashed] (0.25,-2.65) -- (0.5,-2);
    \end{scope}
\end{tikzpicture}
\vspace{-1em}
\caption{Left: ${\bf c} \in \ell_2(\Lambda)$ for $\Lambda \subset \vee_{\mathfrak I}$; Middle: ${\bf d} \in \ell_2(\breve \Lambda)$ for $\breve \Lambda \subset \vee_{\mathfrak I}$; Right: in-place sum ${\bf c } := {\bf c} + {\bf d}$.}
\label{fig:tree-refine}
\end{figure}

\subsubsection{Tree operations in time}
The routines \texttt{eval}, \texttt{evalupp}, and \texttt{evallow} from \S\ref{sec:eval_time} involve various level-wise index sets (represented as arrays of references into their mother trees). One example is $\breve{\Pi}_B = \{\lambda \in \breve{\Pi} : \big|\supp \breve{\phi}_\lambda \cap \cup_{\mu \in \Lambda_\ell}S(\mu)\big|>0\}$, which we constructed efficiently using the domain mother tree; see
Algorithm~\ref{alg:constructPiBout}.

\begin{algorithm}\label{alg:constructPiBout}
\KwData{$\ell \in \N$, $\breve{\Pi} \subset \breve{\Delta}_{\ell-1}$, $\Lambda_\ell \subset \vee_\ell$.  }
\KwResult{$[\breve{\Pi}_A, \breve{\Pi}_B]$ where $\breve{\Pi}_A=\breve{\Pi} \setminus \breve{\Pi}_B$,  $\breve{\Pi}_B=\{\lambda \in \breve{\Pi} : \big|\supp \breve{\phi}_\lambda \cap \cup_{\mu \in \Lambda_\ell}S(\mu)\big|>0\}$.
}
$\breve{\Pi}_A := \emptyset$\;
$\breve{\Pi}_B := \emptyset$\;
\For{$\mu \in \Lambda_\ell$}{
    \For(\tcp*[f]{We have $S(\psi_\mu) = \supp \psi_\mu$.}){$T \in \psi_\mu.\mathtt{support}$}{
        $T$.\texttt{parent}.\texttt{marked} := true\;
    }
}
\For{$\lambda \in \breve{\Pi}$}{
    \If{$\exists T \in \breve \phi_\lambda.\mathtt{support}$ with $T.\mathtt{marked}$ = true}
    {
        $\breve{\Pi}_B$.insert($\lambda$)\;
    }
    \Else{
        $\breve{\Pi}_A$.insert($\lambda$)\;
    }
}
\For{$\mu \in \Lambda_\ell$}{
    \For{$T \in \psi_\mu.\mathtt{support}$}{
        $T$.\texttt{parent}.\texttt{marked} := false\;
    }
}
\caption{The construction of $\breve{\Pi}_B$.}
\end{algorithm}

We can apply the linear operators appearing in the routines of~\S\ref{sec:eval_time} efficiently by again traversing the domain mother tree;
for example, Algorithm~\ref{alg:apply_operator} details a matrix-free application of $(A\Phi|_{\Pi})(\breve{\Phi}|_{\breve{\Pi}})$.
\begin{algorithm}\label{alg:apply_operator}
\KwData{Index sets $\Pi \subset \Delta_{\ell}, \breve{\Pi} \subset \breve{\Delta}_{\ell}$, $\vec{d} \in \ell_2(\Pi)$, \mbox{local and linear} $A \colon \spann \Phi \to \spann \breve{\Phi}'.$}
\KwResult{$\vec{e} = (A\Phi|_{\Pi})(\breve{\Phi}|_{\breve{\Pi}}) \vec{d}$}

\lFor{$\lambda \in \Pi$}{$\phi_\lambda.\mathtt{data} := \vec{d}_{\lambda}$}

\For{$\mu \in \breve{\Pi}$}{
        ${\bf e}_\lambda := 0$\;
    \For{$T \in \breve{\phi}_\mu.\mathtt{support}$} {
        \For(\tcp*[f]{$\{\phi_\lambda : \lambda \in \Delta_\ell, |\supp \phi_\lambda \cap T| > 0\}$}){$\phi_\lambda \in T.\mathtt{functions}(\Delta_{\ell})$} {
    ${\bf e}_{\lambda} := {\bf e}_{\lambda} + A(\phi_\lambda)(\breve{\phi}_\mu|_T) \cdot \phi_\lambda.\mathtt{data}$\;
        }

    }
}
\lFor{$\lambda \in \Pi$}{$\phi_\lambda.\mathtt{data} := 0$}

\caption{The computation of $\vec{e} = (A\Phi|_{\Pi})(\breve{\Phi}|_{\breve{\Pi}}) \vec{d}$.}

\end{algorithm}

\subsubsection{Operations in space}\label{sec:triangview}\label{sec:multigrid}
We can construct a triangulation $\tria$ from a vertex tree $N_\tria$
in linear complexity. First mark every $\nu \in N_\tria$ in its mother tree, then traverse the
domain mother tree $\mathfrak T$. A triangle $T$ visited in this traversal is in $\tria$ exactly when the newest vertex of its children is not marked.

For the preconditioners ${\bf K}_\mu^\delta$ and ${\bf K}_\lambda^\delta$ from~\S\ref{sec:preconds} we use multigrid.
We apply multiplicative V-cycle multigrid, in each
cycle applying one pre- and one post Gauss-Seidel smoother with reversed ordering of the unknowns.

In view of obtaining a linear complexity algorithm, at level $k$ we restrict smoothing to the vertices of generation $k$ as well as their godparents, cf.~\cite{Wu2017a}.
For $\tria \in \bbT$ we consider $W_{\tria}$, the space of continuous piecewise linears w.r.t.~$\tria$, zero on $\partial \Omega$, now equipped with the single-scale basis $\Phi_{\tria}$.   Set  $L=L(\tria):=\max_{T \in \tria} \gen(T)$, and define  the sequence
\[
\tria_\bot = \tria_0 \prec\tria_1 \prec \cdots \prec \tria_L=\tria \subset \bbT
\]
where $\tria_{k-1}$ is constructed from $\tria_k$ by removing all vertices $\nu \in N_{\tria_k}$ for which $\gen(\nu)=k$.
For $1 \leq k \leq L$, let $M_k$ be the set of new vertices and their godparents, i.e., $M_k := \bigcup_{\nu \in N_{\tria_k} \setminus N_{\tria_{k-1}}}\{ \nu\} \cup \gp(\nu)$, and let $M_{k,0} := M_k \setminus \partial \Omega$ be the vertices not on the boundary.
We consider the multilevel decomposition, cf.~\cite{Wu2017a},
\begin{equation}\label{eq:multilevel}
W_{\tria_L} = W_{\tria_0} + \sum_{k=1}^{L} \sum_{\nu \in M_{k,0}} \spann \phi_{k, \nu}, \quad \text{where}\quad \phi_{k, \nu} := \phi_{\tria_k, \nu}.
\end{equation}


For $1 \leq k \leq L$, let ${\bf P}_k$ be the \emph{prolongation matrix}, i.e., the matrix representation of the embedding $W_{\tria_{k-1}} \to W_{\tria_k}$, and enumerate the vertices $M_{k,0}$ as $(\nu_k^i)_{i =1}^{n_k}$.
Algorithm~\ref{alg:multigrid} details a (non-recursive) implementation of a single multiplicative V-cycle for the multilevel decomposition \eqref{eq:multilevel} using Gauss-Seidel smoothing.
We assume the availability of an efficient coarse-grid solver; in our application, a direct solve suffices.
For linear complexity, we use in-place vector updates restricted to non-zeros.

Note that this multigrid method is given in terms of the single-scale basis $\Phi_\tria$;  it can be transformed to the hierarchical basis  $\Psi_\tria$ similarly to~\eqref{eqn:hb2ss}. Multiple V-cycles are done by setting $u_0 := 0$ and iterating $u_k := \text{MG}(A, f - A u_{k-1})$.

\begin{algorithm}[H]
\caption{Single multiplicative V-cycle multigrid $\text{MG}(A, f)$.}
\label{alg:multigrid}
\KwData{
Some $f \in W_{\tria}'$ and a linear operator $A \colon W_{\tria} \to W_{\tria}'$. }
\KwResult{$u = {\bf u}^\top \Phi_\tria \in W_\tria$, the result of a single V-cycle applied to $f$.}
${\bf r} := f(\Phi_\tria)$\;  
\For{$L \geq k \geq 1$}{
\For{$\nu = \nu_k^1, \ldots, \nu_k^{n_k}$} {
    $r_{k, \nu} := {\bf r}_{\nu}$\;
    $e_{k, \nu} := r_{k, \nu} / (A\phi_{k, \nu})(\phi_{k, \nu})$\;
    ${\bf r} := {\bf r} - e_{k,\nu} (A\phi_{k, \nu})(\Phi_{\tria_k})$\;
}
${\bf r} := {\bf P}_k^\top {\bf r}$\;
}
~\\
Solve $ (A \Phi_{\tria_0})(\Phi_{\tria_0}){\bf u}= {\bf r}$\;
~\\
\For{$1 \leq k \leq L$}{
${\bf u} := {\bf P}_k {\bf u}$\;
\For{$\nu = \nu_k^{n_k}, \ldots, \nu_k^1$} {
    ${\bf u}_{\nu} := {\bf u}_{\nu} + e_{k, \nu}$\; 
${\bf u}_{\nu} := {\bf u}_{\nu} + (r_{k, \nu} -  (A \phi_{k, \nu})( {\bf u}^\top \Phi_{\tria_k}))/(A\phi_{k,\nu})(\phi_{k,\nu})$\;
}
}
\end{algorithm}


\subsection{Double-trees and tensor-product operators}
For every node in a double-tree ${\bf \Lambda} \subset \vee^0 \times \vee^1$,
we store a reference to the underlying pair of nodes in their mother trees.
This allows growing double-trees intuitively, and allows comparing different double-trees over the same pair of mother trees.
C++ templates allow us to re-use much of the tree code without runtime performance loss.

In \S\ref{Sappl-of-tensors} we saw how to apply a tensor-product operator. For this, we first construct the double-trees
${\bf \Sigma}$ and ${\bf \Theta}$; construction of ${\bf \Sigma}$ is illustrated
in Algorithm~\ref{alg:sigma}. 
Evaluation of the operator then reduces to the four simple steps of Algorithm~\ref{alg:tensor}.
\begin{algorithm}\label{alg:sigma}
\KwData{$\breve{\bm{\Lambda}} \subset \breve{\vee}^0 \times \breve{\vee}^1$, $\bm{\Lambda} \subset \vee^0 \times \vee^1$}
\KwResult{${\bf \Sigma}$ for application of Theorem~\ref{thm1} with $\breve{\bf \Lambda}$ and ${\bf \Lambda}$.}
${\bf \Sigma} := P_0 \bm{\Lambda} \times \{ \nu \in P_1 \breve{\bm{\Lambda}} : |\nu| = 0\}$\;
\For{$\lambda \in {\bm \Sigma}.\mathtt{project}(0)$}{
\For{$T \in \phi_\lambda.\mathtt{support}$} {
    \For{$\mu \in T.\mathtt{functions}(\breve{\vee}^0_{|\lambda|})$} {
        ${\bf \Sigma}.\mathtt{fiber}(1, \lambda).\mathtt{union}(\breve{{\bf \Lambda}}.\mathtt{fiber}(1, \mu))$\;
    }
}
}
\caption{Function $\mathtt{GenerateSigma}({\bf \breve{\Lambda}}, {\bf \Lambda})$. }
\end{algorithm}

\begin{algorithm}\label{alg:tensor}
     \KwData{$\bm{\Lambda} \subset \vee^0 \times \vee^1$, $\breve{\bm{\Lambda}} \subset \breve{\vee}^0 \times \breve{\vee}^1$, ${\bf c} \in \ell_2({\bf \Lambda})$, ${\bf d} \in \ell_2({\breve{\bf \Lambda}})$.}
${\bf \Sigma} := \mathtt{GenerateSigma}(\breve{\bf \Lambda}, {\bf \Lambda})$\;
${\bf \Theta} := \mathtt{GenerateTheta}(\breve{\bf \Lambda}, {\bf \Lambda})$\;
${\bf s} := {\bf 0} \in \ell_2(\bm \Sigma)$\;
${\bf t} := {\bf 0} \in \ell_2(\bm \Theta)$\;
${\bf l} := {\bf 0} \in \ell_2(\breve{\bm \Lambda})$\;
\BlankLine
\lFor{$\lambda \in {\bf s}.\mathtt{project}(0)$}{%
$\mathtt{eval}(A_1)({\bf s}.\mathtt{fiber}(1, \lambda), {\bf c}.\mathtt{fiber}(1,\lambda))$}
\lFor{$\mu \in {\bf l}.\mathtt{project}(1)$}{%
$\mathtt{evallow}(A_0)({\bf l}.\mathtt{fiber}(0, \mu), {\bf s}.\mathtt{fiber}(0,\mu))$}
\lFor{$\mu \in {\bf t}.\mathtt{project}(1)$}{%
$\mathtt{evalupp}(A_0)({\bf t}.\mathtt{fiber}(0, \mu), {\bf c}.\mathtt{fiber}(0,\mu))$}
\lFor{$\lambda \in {\bf d}.\mathtt{project}(0)$}{%
$\mathtt{eval}(A_1)({\bf d}.\mathtt{fiber}(1, \lambda), {\bf t}.\mathtt{fiber}(1,\lambda))$}
\BlankLine
${\bf d} := {\bf d} + {\bf l}$\;
\caption{Algorithm to evaluate ${\bf d} = R_{\breve {\bm \Lambda}}({\bm A}_0 \otimes {\bm A}_1) I_{\bm \Lambda} {\bf c}$.}
\end{algorithm}

\subsubsection{Memory optimizations}\label{sec:memopts}
As the memory consumption of a double-tree is significant, at around 280 bytes per node, we want to have as few double-trees in memory as possible.
By storing the nodes of ${\bf \Lambda}$ in a persistent container, every node is uniquely identified with its index in the container.
This induces a mapping $\R^{\# {\bf \Lambda}} \leftrightarrow \ell_2({\bf \Lambda})$ and allows us to overlay multiple vectors on the same underlying double-tree in a memory-friendly way.

The ${\bf \Sigma}$ generated by Algorithm~\ref{alg:sigma} for the application of a tensor-product operator can play the role of ${\bf \Theta}$ necessary for the application of its transpose operator (and vice versa). This allows tensor-product operators and their transposes to share the double-trees ${\bf \Sigma}$ and ${\bf \Theta}$.

With these insights, our implementation of the heat equation has at most 5 different double-trees in memory.

\subsection{The adaptive loop}\label{sec:algo-impl}
%
In the refine step of the adaptive loop, we first mark a set $J$ of nodes in ${\bf \Lambda}^{\udelta} \setminus {\bf \Lambda}^\delta$ using D\"orfler marking (possible in linear complexity; cf.~\cite{Pfeiler2019}). We then refine
${\bf \Lambda}^{\delta}$ to the smallest double-tree containing $J$:
\begin{enumerate}
\item mark all nodes in ${\bf \Lambda}^{\udelta}$ that are also present in ${\bf \Lambda}^{\delta}$ ((ii) in Fig.~\ref{fig:adaptrefine});
\item traverse ${\bf \Lambda}^{\udelta}$ from every node in $J$, top-down in level-wise order, until hitting a previously marked node. Mark all nodes along the way ((iii--iv) in Fig.~\ref{fig:adaptrefine});
\item $\mathtt{union}$ the marked nodes of ${\bf \Lambda}^{\udelta}$ into ${\bf \Lambda}^{\delta}$ ((v) in Fig.~\ref{fig:adaptrefine}).
\end{enumerate}
As  $\# {\bf \Lambda}^{\udelta} \lesssim \# {\bf \Lambda}^{\delta}$ and we visit every node of ${\bf \Lambda}^{\udelta}$ at most twice, the traversal is linear in $\#{\bf \Lambda}^\delta$. See also Figure~\ref{fig:adaptrefine}.
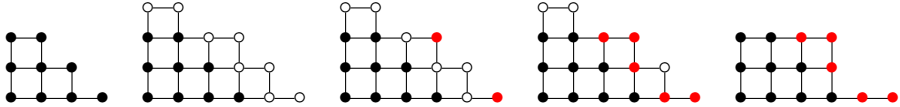
\begin{figure}[ht]
\centering
\begin{tikzpicture}[scale=0.4]
\begin{scope}
    \draw (0,0) -- (3,0); \draw (0,1) -- (2,1); \draw (0,2) -- (1,2);
    \draw (0,0) -- (0,2); \draw (1,0) -- (1,2); \draw (2,0) -- (2,1);
    \foreach \xy in {(0,0), (1,0), (0,1), (1,1), (2,0), (3,0), (2,1), (0,2), (1,2)} {
    \node at \xy{\textbullet};
    }
\end{scope}
\begin{scope}[shift={(4.5,0)}]
    \draw (0,0) -- (5,0); \draw (0,1) -- (4,1); \draw (0,2) -- (3,2); \draw (0,3) -- (1,3);
    \draw (0,0) -- (0,3); \draw (1,0) -- (1,3); \draw (2,0) -- (2,2); \draw (3,0) -- (3,2); \draw (4,0) -- (4,1);
    \foreach \xy in {(0,0), (1,0), (0,1), (1,1), (2,0), (3,0), (2,1), (0,2), (1,2)} {
    \node at \xy{\textbullet};
    }
    \foreach \xy in {(0,3), (1,3), (2,2), (3,2), (3,1), (4,1), (4,0), (5,0)} {
    \node[white] at \xy{\textbullet};
    \node at \xy{\textopenbullet};
    }
\end{scope}
\begin{scope}[shift={(11,0)}]
    \draw (0,0) -- (5,0); \draw (0,1) -- (4,1); \draw (0,2) -- (3,2); \draw (0,3) -- (1,3);
    \draw (0,0) -- (0,3); \draw (1,0) -- (1,3); \draw (2,0) -- (2,2); \draw (3,0) -- (3,2); \draw (4,0) -- (4,1);
    \foreach \xy in {(0,0), (1,0), (0,1), (1,1), (2,0), (3,0), (2,1), (0,2), (1,2)} {
    \node at \xy{\textbullet};
    }
    \foreach \xy in {(0,3), (1,3), (2,2), (3,1), (4,1), (4,0)} {
    \node[white] at \xy{\textbullet};
    \node at \xy{\textopenbullet};
    }
    \foreach \xy in {(3,2), (5,0)} {
    \node[red] at \xy{\textbullet};
    }
\end{scope}
\begin{scope}[shift={(17.5,0)}]
    \draw (0,0) -- (5,0); \draw (0,1) -- (4,1); \draw (0,2) -- (3,2); \draw (0,3) -- (1,3);
    \draw (0,0) -- (0,3); \draw (1,0) -- (1,3); \draw (2,0) -- (2,2); \draw (3,0) -- (3,2); \draw (4,0) -- (4,1);
    \foreach \xy in {(0,0), (1,0), (0,1), (1,1), (2,0), (3,0), (2,1), (0,2), (1,2)} {
    \node at \xy{\textbullet};
    }
    \foreach \xy in {(0,3), (1,3), (4,1)} {
    \node[white] at \xy{\textbullet};
    \node at \xy{\textopenbullet};
    }
    \foreach \xy in {(2,2), (3,1), (3,2), (4,0), (5,0)} {
    \node[red] at \xy{\textbullet};
    }
\end{scope}
\begin{scope}[shift={(24,0)}]
    \draw (0,0) -- (5,0); \draw (0,1) -- (3,1); \draw (0,2) -- (3,2);
    \draw (0,0) -- (0,2); \draw (1,0) -- (1,2); \draw (2,0) -- (2,2); \draw (3,0) -- (3,2);
    \foreach \xy in {(0,0), (1,0), (0,1), (1,1), (2,0), (3,0), (2,1), (0,2), (1,2)} {
    \node at \xy{\textbullet};
    }
    \foreach \xy in {(2,2), (3,1), (3,2), (4,0), (5,0)} {
    \node[red] at \xy{\textbullet};
    }
\end{scope}
\end{tikzpicture}
\vspace{-1em}
\caption{Adaptive refinement of a double-tree with underlying \emph{unary} mother trees. Left to right: (i) ${\bf \Lambda}^\delta$; (ii) ${\bf \Lambda}^{\udelta}$ with nodes in ${\bf \Lambda}^{\udelta} \setminus {\bf \Lambda}^\delta$ in white;  (iii) nodes in $J$ marked in red; (iv) nodes marked in the top-down traversal; (v) refined ${\bf \Lambda}^\delta$.}
\label{fig:adaptrefine}
\end{figure}

\section{Numerical experiments}
\label{sec:numerical}
We consider the heat equation~\eqref{eqn:heat}, and assess our implementation of the adaptive Algorithm~\ref{alg:adaptive} for its numerical solution.
Complementing the convergence results gathered in~\cite[\S7]{Stevenson2020b}, here we provide results on the practical performance of the adaptive loop.
Results were gathered on a multi-core 2.2 GHz machine, provided by the Dutch national e-infrastructure with the support of SURF Cooperative.

\subsection{The adaptive loop}\label{sec:adaptloop}
We summarize the main results from~\cite[\S7]{Stevenson2020b}. We run Algorithm~\ref{alg:adaptive} with $\theta = \tfrac{1}{2}$ and $\xi = \tfrac{1}{2}$.
We consider four problems.

In the \emph{smooth problem},
we select $\Omega := [0,1]^2$ and prescribe the solution
\[
u(t,x,y) := (1+t^2) x(1-x) y(1-y).
\]

In the \emph{moving peak} problem, we again select $\Omega := [0,1]^2$ with prescribed solution
\[
u(t,x,y) := x(1-x)y(1-y) \exp(-100 [(x-t)^2 + (y-t)^2]);
\]
here, $u$ is essentially zero outside a small strip along the diagonal $(0,0,0)$ to $(1,1,1)$.

In the \emph{cylinder} problem, we select $\Omega := [-1,1]^2 \setminus [-1,0]^2$ with data
\[
u_0 \equiv 0, \quad \text{and} \quad g(t,x,y) := t \cdot \bbone_{\{x^2 + y^2 < 1/4\}}.
\]
The solution has singularities in the re-entrant corner and along the wall of the cylinder $\{(t,x,y): x^2 + y^2 = 1/4\}$.

In the \emph{singular} problem, we select $\Omega := [-1,1]^2 \setminus [-1,0]^2$ with data $u_0 \equiv 1$ and $g \equiv 0$; the solution then has singularities along $\{0\} \times \partial \Omega$ and $I \times \{(0,0)\}$.

\subsubsection{Convergence}
To estimate the error $\|u - \hat u^\delta\|_X$, we measure the residual error estimator $\|{\bf r}^\delta(\hat u^\delta)\|$ from~\eqref{eqn:residual}; see also Lemma~\ref{lem:inexact}.
In the left pane of Figure~\ref{fig:convergence}, for the first three problems, we observe a convergence rate of $1/2$, which is the best that can be expected from our family of trial spaces $(X^\delta)_{\delta \in \Delta}$. For the singular problem, the reduced rate $0.4$ is found; it is unknown if a better rate can be expected.
\begin{figure}
\centering
\includegraphics[width=0.49\linewidth]{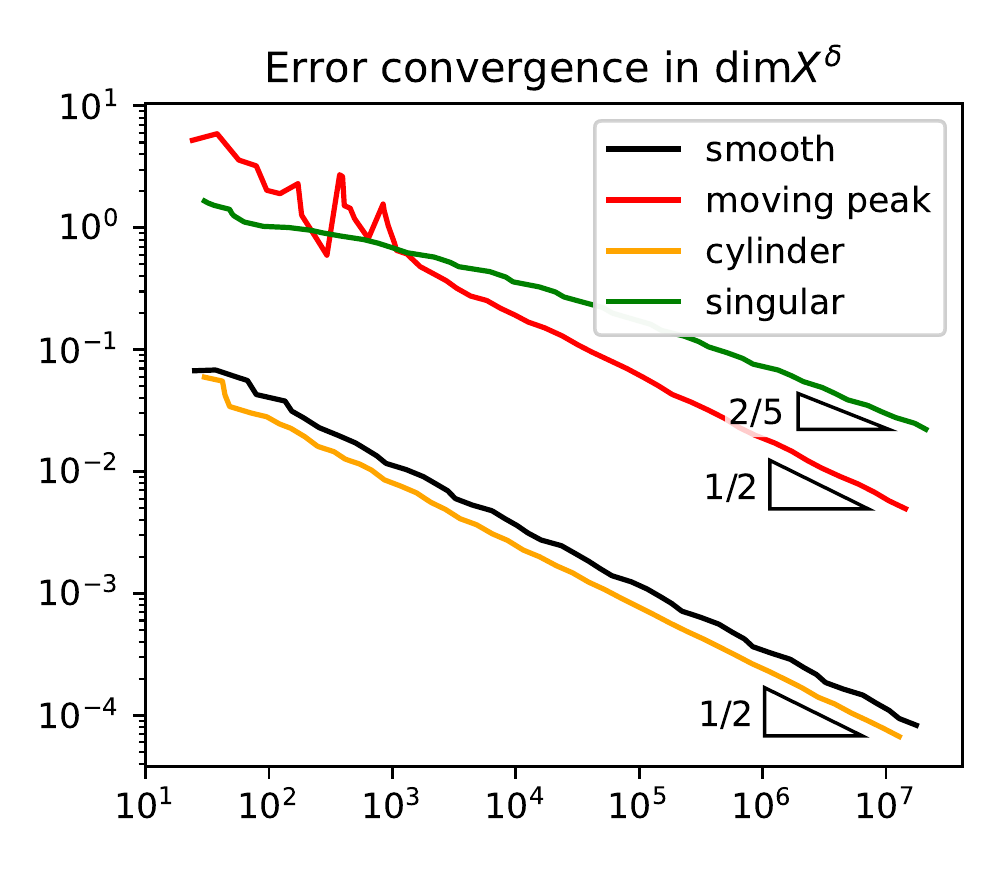}
\includegraphics[width=0.49\linewidth]{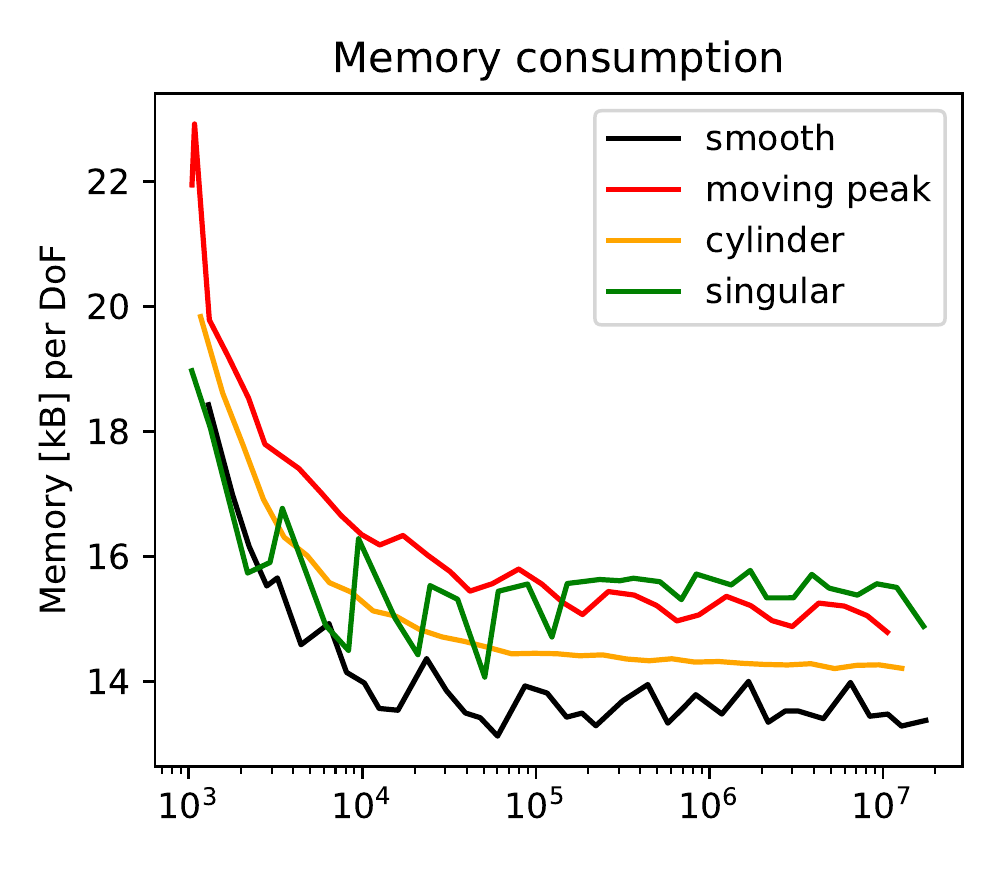}
\vspace{-1em}
\caption{Error convergence and peak memory usage of the adaptive loop for the four problems of~\S\ref{sec:adaptloop}.}
\label{fig:convergence}
\end{figure}

\subsubsection{Memory}
The right pane of Figure~\ref{fig:convergence} shows the peak memory consumption after every iteration of the adaptive algorithm.
We see that the peak memory is linear in $\dim X^\delta$, stabilizing to around 15kB per degree of freedom.
This is relatively high due to our implementation based on double-trees.
In fact, the double-trees together make up around 85\% of the total memory. For the singular problem, the largest double-tree ${\bf \Lambda}^\udelta_Y$ occupies around 40\% of the total memory.

\subsection{Linearity of operations}
The majority of our runtime is spent in the application of bilinear forms. In this section, we measure the application times to assert their linear complexity.

\subsubsection{In time}
We select three sequences $\{\Lambda_U\}$, $\{\Lambda_L\}$, $\{\Lambda_R\}$ of trees in $\vee_\Sigma$,
one uniformly refined and two graded towards the left and right respectively. For each such tree $\Lambda \subset \vee_\Sigma$, we define a corresponding tree $\breve \Lambda := \{ \mu \in \vee_\Xi : \exists \lambda \in \Lambda, |\lambda| = |\mu|,  |\supp \xi_\mu \cap \supp \sigma_\lambda| > 0\} \subset \vee_\Xi$.

We select the bilinear forms $M_t$ and $D_t$ from~\eqref{eqn:bilforms}, and run the algorithms from~\S\ref{sec:eval_time}. We see in Figure~\ref{fig:time-adapt} that the runtime per degree of freedom stabilizes to $10^{-3}$ ms, essentially independent of the bilinear form and the trees. We suspect the increase until $10^7$ degrees of freedom has to do with cache locality.
\begin{figure}
\centering
\includegraphics[width=\linewidth]{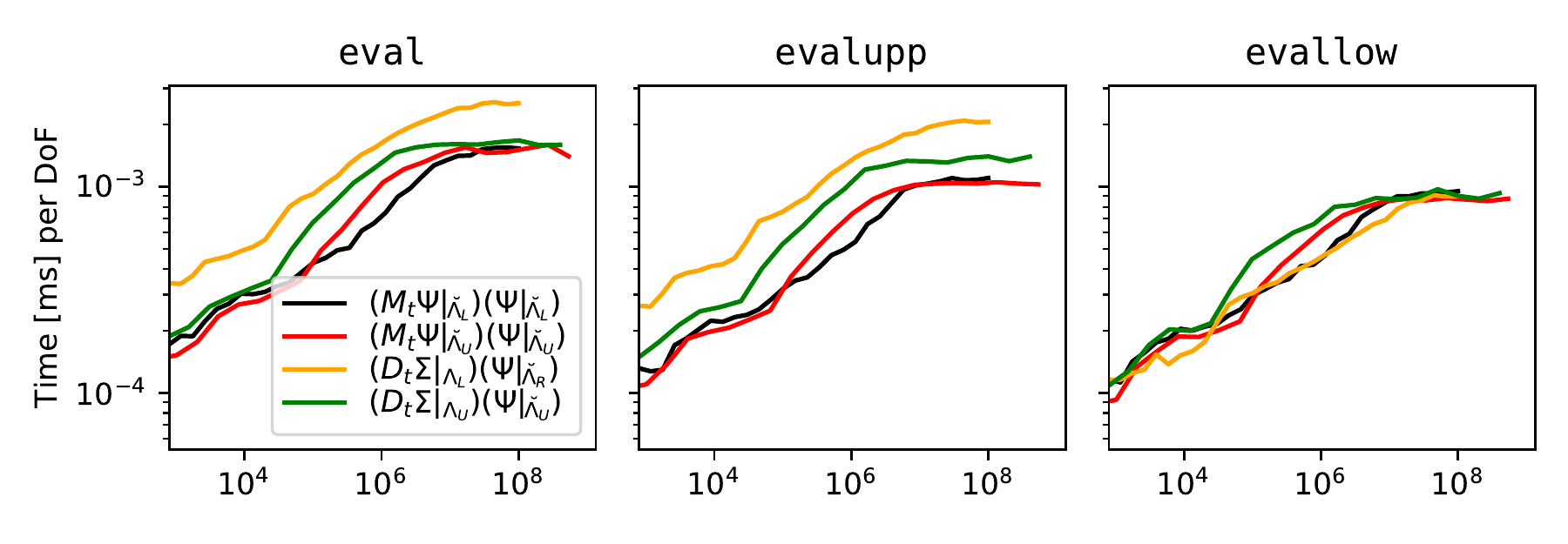}
\vspace{-2em}
\caption{Time (in ms) per DoF of bilinear form evaluations in time.}
\label{fig:time-adapt}
\end{figure}

\subsubsection{In space}
On the L-shaped domain $\Omega := [-1,1]^2 \setminus [-1,0]^2$, we select two sequences of hierarchical basis trees, one uniformly refined and the other refined by a standard adaptive loop on $-\Delta u = 1$, $u|_{\partial \Omega} = 0$.

For a hierarchical basis tree $\Psi_\tria = \{ \psi_\nu : \nu \in N_{\tria, 0}\}$, we denote the stiffness matrix $\langle \nabla \Psi_\tria, \nabla \Psi_\tria \rangle_{L_2(\Omega)}$
as ${\bf A}_\tria$.
We measure the runtime of the conversion from vertex tree $N_\tria$ to triangulation $\tria$ (cf.~\S\ref{sec:triangview}), the application time of ${\bf A}_\tria$ through~\eqref{eqn:hb2ss}, and that of multigrid on ${\bf A}_\tria$
(with 1 and 3 V-cycles) through Algorithm~\ref{alg:multigrid}.
Figure~\ref{fig:space-adapt} confirms that the relative runtime of every operation is essentially independent of the refinement strategy. Interesting is again the increase until $10^5$ degrees of freedom.
\begin{figure}
\centering
\includegraphics[width=0.5\linewidth]{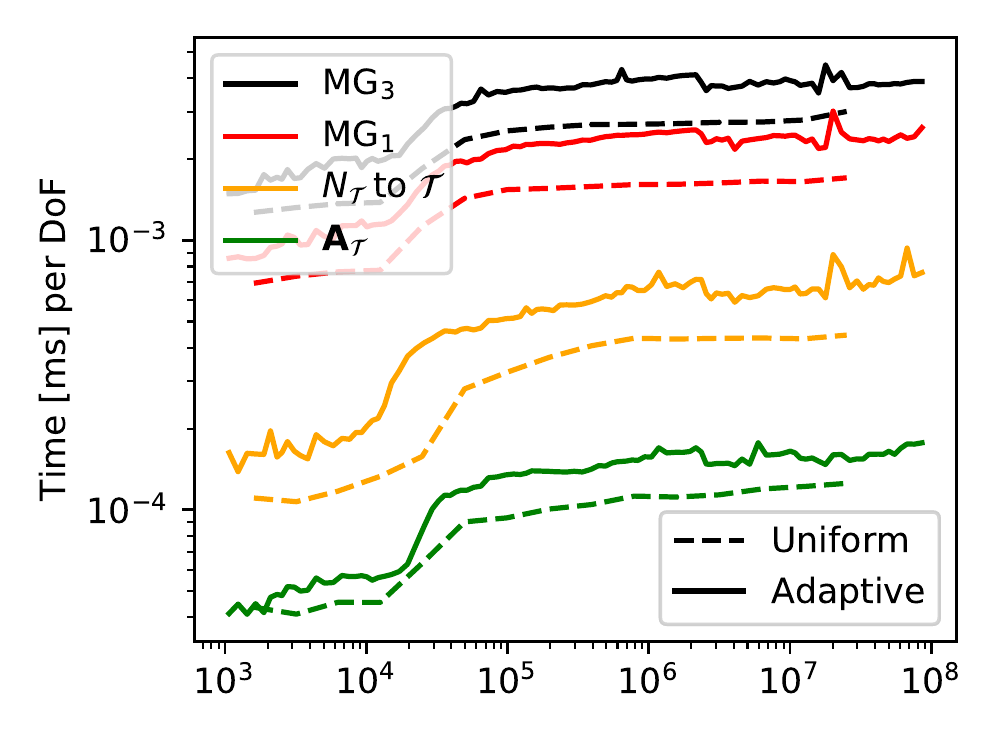}
\vspace{-1em}
\caption{Time (in ms) per DoF of important operations in space, for uniform and adaptive refinements.}
\label{fig:space-adapt}
\end{figure}

\subsubsection{In space-time}
Solving~\eqref{eqn:discr-schur} using PCG requires the application of the four linear operators ${E_Y^\udelta} B E_X^\delta$, ${E_X^\delta}' \gamma_0' \gamma_0 E_X^\delta$, $K_X^\delta$, and $K_Y^\udelta$.
For the first two, Corollary~\ref{cor:spacetime-linear} asserts that their application time is of linear complexity, while for the preconditioners $K_X^\delta$ and $K_Y^\udelta$, this follows from the block-diagonal structure of their matrix representation.

We run the adaptive algorithm on the four problems of~\S\ref{sec:adaptloop}. Figure~\ref{fig:spacetime-adapt} shows that the application time of the aforementioned operators is essentially independent of the problem, even though the underlying double-trees are vastly different. We again see an increase in relative runtime until $10^6$ degrees of freedom.

\begin{figure}
\centering
\includegraphics[width=\linewidth]{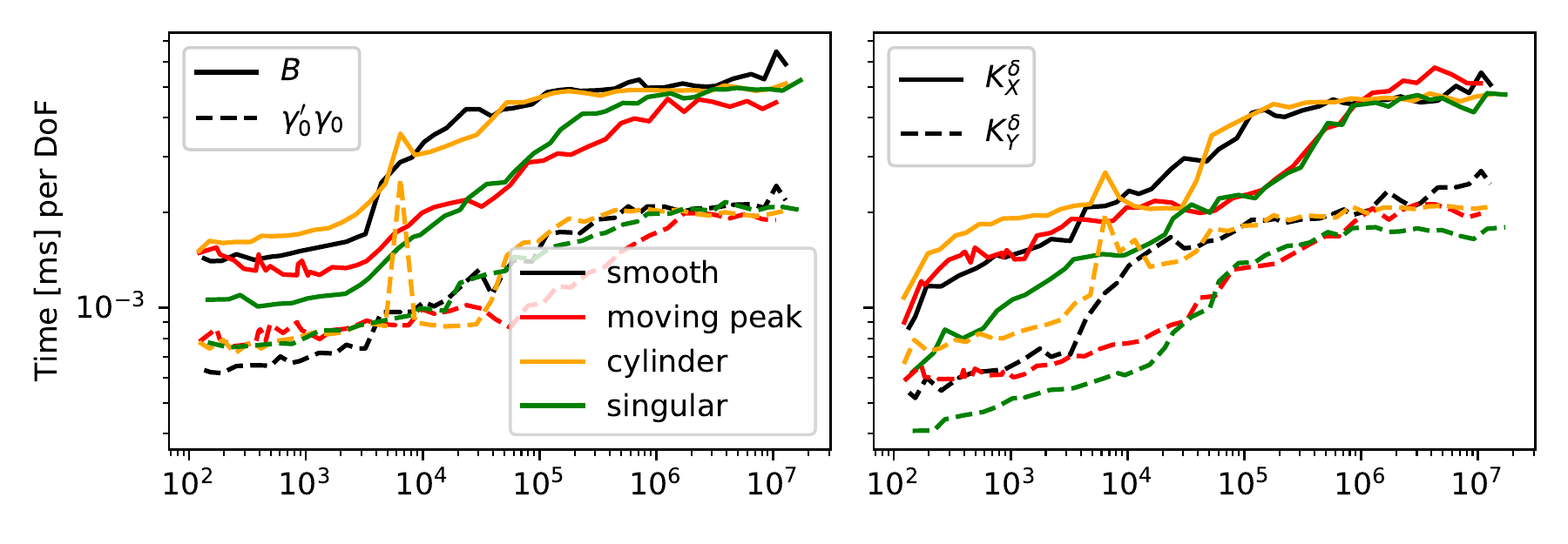}
\vspace{-2em}
\caption{Time (in ms) per DoF of the four bilinear forms applied in the solve step of the adaptive algorithm.}
\label{fig:spacetime-adapt}
\end{figure}

Figure~\ref{fig:spacetime-time} shows the runtimes of the solve, estimate, mark and refine steps of the adaptive loop. We confirm that each step is of linear complexity, and that  the total runtime is governed by the solve and estimate steps.

\begin{figure}
\centering
\includegraphics[width=\linewidth]{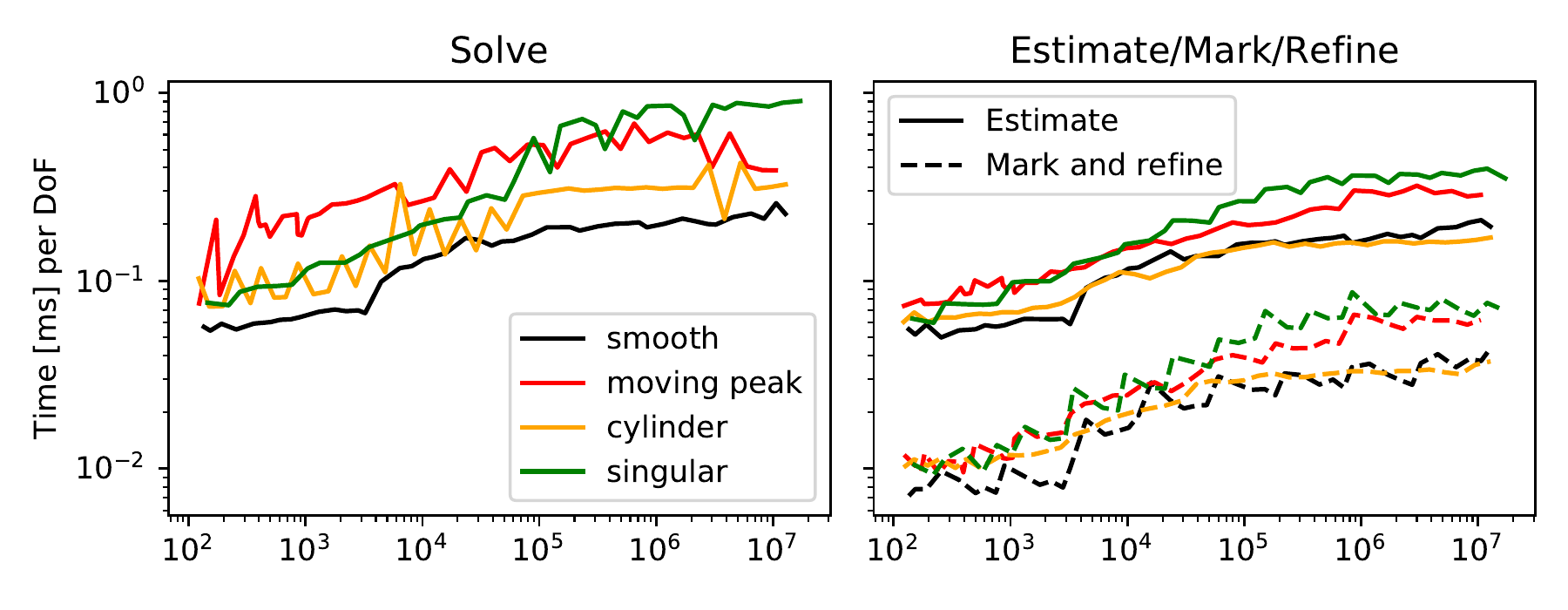}
\vspace{-2em}
\caption{Time (in ms) per DoF of the steps in the adaptive loop.}
\label{fig:spacetime-time}
\end{figure}

\subsection{Shared-memory parallelism}
Most of our execution time is spent applying the linear operators from Figure~\ref{fig:spacetime-adapt}.
We can obtain a significant speedup with multithreading.
In Algorithm~\ref{alg:tensor}, all fibers inside each of the four for-loops are disjoint, and we can easily parallelize each loop using OpenMP.



We run the parallel code on the smooth and singular problems.
The right pane of Figure~\ref{fig:spacetime-parallel} shows decent parallel performance for the singular problem, with 10$\times$ speedup at 16 cores.
The left pane however reveals a load balancing issue: as $u$ is smooth, the two fibers $({\bf \Lambda}^\delta_{0})_{1,\lambda}$ with $|\lambda| = 0$ contain the majority of the degrees of freedom. This results in poor parallel efficiency for the first and fourth loop in Algorithm~\ref{alg:tensor}.

\begin{figure}
\centering
\includegraphics[width=\linewidth]{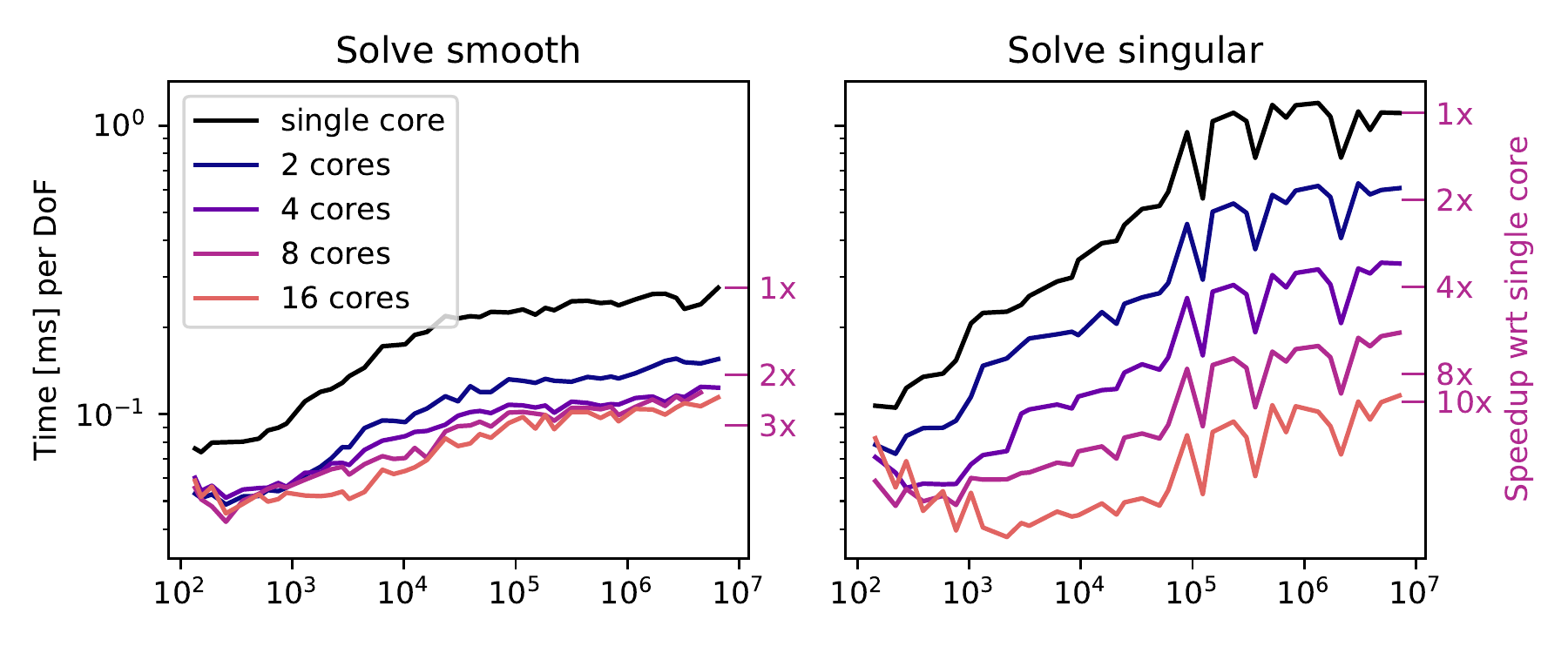}
\vspace{-2em}
\caption{Speedup and time (in ms) per DoF of the solve step in the adaptive loop, for different number of parallel processors.}
\label{fig:spacetime-parallel}
\end{figure}

\section{Conclusion}
We discussed an implementation of an adaptive solver for a space-time variational formulation of parabolic evolution equations where every step is of linear complexity.

We constructed a family of trial spaces spanned by tensor-products of wavelets in time and hierarchical basis functions in space. The resulting adaptive loop is able to resolve singularities locally in space and time, and we proved its $r$-linear convergence.

After imposing a \emph{double-tree} constraint on the index set of the trial spaces,
we \newer{devised} an abstract algorithm that is able to apply the system matrices in linear complexity.
We achieve this complexity in practice by a \emph{tree-based} implementation.
The numerical results show high performance of the adaptive loop as a whole.

\newer{\subsection*{Acknowledgements} The authors wish to thank their advisor Rob Stevenson for the many helpful comments.}


\appendix
\section{Proofs of Theorems in \S\ref{sec:lincomp}}\label{app:proofs}

\thmeval*
\begin{proof}
By locality of the collections $\breve{\Phi}$ and $\breve{\Psi}$, and sparsity of the matrices $\breve{\mathfrak{p}}_\ell$ and $\breve{\mathfrak{q}}_\ell$, we see that $\# \breve{\underline{\Pi}} \lesssim \#\breve{\Pi}_B +\# \breve{\Lambda}_\ell \lesssim \# \Lambda_\ell+\# \breve{\Lambda}_\ell$. So after sufficiently many recursive calls, the current set  $\breve{\Pi}  \cup \breve{\Lambda}$ will be empty.
For use later, we note that similarly
$\# \underline{\Pi}  \lesssim \#\Pi_B +\#\Lambda_\ell \lesssim \#\breve{\Lambda}_\ell+\#\breve{\Pi}_B +\#\Lambda_\ell \lesssim \#\Lambda_\ell+\#\breve{\Lambda}_\ell$.

For $\breve{\Pi}  \cup \breve{\Lambda}= \emptyset$, the call produces nothing, which is correct.

Now let  $\breve{\Pi}  \cup \breve{\Lambda} \neq \emptyset$.
From $\Lambda$ being an $\ell$-tree, the definitions of $S(\cdot)$ and $\breve{\Pi}_A$, and the locality of $A$, one has
\[
\vec{e}|_{\breve{\Pi}_A}=(Au)(\breve{\Phi}|_{\breve{\Pi}_A})
=(A(\vec{d}^\top\Phi|_{\Pi}))(\breve{\Phi}|_{\breve{\Pi}_A}).
\]

By choice of $\underline{\Pi}$ we have
\[
\underline{u}:=\underline{\vec{d}}^\top\Phi|_{\underline{\Pi}}+\vec{c}|_{\Lambda_{\ell+1\uparrow}}^\top \Psi|_{\Lambda_{\ell+1\uparrow}} =(\vec{d}|_{\Pi_B})^\top \Phi|_{\Pi_B} +\vec{c}^\top \Psi|_{\Lambda}=u-(\vec{d}|_{\Pi_A})^\top \Phi|_{\Pi_A}.
\]

By induction the recursive call  yields
$\underline{\vec{e}}=(A\underline{u})(\breve{\Phi}|_{\breve{\underline{\Pi}}})$, and
$\underline{\vec{f}}=(A\underline{u})(\breve{\Psi}|_{\breve{\Lambda}_{\ell+1 \uparrow}})$.
From  $\breve{\Lambda}$ being an $\ell$-tree, the definitions of $\breve{S}(\cdot)$ and $\Pi_A$, and the locality of $A$, we have
\[
(Au)(\breve{\Psi}|_{\breve{\Lambda}_{\ell \uparrow}})=(A \underline{u})(\breve{\Psi}|_{\breve{\Lambda}_{\ell \uparrow}}),
\]
and so in particular $\vec{f}|_{\Lambda_{\ell+1,\uparrow}}=\underline{\vec{f}}$.

The definition of $\breve{\underline{\Pi}}$ shows that
\[
\breve{\Phi}|_{\breve{\Pi}_B}=(\breve{\mathfrak{p}}_{\ell} ^\top \breve{\Phi}|_{\underline{\breve{\Pi}}})|_{\breve{\Pi}_B},\quad \breve{\Psi}|_{\breve{\Lambda}_\ell}=(\mathfrak{\breve{q}}_\ell^\top \breve{\Phi}|_{\underline{\breve{\Pi}}})|_{\breve{\Lambda}_\ell}.
\]
We conclude that
\[
\vec{f}|_{\breve{\Lambda}_\ell}=(Au)(\breve{\Psi}|_{\breve{\Lambda}_\ell})=(A\underline{u})(\breve{\Psi}|_{\breve{\Lambda}_\ell})
=
\big(\mathfrak{\breve{q}}_\ell^\top \underline{\vec{e}}\big)|_{\breve{\Lambda}_\ell},
\]
and from $|\supp \phi_\lambda \cap \supp \breve{\phi}_\mu|=0$ for $(\lambda,\mu) \in \Pi_A \times\breve{\Pi}_B$, that
\[
\vec{e}|_{\breve{\Pi}_B}=  (Au)(\breve{\Phi}|_{\breve{\Pi}_B})=(A\underline{u})(\breve{\Phi}|_{\breve{\Pi}_B})
=   \big(\breve{\mathfrak{p}}_{\ell} ^\top \underline{\vec{e}}\big)|_{\breve{\Pi}_B}.
\]

From the assumptions on the collections $\Phi$, $\breve{\Phi}$, $\breve{\Psi}$, and $\Psi$, and their consequences on the sparsity of the matrices $\mathfrak{p}_{\ell} $, $\breve{\mathfrak{p}}_{\ell}$, $\mathfrak{q}_\ell$, and $\mathfrak{\breve{q}}_\ell$, one infers that the total cost of the evaluations of the statements in \texttt{eval} is ${\mathcal O}(\# \breve{\Pi}+\# \breve{\Lambda}_\ell+\# \Pi+\# \Lambda_\ell)$ plus the cost of the recursive call. Using
$ \# \breve{\underline{\Pi}} + \# \underline{\Pi} \lesssim \# \breve{\Lambda}_\ell+\# \Lambda_\ell$
and induction, we conclude the second statement of the theorem.
\end{proof}

\thmevalupp*
\begin{proof}By locality of the collections $\breve{\Phi}$ and $\breve{\Psi}$, and sparsity of the matrices $\breve{\mathfrak{p}}_\ell$ and $\breve{\mathfrak{q}}_\ell$, we see that
$\# \breve{\underline{\Pi}} \lesssim \#\breve{\Pi}_B +\# \breve{\Lambda}_\ell \lesssim \# \Lambda_\ell+\# \breve{\Lambda}_\ell$. So after sufficiently many recursive calls, the current set  $\breve{\Pi}  \cup \breve{\Lambda}$ will be empty.
Notice that
$\# \underline{\Pi}  \lesssim \#\Lambda_\ell$.

For $\breve{\Pi}  \cup \breve{\Lambda}= \emptyset$, the call produces nothing, which is correct.

Now let  $\breve{\Pi}  \cup \breve{\Lambda} \neq \emptyset$.
From $\Lambda$ being an $\ell$-tree, the definitions of $S(\cdot)$ and $\breve{\Pi}_A$, and the locality of $A$, one has
\[
\vec{e}|_{\breve{\Pi}_A}=(Au)(\breve{\Phi}|_{\breve{\Pi}_A})=(A(\vec{d}^\top\Phi|_{\Pi}))(\breve{\Phi}|_{\breve{\Pi}_A}).
\]

By definition of $\underline{\Pi}$ we have
\[
\underline{u}:=\underline{\vec{d}}^\top\Phi|_{\underline{\Pi}}+\vec{c}|_{\Lambda_{\ell+1\uparrow}}^\top \Psi|_{\Lambda_{\ell+1\uparrow}} =\vec{c}^\top \Psi|_{\Lambda}=u-\vec{d}^\top \Phi|_{\Pi}.
\]

By induction the recursive call  yields
$\underline{\vec{e}}=(A\underline{u})(\breve{\Phi}|_{\breve{\underline{\Pi}}})$, and
$\underline{\vec{f}}={\bf U}_{\breve{\Lambda}_{\ell+1 \uparrow} \times \Lambda_{\ell+1 \uparrow}} c|_{\Lambda_{\ell+1 \uparrow}}=\vec{f}|_{\breve{\Lambda}_{\ell+1 \uparrow}}$.

The definition of $\breve{\underline{\Pi}}$ shows that
\[
\breve{\Phi}|_{\breve{\Pi}_B}=(\breve{\mathfrak{p}}_{\ell} ^\top \breve{\Phi}|_{\underline{\breve{\Pi}}})|_{\breve{\Pi}_B},\quad \breve{\Psi}|_{\breve{\Lambda}_\ell}=(\mathfrak{\breve{q}}_\ell^\top \breve{\Phi}|_{\underline{\breve{\Pi}}})|_{\breve{\Lambda}_\ell}.
\]
We conclude that
\[
\vec{f}|_{\breve{\Lambda}_\ell}= (A(\vec{c}^\top \Psi|_{\Lambda}))(\breve{\Psi}|_{\breve{\Lambda}_\ell})=(A\underline{u})(\breve{\Psi}|_{\breve{\Lambda}_\ell})=
\big(\mathfrak{\breve{q}}_\ell^\top \underline{\vec{e}}\big)|_{\breve{\Lambda}_\ell},
\]
and
\begin{align*}
\vec{e}|_{\breve{\Pi}_B}=(Au)(\breve{\Phi}|_{\breve{\Pi}_B}) &= (A\underline{u})(\breve{\Phi}|_{\breve{\Pi}_B})+(A(\vec{d}^\top\Phi|_{\Pi}))(\breve{\Phi}|_{\breve{\Pi}_B})\\&=   \big(\mathfrak{p}_{\ell} ^\top \underline{\vec{e}}\big)|_{\breve{\Pi}_B}+(A(\vec{d}^\top\Phi|_{\Pi}))(\breve{\Phi}|_{\breve{\Pi}_B}).
\end{align*}

From the assumptions on the collections $\Phi$, $\breve{\Phi}$, $\breve{\Psi}$, and $\Psi$, and their consequences on the sparsity of the matrices $\mathfrak{p}_{\ell} $, $\breve{\mathfrak{p}}_{\ell}$, $\mathfrak{q}_\ell$, and $\mathfrak{\breve{q}}_\ell$, one infers that the total cost of the evaluations of the statements in \texttt{eval} is ${\mathcal O}(\# \breve{\Pi}+\# \breve{\Lambda}_\ell+\# \Pi+\# \Lambda_\ell)$ plus the cost of the recursive call. Using
$ \# \breve{\underline{\Pi}} + \# \underline{\Pi} \lesssim \# \breve{\Lambda}_\ell+\# \Lambda_\ell$
and induction, we conclude the second statement of the theorem.
\end{proof}

\thmevallow*
\begin{proof}
Notice that
$\# \underline{\Pi}  \lesssim \#\Lambda_\ell+\#\Pi_B \lesssim  \#\Lambda_\ell+\#\breve{\Lambda}_\ell$.

For $\breve{\Pi} \cup \breve{\Lambda}= \emptyset$, the call produces nothing, which is correct.

Now let  $\breve{\Pi} \cup \breve{\Lambda}\neq \emptyset$. The definitions of $\underline{\breve{\Pi}}$ and $\underline{\Pi}_B$ show that
\[
\vec{f}|_{\breve{\Lambda}_\ell} =
(A\Phi|_{\Pi})(\breve{\Psi}|_{\breve{\Lambda}_\ell})\vec{d}=
(A\Phi|_{\Pi})(\breve{\Psi}|_{\breve{\Lambda}_\ell})\vec{d}|_{\Pi_B}=
\big(\mathfrak{\breve{q}}_\ell^\top (A\Phi|_{\underline{\Pi}_B})(\breve{\Phi}|_{\underline{\breve{\Pi}}})\mathfrak{p}_{\ell} \vec{d}|_{\Pi_B}\big)|_{\breve{\Lambda}_{\ell}}=
(\mathfrak{\breve{q}}_\ell^\top \underline{\vec{e}})|_{\breve{\Lambda}_{\ell}}
\]

From  $\breve{\Lambda}$ being an $\ell$-tree, the definitions of $\breve{S}(\cdot)$ and $\Pi_B$, and the locality of $a$, and for the third equality, the definition of $\underline{\Pi}$, one has
\begin{align*}
f|_{\breve{\Lambda}_{\ell+1 \uparrow}}&=a(\breve{\Psi}|_{\breve{\Lambda}_{\ell+1 \uparrow}},\Phi|_\Pi)\vec{d}
+
{\bf L}|_{\breve{\Lambda}_{\ell+1 \uparrow} \times \Lambda_\ell}\vec{c}|_{\Lambda_\ell}+
{\bf L}|_{\breve{\Lambda}_{\ell+1 \uparrow} \times \Lambda_{\ell+1 \uparrow }}\vec{c}|_{\Lambda_{\ell+1\uparrow}}\\
& = (A\Phi|_\Pi)(\breve{\Psi}|_{\breve{\Lambda}_{\ell+1 \uparrow})}\vec{d}|_{\Pi_B}+
(A\Psi|_{\Lambda_\ell})(\breve{\Psi}|_{\breve{\Lambda}_{\ell+1 \uparrow}})
\vec{c}|_{\Lambda_\ell}+{\bf L}|_{\breve{\Lambda}_{\ell+1 \uparrow} \times \Lambda_{\ell+1 \uparrow }}\vec{c}|_{\Lambda_{\ell+1\uparrow}}
\\
& = (A\Phi|_{\underline{\Pi}})(\breve{\Psi}|_{\breve{\Lambda}_{\ell+1 \uparrow}})\underline{\vec{d}}+{\bf L}|_{\breve{\Lambda}_{\ell+1 \uparrow} \times \Lambda_{\ell+1 \uparrow }}\vec{c}|_{\Lambda_{\ell+1\uparrow}}\\
& =\mathtt{evallow}(A)(\ell+1, \breve{\Lambda}_{\ell+1\uparrow}, \underline{\Pi}, \Lambda_{\ell+1\uparrow},\underline{\vec{d}}, \vec{c}|_{\Lambda_{\ell+1\uparrow}}) \end{align*}
by induction.

From the assumptions on the collections $\Phi$, $\breve{\Psi}$, and $\Psi$, and their consequences on the sparsity of the matrices $\mathfrak{p}_{\ell} $, $\mathfrak{q}_\ell$, and $\mathfrak{\breve{q}}_\ell$, one easily infers that the total cost of the evaluations of the statements in \texttt{evallow} is ${\mathcal O}(\# \breve{\Lambda}_\ell+\# \Pi+\# \Lambda_\ell)$ plus the cost of the recursive call. Using
$ \# \underline{\Pi} \lesssim \# \breve{\Lambda}_\ell+\# \Lambda_\ell$
and induction, we conclude the second statement of the theorem.
\end{proof}

\thmevaldbl*
\begin{proof}We write
\begin{align} \nonumber
R_{\breve{\bm{\Lambda}}} ({\bf A}_0 \otimes {\bf A}_1) I_{\bm{\Lambda}}=&R_{\breve{\bm{\Lambda}}} (({\bf L}_0+{\bf U}_0) \otimes {\bf A}_1) I_{\bm{\Lambda}}\\
\label{first}
=&R_{\breve{\bm{\Lambda}}} ({\bf L}_0\otimes \mathrm{Id})(\mathrm{Id} \otimes {\bf A}_1) I_{\bm{\Lambda}}+
\\
\label{second}
&R_{\breve{\bm{\Lambda}}} (\mathrm{Id} \otimes {\bf A}_1) ({\bf U}_0 \otimes \mathrm{Id}) I_{\bm{\Lambda}}.
\end{align}

Considering \eqref{first}, the range of $(\mathrm{Id} \otimes {\bf A}_1) I_{\bm{\Lambda}}$ consists of vectors whose entries with first index outside $P_0\bm{\Lambda}$ are zero.
In view of the subsequent application of ${\bf L}_0 \otimes \mathrm{Id}$, furthermore only those indices $(\lambda,\gamma) \in P_0\bm{\Lambda} \times \breve{\vee}^1$ of these vectors might be relevant for which $\exists (\mu,\gamma) \in \breve{\bm{\Lambda}}$, i.e. $\gamma \in \bm{\Lambda}_{1,\mu}$, with $|\mu|>|\lambda|$ and $|\breve{S}^0(\mu) \cap S^0(\lambda)|>0$.
Indeed $|\breve{S}^0(\mu) \cap S^0(\lambda)|=0$ implies
$|\supp \breve{\psi}^0_\mu \cap \supp \psi^0_\lambda|=0$, and so ${A_0}(\breve{\psi}^0_\mu,\psi^0_\lambda)=0$.
If for given $(\lambda,\gamma)$ such a pair $(\mu,\gamma)$ exists for  $|\mu|>|\lambda|$, then such a pair exists for $|\mu|=|\lambda|+1$ as well, because $\breve{\bm{\Lambda}}_{0,\gamma}$ is a tree, and $\breve{S}^0(\mu') \supset \breve{S}^0(\mu)$ for any ancestor $\mu'$ of $\mu$.
In order words, the condition $|\mu|>|\lambda|$ can be read as $|\mu|=|\lambda|+1$. The set of $(\lambda,\gamma)$ that we just described is given by the set $\bm{\Sigma}$, and so we infer that
\[
R_{\breve{\bm{\Lambda}}} ({\bf L}_0\otimes \mathrm{Id})(\mathrm{Id} \otimes {\bf A}_1) I_{\bm{\Lambda}}=R_{\breve{\bm{\Lambda}}} ({\bf L}_0\otimes \mathrm{Id})I_{\bm{\Sigma}}  R_{\bm{\Sigma}}(\mathrm{Id} \otimes {\bf A}_1) I_{\bm{\Lambda}}.
\]

Now let $(\lambda,\gamma) \in \bm{\Sigma}$. Using that $P_0\bm{\Lambda}$ is a tree, and $S^0(\lambda) \subset S^0(\lambda')$ for any ancestor $\lambda'$ of $\lambda$, we infer that $(\lambda',\gamma) \in \bm{\Sigma}$. Using that for any $\mu \in P_0 \breve{\bm{\Lambda}}$, $\breve{\bm{\Lambda}}_{1,\mu}$ is a tree, we infer that for any ancestor $\gamma'$ of $\gamma$, $(\lambda,\gamma') \in \bm{\Sigma}$, so that $\bm{\Sigma}$ is a double-tree.

For any $\mu \in \breve{\vee}^0$, the number of $\lambda \in \vee^0$ with $|\mu|=|\lambda|+1$ and $|\breve{S}^0(\mu) \cap S^0(\lambda)|>0$ is uniformly bounded,
from which we infer that $\# \bm{\Sigma} \lesssim \sum_{\mu \in P_0 \breve{\bm{\Lambda}}} \#\breve{\bm{\Lambda}}_{1,\mu}=\#\breve{\bm{\Lambda}}$.

Considering \eqref{second}, the range of $({\bf U}_0 \otimes \mathrm{Id}) I_{\bm{\Lambda}}$ consists of vectors that can only have non-zero entries for indices $(\mu,\lambda) \in \breve{\vee}^0 \times P_1\bm{\Lambda}$ for which there exists a $\gamma \in \bm{\Lambda}_{0,\lambda}$ with $|\gamma|\geq |\mu|$ and $|\breve{S}^0(\mu) \cap S^0(\gamma)|>0$.
Since $\bm{\Lambda}_{0,\lambda}$ is a tree, and $S^0(\gamma') \supset S^0(\gamma)$ for any ancestor $\gamma'$ of $\gamma$,
equivalently $|\gamma|\geq |\mu|$ can be read as $|\gamma|= |\mu|$. Furthermore, in view of the subsequent application of $R_{\breve{\bm{\Lambda}}} (\mathrm{Id} \otimes {\bf A}_1)$, it suffices to consider
those indices $(\mu,\lambda)$ with $\mu \in P_0\breve{\bm{\Lambda}}$.
The set of $(\mu,\lambda)$ that we just described is given by the set $\bm{\Theta}$, and so we infer that
\[
R_{\breve{\bm{\Lambda}}} (\mathrm{Id} \otimes {\bf A}_1) ({\bf U}_0 \otimes \mathrm{Id}) I_{\bm{\Lambda}}=R_{\breve{\bm{\Lambda}}} (\mathrm{Id} \otimes {\bf A}_1) I_{\bm{\Theta}} R_{\bm{\Theta}} ({\bf U}_0 \otimes \mathrm{Id}) I_{\bm{\Lambda}}.
\]

Now let $(\mu,\lambda) \in \bm{\Theta}$. If $\lambda'$ is an ancestor of $\lambda$, then
from $P_0 \bm{\Lambda}$ being a tree, and $\bm{\Lambda}_{0,\lambda} \subset \bm{\Lambda}_{0,\lambda'}$, we have $(\mu,\lambda') \in \bm{\Theta}$. If $\mu'$ is an ancestor of $\mu$, then from $P_0 \breve{\bm{\Lambda}}$ being a tree, and $\breve{S}^0(\mu') \supset \breve{S}^0(\mu)$, we infer that $(\mu',\lambda) \in \bm{\Theta}$, and thus that $\bm{\Theta}$ is a double-tree.

For any $\gamma \in \vee^0$, the number of $\mu \in \breve{\vee}^0$ with $|\mu|=|\gamma|$ and $|\breve{S}^0(\mu) \cap S^0(\gamma)|>0$ is uniformly bounded,
from which we infer that $\# \bm{\Theta} \lesssim \sum_{\lambda \in P_1 \bm{\Lambda}} \#\bm{\Lambda}_{0,\lambda}=\#\bm{\Lambda}$.
\end{proof}

\end{document}